\documentclass[11pt, a4paper]{amsart}
\usepackage{amsmath,amssymb,amsthm,mathtools,wasysym,calc,verbatim,tikz,url,hyperref,mathrsfs,cite,fullpage,bbm,thmtools}
\usepackage[noabbrev,capitalize,nameinlink]{cleveref}
\usepackage[shortlabels]{enumitem}

\DeclareMathOperator{\cof}{\cof}

\newcommand{\nmod}[1]{~(\mathrm{mod}~#1)}

\newcommand{\mc}[1]{\mathcal{#1}}

\theoremstyle{plain}

\newcommand{\eps}{\varepsilon}

\usepackage{amsmath, amsthm, amssymb}
\usepackage{hyperref}
\usepackage{mathrsfs}
\usepackage{asymptote}
\usepackage{graphicx}
\usepackage{theoremref}

\theoremstyle{plain}
\newtheorem{theorem}{Theorem}
\newtheorem{lemma}{Lemma}
\newtheorem{proposition}{Proposition}
\newtheorem{corollary}{Corollary}

\newtheorem{example}{Example}

\theoremstyle{definition}
\newtheorem{definition}{Definition}
\theoremstyle{remark}
\newtheorem*{remark}{Remark}

\numberwithin{equation}{section}
\numberwithin{lemma}{section}
\numberwithin{proposition}{section}
\numberwithin{corollary}{section}
\numberwithin{definition}{section}
\numberwithin{claim}{section}

\usepackage[margin=1.0in]{geometry}

\begin{document}

\title[Efficient Equidistribution of Nilsequences]{Efficient Equidistribution of Nilsequences}
\author{James Leng}
\address[James Leng]{Department of Mathematics\\
     UCLA \\ Los Angeles, CA 90095, USA.}
\email{jamesleng@math.ucla.edu}

\maketitle

\begin{abstract}
We give improved bounds for the equidistribution of (multiparameter) nilsequences subject to any degree filtration. The bounds we obtain are single exponential in dimension, improving on double exponential bounds of Green and Tao. To obtain these bounds, we overcome ``induction of dimension'' which is ubiquitous throughout higher order Fourier analysis.

The improved equidistribution theory is a crucial ingredient in the quasi-polynomial $U^4[N]$ inverse theorem of the author and its extension to the quasi-polynomial $U^{s + 1}[N]$ inverse theorem in joint work with Sah and Sawhney. These results lead to further applications in combinatorial number theory such as bounds for linear equations in the primes which save an arbitrary power of logarithm, which match the bounds Vinogradov obtained for the odd Goldbach conjecture.
\end{abstract}

\section{Introduction}
In 2001, Gowers \cite{Gow98, Gow01a} introduced \emph{Gowers norms} in groundbreaking work giving effective bounds for Szemer\'edi's theorem as a way to measure ``pseudorandomness" in Szemer\'edi's theorem. An important technical result proved in work of Gowers was that a function with large Gowers norm correlates with a function that is roughly constant along short arithmetic progressions. Motivated by ergodic theoretic work of Host and Kra \cite{HK05} and Ziegler \cite{Zie07}, Green and Tao, and Green, Tao, and Ziegler in a series of works \cite{GT12, GT10, GT12c, GTZ12, GTZ11} formulated and proved the Gowers inverse conjecture, which stated that nilsequences are precisely the obstructions to Gowers uniformity. Using these tools, they were able to generalize Vinogradov's approach to a ``nilpotent" circle method and count linear configurations in the primes. \\\\
A key result needed in their analysis is a quantitative equidistribution theorem of nilsequences \cite[Theorem 1.16]{GT12}, which we list below (relevant definitions such as $\|\cdot \|_{C^\infty[N]}$ can be found in Section~\ref{s:not}).
\begin{theorem}\label{t:GreenTao}
Let $F(g(n)\Gamma)$ be a nilsequence on a nilmanifold $G/\Gamma$ with a $\delta^{-1}$-rational Mal'cev basis such that $F$ has Lipschitz norm (which is the sum of the Lipschitz constant and the $L^\infty$ norm) at most $1$. If 
$$\left|\mathbb{E}_{n \in [N]} F(g(n)\Gamma) - \int_{G/\Gamma} F d\mu\right| \ge \delta$$
then either $N \ll_{G/\Gamma, \delta} 1$ or there exists a nonzero homomorphism $\eta: G \to \mathbb{R}$ of modulus at most $\delta^{-O_{G/\Gamma}(1)}$ such that $\eta(\Gamma) \subseteq \mathbb{Z}$ such that $\|\eta \circ g\|_{C^\infty[N]} \le \delta^{-O_{G/\Gamma}(1)}.$
\end{theorem}
A corollary of Theorem~\ref{t:GreenTao} is the ``Ratner-type factorization theorem" of Green and Tao \cite[Corollary 1.20]{GT12}, which is often used in applications. See \cite[Definition 1.2]{GT12} for definition of ``$\delta$-equidistributed" and ``totally $\delta$-equidistributed"; these notions are not quantitatively efficient for our purposes.
\begin{theorem}\label{t:RatnerFactorization}
Let $G/\Gamma$ be a nilmanifold with complexity $M \ge 2$ and $g(n)$ be a polynomial sequence on $G$. For each $A \ge 2$, there exists some $\delta$ with $M \le \delta^{-1} \le M^{O_{G, \Gamma, A}(1)}$ and a factorization $g = \varepsilon g_1 \gamma$ where
\begin{itemize}
    \item $\varepsilon$ is $(\delta^{-1}, N)$-smooth, meaning that for all $n \in [N]$, $d(\varepsilon(0), \mathrm{id}_G) \le \delta^{-1}$ and $d(\varepsilon(n - 1), \varepsilon(n)) \le \frac{\delta^{-1}}{N}$;
    \item $\gamma$ is $\delta^{-1}$-periodic;
    \item $g_1$ is $\delta^{A}$-equidistributed inside a subnilmanifold $\tilde{G}/\tilde{\Gamma}$ with complexity at most $\delta^{-1}$ where $\tilde{G}$ is a subgroup of $G$ with rationality at most $\delta^{-1}$. 
\end{itemize}
\end{theorem}
The abelian degree one case is the following.
\begin{proposition}\label{p:AbelianRatnerFactorization}
Let $\alpha \in \mathbb{R}^d/\mathbb{Z}^d$. Then given $M \ge 2$ and $A \ge 2$, there exists some $\delta$ with $M \le \delta^{-1} \le M^{O_{d, A}(1)}$ and a factorization $\alpha = \varepsilon + \alpha' + \gamma$ where
\begin{itemize}
    \item $\|\varepsilon\|_{\mathbb{R}/\mathbb{Z}} \le \frac{\delta^{-1}}{N}$;
    \item there exists some nonzero integer $k \le \delta^{-1}$ such that $k\gamma \in \mathbb{Z}^d$; and
    \item $n \mapsto \alpha'n$ is totally $\delta^{A}$-equidistributed inside a subgroup $\tilde{G}$ of $\mathbb{R}^d/\mathbb{Z}^d$ of rationality at most $\delta^{-1}$.
\end{itemize}
\end{proposition}
Tao and Ter{\"a}v{\"a}inen \cite{TT21}, quantified the above results and proved that quantities $M^{O_{A,G,\Gamma}(1)}$ are in fact $M^{A^{O_k(d^{O_k(1)})}}$ (here, $k$ is the degree of the polynomial sequence); e.g. the exponent is exponential in dimension. The authors of \cite{TT21} raise the question of whether the exponent may be taken to be polynomial in dimension. The starting point of this work is that even in the abelian degree one case, the Ratner-type factorization theorem cannot have bounds which are single exponential in dimension.
\begin{example}\label{e:badexample}
Let $L$ be a large integer and consider the example 
$$g(n) = (\alpha n, L\alpha n, L^2\alpha n, L^4\alpha n, \dots, L^{2^d}\alpha n)$$
(a linear orbit on $\mathbb{T}^{d + 2}$) with $\alpha$ not $O(L^{2^{2^d}}/N)$ close to any rational point with denominator at most $L^{2^{2^d}}$. The factorization theorem with $A = 2$ and $M = L$, yields that this $L^{2^{d + 1}}$-equidistributes in the subtorus $(x, Lx, \dots, L^{2^d}x)$, whose rationality is double exponential in $d$. Note that an application of the Ratner-type factorization theorem \textbf{does not} yield equidistribution in the larger group
$$(x, L x, \dots, L^{2^k}x, y_{k+ 1}, \dots, y_d).$$
This is because this group is $\varepsilon^{-1} = L^{2^k}$-rational, whereas such an orbit is not $\varepsilon^A$-equidistributed since it lies in the subgroup $y_{k + 1} = \varepsilon^{-2} x$. Note that if we worked with 
$$g_1(n) = (\alpha n, L^2\alpha n, L^{2^2}\alpha n, \dots, L^{2^{2^d}}\alpha n),$$
then the factorization theorem would yield equidistribution in $(x, L^2x, \dots, L^{2^{2^{O(1)}}}x, \dots, x_{d + 2})$, which give polynomial losses. In this case, $L^{A2^{2^k}}$ is far less than $L^{2^{2^{k+1}}}$, causing us to not need to pass to a subgroup after $O(1)$ many iterations.
\end{example}
This may seem like a rather silly example, for the reader can question why one cannot simply reformulate the Ratner-type factorization theorem with a choice of $A$ (say $A = 1 + o_{d \to \infty}(1)$) that does not yield losses double exponential in dimension. We will now explain why (for the sake of applications) the theorem specifies equidistribution of scale at most $\delta^A$, where $\delta^{-1}$ upper bounds the rationality of the sub-torus. Given a Lipschitz function $F\colon \mathbb{R}^d/\mathbb{Z}^d \to \mathbb{C}$ of norm $\le 1$, $\delta \in (0, 1/10)$, and $\alpha \in \mathbb{R}^d/\mathbb{Z}^d$, an application of Proposition~\ref{p:AbelianRatnerFactorization} gives us a decomposition of $\alpha = \varepsilon + \alpha' + \gamma$. Let $G'$ be the subgroup $\alpha'$ lives in, $H$ be any coset of $G'$, $Q \le \delta^{-1}$ be the period of $\gamma$, and $P$ be a subprogression of $[N]$ with common difference $Q$ and size at least $\delta^{4} N$. The key point is that $F$ has Lipschitz norm on $G'/\mathbb{Z}^d$ bounded by $d\delta^{-1}$ (say) and thus 
$$\left|\mathbb{E}_{n \in P} F(\alpha n) - \int_{G'} \tilde{F} d\mu\right| \le \left| \mathbb{E}_{n \in P} \tilde{F}(\alpha'n) - \int_{G'} \tilde{F} d\mu\right| + O(\delta^2) \ll \delta^A\|\tilde{F}\|_{\mathrm{Lip}(\tilde{G})} + \delta^{2} \ll \delta^2$$
Here $\tilde{F} = F(\varepsilon_{P} + Q\cdot)$ for an element $\varepsilon_P\in \mathbb{R}^d$; this essentially encodes the coset that $\alpha n$ nearly lives in for $n\in P$. We see that if $A$ is too small (say $A = 1 + o_{d \to \infty}(1)$), then $\delta^A \|\tilde{F}\|_{\mathrm{Lip}(\tilde{G})}$ would be large and the above inequality would be essentially useless for applications of the nilsequence form of the Ratner-type factorization theorem. This explains why we must have the threshold for equidistribution, which is the quantity $\delta^A$, to be significantly smaller than the rationality of the subgroup, which is $\delta^{-1}$. \\\\
Such a result has proven to be crucial throughout higher order Fourier analysis as it provides a ``structure vs. randomness'' theory for nilsequences. In particular, ``structure'' comes in the form of subgroup $G'$ on which the polynomial sequence lives and ``randomness' in that the polynomial sequence equidistributes on $G'/(G'\cap \Gamma)$. This theory has led to the Green-Tao-Ziegler deduction of the $U^{s + 1}$ inverse theorem \cite{GTZ12, GTZ11}, was an integral part of their result on linear equations in primes \cite{GT10, GT12c}, and has continued to be used in various more recent contexts in higher order Fourier analysis \cite{TT21, MW22, PW23}. \\\\
For establishing efficient bounds in higher order Fourier analysis, a double exponential loss is inefficient. To explain why, we require the statement of the quasi-polynomial $U^{s + 1}[N]$ inverse theorem proven in joint work with Sah and Sawhney \cite{LSS24b}. 
\begin{theorem}\label{thm:inverse}
Fix $\delta\in (0,1/2)$. Suppose that $f\colon[N]\to\mathbb{C}$ is $1$-bounded and
\[\|f\|_{U^{s+1}[N]}\ge\delta.\]
Then there exists a nilmanifold $G/\Gamma$ of degree $s$, complexity at most $M$, and dimension at most $d$ as well as a function $F$ on $G/\Gamma$ which is at most $K$-Lipschitz such that 
\[|\mathbb{E}_{n\in[N]}[f(n)\overline{F(g(n)\Gamma)}]|\ge\epsilon,\]
where we may take
\[d\le\log(1/\delta)^{O_s(1)}\emph{ and }\eps^{-1},K,M\le\exp(\log(1/\delta)^{O_s(1)}).\]    
\end{theorem}
We note that even assuming the quasi-polynomial inverse theorem, any further application of an equidistribution theorem with losses double exponential in dimension would result in a loss of the form $\exp(\exp(\log(1/\delta)^{O(1)}))$, which for applications is inefficient. Under $t$ many iterations, the losses compound to an $O(t)$ iterated exponential loss. In particular, a double exponential loss in equidistribution immediately results in $O(s^2)$ many iterated logarithms for the strategy of \cite{LSS24b}.
\subsection{Induction on dimension and Previous Work}
We recall the proof of Proposition~\ref{p:AbelianRatnerFactorization} and see where the double exponential bounds in dimension arise. The proof proceeds algorithmically as follows.
\begin{itemize}
    \item[(i)] If $\alpha n$ is totally $\delta$-equidistributed, we are done.
    \item[(ii)] Otherwise, there exists some nonzero frequency $k \in \mathbb{Z}^d$ with height $|k| \le \delta^{-1 - o_{d \to \infty}(1)}$ such that $\|k \cdot \alpha\|_{\mathbb{R}/\mathbb{Z}} \le \frac{\delta^{-O(d)}}{N}$. 
    \item[(iii)] We may thus write $\alpha = \varepsilon + \alpha' + \gamma$ where $\gamma$ is $\delta^{-O(d)}$-rational, $\varepsilon$ is $(d\delta^{-1 - o(1)}, N)$-smooth, and $\alpha'$ lies in a lower dimensional subgroup.
    \item[(iv)] Iterate this procedure along progressions $Q$ for which $\varepsilon n$ is roughly constant, and $\gamma n$ lies in a fixed coset of $\mathbb{Z}^d$ whenever $n \in Q$.
\end{itemize}
Notice that in item (iv), we have that the function $F'$ we consider is $\delta^{-1-o(1)}$-Lipschitz instead of the original $1$-Lipschitz. Upon iterating with $F'$, we need to check frequencies of height $\delta^{-2-o(1)}$; at the $t$-th step of iteration we have that the height of frequencies we are considering are $\delta^{-2^{t}}$ (say) and this is ultimately the cost of the double exponential loss. This however comes from the fact that we terminate if an only if $\alpha n$ equidistributes and thus force ourselves to consider all possible characters at each stage of the iteration. The issue present is that we are using an \emph{induction on dimensions} argument; in such a context, even an iteration $\delta \mapsto \delta^2$ is unacceptable since this compounds to losses of $\delta^{2^d}$ under induction. \\\\
One possible way to overcome this issue is to observe the situation in higher order Fourier analysis over the group $\mathbb{F}_p^n$. There, although there is no such Ratner-type factorization theorem, the work of Gowers and Wolf \cite{GW11b} still obtains a somewhat satisfactory structure vs. randomness theory from proving a decomposition theorem of the \emph{exponential of a polynomial} into well-equidistributed parts. \\\\
This motivates that over $\mathbb{Z}/N\mathbb{Z}$, one should still be able to obtain a satisfactory theory from merely considering the \emph{nilsequence} $F(g(n)\Gamma)$ rather than the \emph{polynomial sequence} $g(n)$. At least in the abelian case such a change in perspective is at least plausibly useful; one can replace $F$ by its Fourier approximation and understanding the exponential sum of a character is simple. Furthermore given a $1$-Lipschitz function on the torus on $d$-dimensions one only requires at most $\delta^{-O(d^{O(1)})}$ such characters which is an acceptable loss in our context. \\\\
Prior to this work, there was at least one model case where an equidistribution with losses which were single exponential in dimension is achieved. This was the work of Gowers and Wolf \cite{GW11}, Green and Tao \cite{GT17}, and the author \cite{Len22b}. As Bohr sets are morally equivalent to generalized arithmetic progressions, one can write the quadratic phase as a multivariate quadratic polynomial (with variables representing directions in the GAP) and then work with the equidistribution of multidimensional polynomial phase (see \cite{TaoBlog1}). This, combined with the observation in that the indicator function of a Bohr set or a generalized arithmetic progression has bounded Fourier complexity (see \cite[Section 3]{Len22b} for relevant definitions) shows that a non-equidistributed quadratic phase has (approximate) Fourier complexity which is single exponential in dimension. \\\\
The conclusion from above is qualitatively weaker than what Theorem~\ref{t:GreenTao} gives. Theorem~\ref{t:GreenTao} tells us not only that we have bounded Fourier complexity, but also that we should expect the polynomial sequence $g(n)$ to be equidistributed in a rational subgroup. Notably, the methods of \cite{GW11, GT17, Len22b} make no reference to a polynomial sequence. The methods of \cite{GW11, GT17, Len22b} merely state that a non-equidistributed two-step nilsequence has bounded Fourier complexity. After all, there are many, many functions of bounded approximate Fourier complexity, but only a few rational subspaces of a vector space. We highlight that applications of the equidistribution theory such as the complexity one polynomial Szemer\'edi theorem \cite{Len22, PSS23}, as well as results involving counting solutions to linear equations \cite{Alt22, Alt22b, CS14, GT10b, Kuc21, Kuc23} all rely on the qualitatively stronger Theorem~\ref{t:GreenTao} and are thus more robust and general than results obtained using the weaker equidistribution theory as in \cite{GW11, GT17, Len22b}. Therefore, it is desirable to obtain a result of the quality of Theorem~\ref{t:GreenTao} with bounds of the shape that \cite{GW11, GT17, Len22b} obtain.
\subsection{Statement of the main theorem}
We first require the definition of the lower central series.
\begin{definition}\label{d:lowercentraseries}
The \emph{lower central series} of a nilpotent Lie group $G$, is the sequence of nested subgroups $G \supseteq G_{(2)} \supseteq G_{(3)} \cdots $ where $G_{(i)} := [G, G_{(i - 1)}]$. If $s$ is the least integer such that $G_{(s + 1)} = \mathrm{Id}_G$, then $G$ is an $s$-step nilpotent Lie group.
\end{definition}
In order to state our main result we will also require the notion of a vertical character.
\begin{definition}
Consider a nilmanifold $G/\Gamma$ and a function $F\colon G/\Gamma\to\mathbb{C}$. Given a connected, simply connected subgroup $T$ of the center $Z(G)$ which is rational (i.e., $\Gamma\cap T$ is cocompact in $T$) and a continuous homomorphism $\eta\colon T\to\mathbb{R}$ such that $\eta(T\cap\Gamma)\subseteq \mathbb{Z}$, if 
\[F(gx)=e(\eta(g))F(x)\emph{ for all }g\in T\]
we say that $F$ is a $T$-vertical character with frequency $\eta$. If $T = G_{(s)}$, we refer to \emph{$T$-vertical character} as \emph{vertical character}.
\end{definition}
Note that decomposing a function $F$ into functions with a vertical frequency follows from standard Fourier analysis (e.g., Lemma~\ref{l:nilcharacters}). We now state our main result; various quantitative undefined notions can be found in Section~\ref{s:not}.
\begin{restatable}{theorem}{mainresulta}\label{t:mainresult1}
Let $\delta\in (0,1/10)$, $N_1, \ldots, N_\ell > 1$ be integers, $M \ge 1$, and $G/\Gamma$ be a nilmanifold given a degree $k$ filtration and such that $G$ is $s$-step nilpotent. Furthermore, let $F\colon G/\Gamma \to \mathbb{C}$ be an $M$-Lipschitz function with vertical character of nonzero $G_{(s)}$-frequency $\xi$ on an $s$-step nilmanifold of complexity $M$ and dimension $d$ with $|\xi| \le M/\delta$. 

Suppose $g(\vec{n})$ is a polynomial sequence in $G$ with respect to the degree $k$ filtration, and such that
\[|\mathbb{E}_{\vec{n} \in [\vec{N}]} F(g(\vec{n})\Gamma)| \ge \delta.\]

Then one of the following occurs.
\begin{itemize}
    \item $\min_{1\le i\le \ell}N_i\le (M/\delta)^{O_{k,\ell}(d^{O_{k,\ell}(1)})}$
    \item There exist $0< r\le \mathrm{dim}(G/[G,G])$ and a set of horizontal characters $\eta_1, \dots, \eta_r$ of size at most $(M/\delta)^{O_{k,\ell}(d^{O_{k,\ell}(1)})}$ such that 
    \[\|\eta_i \circ g\|_{C^\infty[\vec{N}]} \le (M/\delta)^{O_{k,\ell}(d^{O_{k,\ell}(1)})}\]
    and such that for $w_1,\ldots,w_s\in G' := \cap_{1\le i\le r}\mathrm{ker}(\eta_i)$ 
    \[\xi([[[w_1, w_2], w_3], \ldots,w_s]) = 0.\]
\end{itemize}
\end{restatable}
The crucial feature to note from this theorem is that $G'/\mathrm{ker}(\xi)$ is nilpotent of step at most $s - 1$; this is a ``one-shot" reduction in step and hence avoids the issue of induction on dimensions. In order to do so, the proof requires a shift in perspective of locating a \emph{subspace} on the horizontal torus rather than locating a \emph{single horizontal character}. This is possible since we're proving a theorem for a $G_{(s)}$-vertical character rather than an arbitrary nilsequence which previous works \cite{GT12, Lei05} prove. This shift in perspective is related to a crucial insight of \cite{LSS24b} that improves on \cite{GTZ12}. \\\\
We further remark that the assumption of the $G_{(s)}$-vertical character can be replaced with a vertical character with respect to any rational subgroup of the center of rationality bounded by $(M/\delta)^{O_{k, \ell}(d^{O_{k, \ell}(1)})}$ (see \cite[Theorem~5.5]{LSS24b}) using Theorem~\ref{t:mainresult1} as a black box. Such a statement can be morally proven by iterating Theorem~\ref{t:mainresult1}; however, due to the step decreasing at each stage such iterative arguments take $O_s(1)$ steps instead of $O(d)$ many iterations. \\\\
A consequence of Theorem~\ref{t:mainresult1} is that in view of Theorem~\ref{thm:inverse}, \emph{any} bounded number of iterated applications of Theorem~\ref{t:mainresult1} would result in only quasi-polynomial losses. This is because the composition of two quasi-polynomial functions is still quasi-polynomial and is ultimately why an application of Theorem~\ref{t:mainresult1} instead of \cite{GT12} results in a quasi-polynomial bound in \cite{LSS24b} which in turn has been used to improve bounds on Szemer{\'e}di's theorem \cite{LSS24c}. \\\\
We now briefly discuss applications of this statement. Slight modifications of Theorem~\ref{t:mainresult1} can be inserted in existing arguments of the complexity one polynomial Szemer\'edi theorem over $\mathbb{Z}/N\mathbb{Z}$ due to the author \cite{Len22} to improve the bounds from an iterated logarithm bound to a quasi-polynomial bound; see \cite{LenNew}. In principle, it can also be used in \cite{PSS23} to prove effective bounds on sets lacking shifted squares in place of \cite{GT12} to obtain a quasi-logarithmic bound over the iterated logarithmic bound obtained. \\\\
The improved $U^{s + 1}[N]$ inverse theorem, combined with Theorem~\ref{t:mainresult1}, and the approach to effective linear equations in the primes due to Tao and Ter{\"a}v{\"a}inen \cite{TT21} gives the following result. As is standard, $\Lambda$ denotes the von Mangoldt function. 

\begin{theorem}\label{t:quantthm}
Let $N,d,t,L,A$ be positive integers, and let $\Psi = (\psi_1,\dots,\psi_t)$ be a system of affine-linear forms $\psi_i \colon \mathbb{Z}^d \to \mathbb{Z}$ of the form 
\[\psi_i(\vec{n}) = v_i\cdot \vec{n} + \psi_i(\vec{0})\]
with $v_i\in \mathbb{Z}^d$ such that $|\psi_i(\vec{0})|\le LN$ and $\|v_i\|_{\infty}\le L$. Suppose that $v_i$ are pairwise linearly independent. Let $\Omega\subset [-N,N]^d$ be a convex body. Then 
\begin{align}\label{e:count}
\sum_{\vec{n}\in \Omega\cap \mathbb{Z}^d} \prod_{i=1}^{t}\Lambda(\psi_i(\vec{n})) = \beta_\infty \prod_p \beta_p + O_{t,d,L,A}(N^d (\log N)^{-A}).
\end{align}
Here $\beta_{\infty} := \mathrm{vol}(\Omega\cap \Psi^{-1}(\mathbb{R}_{>0}^{t}))$ and for each prime $p$ we have 
\[\beta_p := \mathbb{E}_{\vec{n}\in (\mathbb{Z}/p\mathbb{Z})^d} \prod_{i=1}^t \frac{p}{p-1} \mathbf{1}_{\psi_i(\vec n) \neq 0}.\]
\end{theorem}
\begin{remark}
The constant $O_{t,d,L,A}(\cdot)$ is ineffective (as is standard) due to the use of Siegel's theorem.  
\end{remark}

The bounds here which save an arbitrary power of logarithm match that of the best known even for systems of complexity one such as three--term arithmetic progressions or the odd Goldbach conjecture due to Vinogradov \cite{Vin37}. \\\\
The key technical input in proving Theorem~\ref{t:quantthm} is proving that $\Lambda$ is well approximated in the Gowers norm by a certain model $\Lambda_{\mathrm{Siegel}}$ which accounts for Siegel zeros; here we obtain quasipolynomial bounds. We defer a formal statement to Section~\ref{s:linearprimes}. As in \cite{Len22b}, one feature of the proof is that we completely avoid the transference principle in proving Theorem~\ref{t:quantthm}. This is due to the fact that the inverse theorem is sufficiently strong to able to appropriately handle sets of logarithmic density such as the primes.

\subsection{Discussion of the proof}
Our proof follows the strategy of \cite{GT12} of using the van der Corput inequality and performing an analysis on the resulting nilsequence. As nilmanifold notation can become quite complex, we shall describe our proof in bracket polynomial formalism. In order to estimate exponential sums of the bracket polynomial, we require the following tools.
\begin{itemize}
    \item The van der Corput inequality; this is used to reduce to a bracket polynomial of lower degree.
    \item A Fourier approximation result (e.g., Lemma~\ref{l:nilcharacters}); this is used to isolate the term with the highest number of nested brackets of the bracket polynomial as an exponential sum.
    \item A Vinogradov-type lemma; as exponential sums of pure polynomial phases require this ingredient, we will need this ingredient too.
    \item The new ingredient in this work, the \emph{refined bracket polynomial lemma}, which can be thought of as a Vinogradov-type lemma for linear bracket polynomials. Roughly speaking, this lemma ``unwinds" or ``unnests" a bracket and converts the information of a linear bracket polynomial to a diophantine condition on its phases.
\end{itemize}
To very roughly describe the proof, we consider the term in the bracket polynomial with the most brackets, say $s$ many brackets. In the language of nilsequences, the highest number of brackets corresponds to (one less than) the ``step" of the nilpotent Lie group. Our goal is to prove diophantine information in the phases of that term. To do so, we argue via induction on the degree of the bracket polynomial via a van der Corput-type argument, obtaining information on $P_h = P(\cdot + h) - P(\cdot)$ for many $h$ if $P$ were our bracket polynomial. Our argument naturally falls into two cases.
\begin{itemize}
    \item The first case is when each polynomial ``in between the brackets'' is degree one (e.g., a term of the form $\alpha n\{\beta n\{\gamma n\}\}$). In this case, a van der Corput type argument would make the term with $s$ brackets into a term with $s - 1$ brackets, where one of the phases is a bracket polynomial of degree one in $h$. We can then use the refined bracket polynomial lemma to that degree one bracket polynomial in $h$ to obtain information on the phases of that bracket polynomial.
    \item The second case is when this does not occur (e.g., $\alpha n^2\{\beta n\}$; here note that $\alpha n^2$ is not degree one). A van der Corput type argument would not make the term with $s$ brackets into a term with $s - 1$ brackets; it will stay with $s$ brackets. Consequently, each phase in between the brackets is a genuine polynomial in $h$ (and not a bracket polynomial in $h$). We can thus use Vinogradov's lemma and induction to obtain information on the phases of this bracket polynomial.
\end{itemize}
Instead of bracket polynomial formalism, we have chosen to work with nilsequence formalism, the advantage being that the proof is much cleaner and illuminating in the language of nilsequences. However, we still feel it is useful for the reader to keep these bracket polynomial heuristics in mind while examining the proof of Theorem~\ref{t:mainresult1}. For a more motivated discussion of the proof, we direct the reader to \cite{LenNew}. \\\\

\subsection{Organization of the paper}
In Section~\ref{s:not}, we define the notation used in the paper. In Section~\ref{s:refbrack}, we prove the refined bracket polynomial lemma. In Section~\ref{s:equidistribute}, we prove Theorem~\ref{t:mainresult1}. In Section~\ref{s:linearprimes}, we prove Theorem~\ref{t:quantthm}. \\\\
In Appendix~\ref{s:aux}, we give further auxiliary lemmas used in the proof in Section~\ref{s:equidistribute}. In Appendix~\ref{s:quantnil}, we prove various quantitative results on Mal'cev bases.

\subsection{Acknowledgements}
We would like to thank Terry Tao for advisement and for suggesting a simpler proof (private communication) of the refined bracket polynomial lemma than the proof the author initially came up with, which can be found in \cite[Appendix C]{LenNew}. We would also like to thank Ashwin Sah and Mehtaab Sawhney for being extremely supportive and helpful to the author; their feedback has helped improve this manuscript immensely. We in addition thank Dan Altman, Jaume de Dios, Ben Green, Zach Hunter, Ben Johnsrude, Borys Kuca, Freddie Manners, Rushil Raghavan, and David Soukup for helpful discussions and suggestions related to this project. In particular, when the author was on the verge of giving up on this project, the preparation of giving an explanation of this project to Borys led the author to the key inspiration needed to prove Theorem~\ref{t:mainresult1}. We are grateful to Joni Ter\"av\"ainen as well for helpful comments and raising the question of formulating an equidistribution theory with better bounds (jointly with Terry Tao) as a remark in \cite{TT21}. The author is supported by an NSF Graduate Research Fellowship Grant No. DGE-2034835.

\section{Notation and conventions}\label{s:not}
In this section, we will mostly go over notation regarding nilpotent Lie groups. Most of our notation follows \cite{GT12}. First, we set some general notation. Given a function $f\colon A \to \mathbb{C}$ defined on a finite set $A$, we set
$$\mathbb{E}_{a \in A} f(a) := \frac{1}{|A|} \sum_{a \in A} f(a).$$
We also denote for $x \in \mathbb{C}$, $e(x) = \exp(2\pi i x)$. Given a real number $r \in \mathbb{R}$, we denote $\{r\} \in (-1/2, 1/2]$ as the difference between $r$ and a nearest integer to $r$. Given a general vector $a = (a_1, \dots, a_d) \in \mathbb{R}^d$, we denote $\{a\} = (\{a_1\}, \dots, \{a_d\})$. The variable $\delta$ will denote a positive small parameter less than $\frac{1}{100}$. If $v = (v_1, \dots, v_d) \in \mathbb{R}^d$ is a vector, we denote $|v| := \|v\|_\infty := \sup_{1 \le i \le d} |v_i|$. We shall also write $\mathbb{T}^d := \mathbb{R}^d/\mathbb{Z}^d$. \\\\
Nilmanifolds arise naturally as a means to ``clean up" bracket polynomials that show up in the inverse theory of Gowers norms. For further reference along these objects and how they relate to higher order Fourier analysis, we refer the reader to \cite{Tao12} or \cite{HK18}.\\\\
We shall work with the following (slightly nonstandard\footnote{\cite{GT12} works with the vertical component as $G_k$ component where $k$ is the smallest nontrivial element of an arbitrary filtration.}) conventions:
\begin{definition}[Horizontal and vertical component]
We denote $G_{(s)}$ as the \emph{vertical component} of $G$ and $G/G_{(2)} = G/[G, G]$ as the \emph{horizontal} component of $G$. We denote the quotient map as $\pi_{\mathrm{horiz}}\colon G \to G/G_{(2)}$. We denote the dimension of $G/G_{(2)}$ as $d_{\mathrm{horiz}}$ and the dimension of $G_{(s)}$ as $d_{(s)}$.
\end{definition}
For $\vec{N}$ a vector in $\mathbb{N}^\ell$ and enumerated as $\vec{N} = (N_1, \dots, N_\ell)$, we define $[\vec{N}] := [N_1] \times [N_2] \times \cdots \times [N_\ell]$ and $|\vec{N}| = N_1N_2\cdots N_\ell$. We will now define smoothness norms of polynomials on $\vec{N}$. 
\begin{definition}[Smoothness norms]
Given $p: \mathbb{Z}^\ell \to \mathbb{C}$ a multivariate polynomial, we expand
$$p(\vec{n}) = \sum_{\vec{j}} \alpha_{\vec{j}}\binom{\vec{n}}{\vec{j}}$$
where denoting $\vec{n} = (n_1, \dots, n_\ell)$ and $\vec{j} = (j_1, \dots, j_\ell)$, we define
$$\binom{\vec{n}}{\vec{j}} := \binom{n_1}{j_1}\binom{n_2}{j_2} \cdots \binom{n_\ell}{j_\ell}.$$
We define
$$\|p\|_{C^\infty[\vec{N}]} \le \sup_{\vec{i} \neq \vec{0}} \vec{N}^{\vec{i}} \|\alpha_i\|_{\mathbb{R}/\mathbb{Z}}$$
where we denote $\vec{N}^{\vec{i}} = N_1^{i_1}N_2^{i_2} \cdots N_\ell^{i_\ell}$.
\end{definition}
We require the following definition regarding the height of a rational number.
\begin{definition}[Height of a rational]
We say an element $r \in \mathbb{Q}$ has \emph{height} at $Q$ if written in reduced form $r = \frac{a}{b}$, we have $|a|, |b| \le Q$.
\end{definition}
We will now define a nilmanifold.
\begin{definition}[Nilmanifold]
A nilmanifold of degree $k$, step $s$, complexity at most $M$, dimension $d$ consists of the following data:
\begin{enumerate}
    \item An $s$-step connected and simply connect nilpotent Lie group $G$ and a discrete cocompact subgroup $\Gamma$;
    \item A filtration $G_\bullet = (G_i)_{i = 0}^\infty$ consisting of subgroups $G_i$ of $G_0$ with $G_{i + 1} \subseteq G_i$ and $G_0 = G_1$ such that $G_\ell = \mathrm{Id}_G$ for $\ell > k$, and $[G_i, G_j] \subseteq G_{i + j}$; and
    \item A basis $\{X_1, \dots, X_d\}$ of $\mathfrak{g} := \log(G)$ known as a \emph{Mal'cev basis}.
\end{enumerate}
We furthermore require the data to satisfy the following.
\begin{enumerate}
    \item For $1 \le i, j \le d$, we have
    $$[X_i, X_j] = \sum_{k > \max(i, j)} a_{ijk} X_k$$
    where $a_{ijk}$ has height at most $M$;
    \item Letting $d_i := \text{dim}(G_i)$, we have $\log(G_i) = \text{span}(X_j: d - d_i < j \le d)$;
    \item Each element $g$ can be expressed uniquely as $\exp(t_1X_1)\exp(t_2X_2) \cdots \exp(t_dX_d)$ with $t_1, \dots, t_d \in \mathbb{R}$ and we denote the map $\psi\colon G \to \mathbb{R}^d$ via $\psi(g) := (t_1, \dots, t_d)$;
    \item $\Gamma$ is precisely all the elements $g$ for which $\psi(g) \in \mathbb{Z}^d$.
\end{enumerate}
We will refer to any Mal'cev basis that satisfies (2), (3), and (4) as a Mal'cev basis \emph{adapted to} the filtration $G_\bullet$.
\end{definition}
Note that each nilpotent Lie group has the \emph{standard filtration} which is defined above as $G_i = [G, G_{i - 1}]$, but not all filtrations have to be the standard filtration. It is true that the standard filtration is minimal in the sense that given a filtration $(G_i)$, $G_i \supseteq G_{(i)}$ \cite[Exercise 1.6.2]{Tao12}. We will now need the definition of a multiparameter polynomial sequence.
\begin{definition}[Multivariable polynomial sequences]
Given a nilmanifold $G/\Gamma$ with filtration $(G_i)_{i = 0}^k$, the multiparameter polynomial sequences $\mathrm{poly}(\mathbb{Z}^\ell, G)$ are the set of maps $g\colon \mathbb{Z}^\ell \to G$ such that
$$g(\vec{n}) = \prod_{\vec{i}}g_{\vec{i}}^{\binom{\vec{n}}{\vec{i}}}$$
where $g_{\vec{i}} \in G_{|\vec{i}|}$ where we take the product with respect to some ordering such that $\vec{i}$ appears before $\vec{j}$ if $|\vec{i}| < |\vec{j}|$ and if $|\vec{i}| = |\vec{j}|$ we order arbitrarily. Note from \cite[Chapter 14]{HK18} that $\mathrm{poly}(\mathbb{Z}^\ell, G)$ forms a group under pointwise multiplication. Furthermore, any ordering gives results in the same group. 
\end{definition}
We will now need definitions of rational and smooth polynomial sequences.
\begin{definition}[Rational and smooth sequences] Given a positive integer $D$, we say an element $g$ of $G$ is \emph{rational with height $D$}, or \emph{$D$-rational}, if $g^D \in \Gamma$. We say a polynomial sequence is $D$-rational if $\gamma(\vec{n})$ is $D$-rational for each $\vec{n} \in\mathbb{Z}^\ell$.

Given a polynomial sequence $\eps(\vec{n})$, we say that $\eps$ is $(K, \vec{N})$-\emph{smooth} if for each $i$,
$$d_{G}(\eps(n + e_i), \eps(n)) \le \frac{K}{N_i} \text{ and } d_{G}(\eps(0), \mathrm{id}_G) \le K.$$
\end{definition}
We will need the definition of a metric on a nilmanifold.
\begin{definition}[Metric of a nilmanifold]\label{d:metricnilmanifold}
Given a nilmanifold $G/\Gamma$ with coordinate map $\psi$, we may define a distance map $d$ to be right invariant and satisfying
$$d_{\psi}(x, y) := \inf\left\{\sum_{i = 1}^n \min(|\psi(x_ix_{i - 1}^{-1})|, |\psi(x_{i - 1}x_i^{-1})|): x_0 = x, x_n = y\right\}.$$
This defines a metric on $G/\Gamma$ via $d(x\Gamma, y\Gamma) := \inf\{d(x', y'): x'\Gamma = x\Gamma, y'\Gamma = y\Gamma\}$.    
\end{definition}
We are finally in a position to define a nilsequence.
\begin{definition}[Nilsequence]
A \emph{nilsequence} is a sequence of the form $n \mapsto F(g(n)\Gamma)$ where $F\colon G/\Gamma \to \mathbb{C}$ is Lipschitz. We shall denote the \emph{Lipschitz norm} or \emph{Lipschitz parameter} of $H\colon X \to \mathbb{C}$ on a metric space $(X, d_X)$ as
$$\|H\|_{\mathrm{Lip}(X)} := \|H\|_{L^\infty(X)} + \sup_{x \neq y \in X} \frac{|H(x) - H(y)|}{d_X(x, y)}.$$
If $\|H\|_{\mathrm{Lip}(X)} \le M$, we say that $H$ is an \emph{$M$-Lipschitz function on $X$}, or simply an \emph{$M$-Lipschitz function}.
\end{definition}
The point of considering Lipschitz functions is that in the case that $G$ is abelian, they can be quantitatively Fourier approximated. See Lemma~\ref{l:nilcharacters} for the corresponding generalization for nilsequences.

%\begin{lemma}\label{l:automorphic}
 %If $F\colon G \to \mathbb{C}$ is $\Gamma$-automorphic, then $\|F\|_{\mathrm{Lip}(G)} = \|F\|_{\mathrm{Lip}(G/\Gamma)}$   
%\end{lemma}
%\begin{proof}
%We have
%\begin{align*}
%|F(x) - F(y)| &= \inf_{\gamma, \gamma' \in \Gamma} |F(x\gamma) - F(y\gamma')|  \\
%&\le \|F\|_{\mathrm{Lip}(G)} \inf_{\gamma, \gamma' \in \Gamma} d(xg, yg') \\
%&= \|F\|_{\mathrm{Lip}(G)} d_{G/\Gamma}(x\Gamma, y\Gamma).
%\end{align*}
%On the other hand, we also have
%$$|F(x) - F(y)| \le \|F\|_{\mathrm{Lip}(G/\Gamma)} d_{G/\Gamma}(x\Gamma, y\Gamma) \le \|F\|_{\mathrm{Lip}(G/\Gamma)}d_{G}(x, y).$$
%This chain of inequalities give $\|F\|_{\mathrm{Lip}(G)} \le \|F\|_{\mathrm{Lip}(G/\Gamma)}$ whereas taking a quotient and taking a supremum over $x$ and $y$ and applying the previous inequalities give
%$$\sup_{x, y} \frac{|F(x) - F(y)|}{d_{G/\Gamma}(x\Gamma, y\Gamma)} = \|F\|_{\mathrm{Lip}(G/\Gamma)} \le \|F\|_{\mathrm{Lip}(G)}.$$
%\end{proof}

We require the following definitions of horizontal characters.
\begin{definition}[Horizontal characters]
\emph{Horizontal characters} on $G/\Gamma$ are continuous homomorphisms $\eta\colon G \to \mathbb{R}$ such that $\eta(\Gamma) \subseteq \mathbb{Z}$; note that since $\mathbb{R}$ is abelian, $\eta$ annihilates $[G, G]$. They are referred to as such since $G/[G, G]\Gamma$ is known as the \emph{horizontal torus}. Since we were given a Mal'cev basis, each horizontal character $\eta$ corresponds to a vector $k \in \mathbb{Z}^d$ and thus embeds in $\Gamma$.    
\end{definition}

\begin{definition}[Orthogonality]
Given a horizontal character $\eta$ and an element $w \in G$, we define an inner product $\langle \eta, w \rangle := \eta(w)$ and we say that $w$ and $\eta$ are \emph{orthogonal} (or correspondingly $w$ is orthogonal to $\eta$ or $\eta$ is orthogonal to $w$) if $\langle \eta, w \rangle = 0$.
\end{definition}
If $H$ is a rational subgroup that contains $[G, G]$, $G/H$ can be identified with the orthogonal complement $\mathfrak{h} = \log(H)$ with respect to the Mal'cev basis; it thus makes sense to talk about linear independence of elements in $G/H$. \\\\
Note that given elements $w_1, \dots, w_{s - 1}$ in $\Gamma/(\Gamma \cap [G, G])$, that $g \mapsto \xi([[[[g, w_1], w_2], \dots], w_{s - 1}])$ defines a horizontal character. Since the notation $[[[[g, w_1], w_2], \dots], w_{s - 1}]$ will appear very often, we introduce the following notation as shorthand.
\begin{definition}[$t - 1$-fold commutator]
    Given $g_1, \dots, g_t$ in a group $G$, we define
    $$[g_1, g_2, \dots, g_t] := [[[g_1, g_2], g_3], \dots, g_t].$$
    We will denote any such commutator as a \emph{$t - 1$-fold commutator} of elements $g_1, \dots, g_t$.
\end{definition}
We will now need the notion of a size of a horizontal and vertical character.
\begin{definition}[Size of a character]
Consider a horizontal character $\eta\colon G \to \mathbb{R}$. To define the \emph{size} or \emph{modulus} of a horizontal character, we recall that writing in Mal'cev coordinates $\eta(g) = k \cdot \psi(g)$ for a vector $k \in \mathbb{Z}^d$. The size or modulus of $\eta$, denoted $|k|$ or $\|k\|_\infty$, is the $L^\infty$ norm of the vector $k$. 

Given a rational subgroup $H$ of the center of $G$ and a homomorphism $\xi\colon H/(\Gamma \cap H) \to \mathbb{R}/\mathbb{Z}$, we denote the \emph{size} or \emph{modulus} of $\xi$, denoted as $|\xi|$ to be the Lipschitz constant of $x \mapsto e(\xi(x))$ on $H/(\Gamma \cap H)$. That is,
$$|\xi| := \sup_{x \neq y \in H/(\Gamma \cap H)} \frac{|e(\xi(x)) - e(\xi(y))|}{d_{G/\Gamma}(x, y)}.$$    
\end{definition}
We will also need the size of an element of $G$.
\begin{definition}[Size of an element]
The \emph{size} of an element $g \in G$, is the $\ell^\infty$ norm of $\psi(g)$; if $H$ is a rational subgroup of $G$ that contains $[G, G]$, the size of an element $\overline{g} \in G/H$ is defined to be the $\inf_{g' \in G\colon \overline{g} = g'\nmod{H}} \|\psi(g')\|_\infty$.
\end{definition}
Finally, we will need notions of rationality of subspaces and subgroups.
\begin{definition}[Rationality of a subspace]
Let $V$ be a vector space with a basis $\mc{B}$ and let $W$ be a subspace. We say that $W$ is at most $Q$-rational if $W$ has a basis of elements $\mc{B}'$ with $w_j\in \mc{B}'$ satisfying
\[w_j = \sum_{v_k\in \mc{B}} a_{jk} v_k\]
such that $a_{jk}$ are rationals with height at most $Q$.
\end{definition}
\begin{definition}[Rationality of a subgroup]
A connected and simply connected subgroup $G'$ of a nilmanifold $G/\Gamma$ given Mal'cev basis $\mc{X}$ is $Q$-rational if $\mathfrak{g}' := \log(G')$ is $Q$-rational with respect to $\mathfrak{g}$ given basis $\mc{X}$.
\end{definition}
\subsection{Reduction techniques}
One technique we use is one used to reduce to the case that $g(0) = 1$ and $|\psi(g(e_i))| \le 1/2$. We encapsulte this in the following lemma.
\begin{lemma}\label{l:trick1}
Consider $G/\Gamma$, a nilmanifold of complexity $M$, dimension $d$, and equipped with a degree $k$ filtration. Given a nilsequence $F(g(\vec{n})\Gamma)$, there exists a nilsequence $(\tilde{F}(\tilde{g}(\vec{n})\Gamma))_{\vec{n} \in \mathbb{Z}^\ell}$ such that
\begin{itemize}
    \item $F(g(\vec{n})\Gamma) = \tilde{F}(\tilde{g}(\vec{n})\Gamma)$ for all $\vec{n} \in \mathbb{Z}^\ell$
    \item $\|\tilde{F}\|_{\mathrm{Lip}(G)} \le M^{O_k(d^{O(1)})}\|F\|_{\mathrm{Lip}(G)}$
    \item $\tilde{g}(0) = 1$ and $|\psi(\tilde{g}(e_i))| \le \frac{1}{2}$ for each $i \in \{1, \dots, \ell\}$
    \item For any horizontal character $\eta$, $\|\eta \circ g\|_{C^\infty[\vec{N}]} = \|\eta \circ \tilde{g}\|_{C^\infty[\vec{N}]}$. 
    \item If $F$ is a vertical character of frequency $\xi$, then $\tilde{F}$ is also a vertical character of frequency $\xi$.
\end{itemize}
\end{lemma}
\begin{proof}
To prove this, we factorize $g(0) = \{g(0)\} [g(0)]$ where $\psi(\{g(0)\}) \in [-1/2, 1/2)^d$ and $[g(0)] \in \Gamma$. Letting $g_1(n) = \{g(0)\}^{-1}g(n)[g(0)]^{-1}$ and $\tilde{F}(x) = F(\{g(0)\}x)$, it follows by Lemma~\ref{l:multiplication} that $\tilde{F}$ has Lipschitz constant $M^{O_k(d^{O(1)})}\|F\|_{\mathrm{Lip}}$ and $\tilde{F}$ is a vertical character of frequency $\xi$. This allows us to reduce to the case that $g(0) = \mathrm{id}_G$. To reduce to the case that $|\psi(g(e_i))| \le 1/2$, we write
$$\tilde{g}(n) := g(n) \prod_{i = 1}^\ell [g(e_i)]^{-n_i}$$
where $g(e_i) = \{g(e_i)\}[g(e_i)]$ with $\psi(\{g(e_i)\}) \in [-1/2, 1/2)^d$ and $[g(e_i)] \in \Gamma$. Note that $|\psi(\tilde{g}(e_i))| = |\psi(\{g(e_i)\})| \le \frac{1}{2}$, for any horizontal character $\eta$ we have $\|\eta \circ \tilde{g}\|_{C^\infty[\vec{N}]}= \|\eta \circ g\|_{C^\infty[\vec{N}]}$, and $F(g(\vec{n})\Gamma) = \tilde{F}(\tilde{g}(\vec{n})\Gamma)$ for each $\vec{n} \in \mathbb{Z}^\ell$. This gives the result.
\end{proof}
Another technique we will mention here is a technique we can use to reduce to the one dimensional vertical torus case. 
\begin{lemma}\label{l:trick2}
Suppose $F$ is a vertical character with frequency a nonzero $\xi$ of size at most $L$ and $F(g(n)\Gamma)$ is a nilsequence on a nilmanifold $G/\Gamma$ of step $s$. Then there exists a nilsequence $\tilde{F}(\tilde{g}(n)\tilde{\Gamma})$ on a nilmanifold $\tilde{G}/\tilde{\Gamma}$ with a Mal'cev basis $(Y_i)_{i = 1}^{d - d_{(s)} + 1}$ such that
\begin{itemize}
    \item $\tilde{G} = G/\mathrm{ker}(\xi)$ and $\tilde{\Gamma} = \Gamma/(\Gamma \cap \mathrm{ker}(\xi))$, and $\tilde{F}(\tilde{g}(n)\tilde{\Gamma}) = F(g(n)\Gamma)$.
    \item The vertical component of $\tilde{G}$ is one-dimensional and for $s \ge 2$, the horizontal component of $G$ and the horizontal component of $\tilde{G}$ are naturally isomorphic. 
    \item If $\eta$ is a horizontal character of size at most $M$ in $\tilde{G}$, then $\eta$ also has size at most $(LM)^{O_k(d)^{O_k(1)}}$ in $G$.
    \item $\|\tilde{F}\|_{\mathrm{Lip}(\tilde{G}/\tilde{\Gamma})} \le (ML)^{O_k(d)^{O_k(1)}}\|F\|_{\mathrm{Lip}(G/\Gamma)}$ and $\tilde{G}/\tilde{\Gamma}$ has complexity at most $(ML)^{O_k(d)^{O(1)}}$.
\end{itemize}
\end{lemma}
\begin{proof}
We first define a Mal'cev basis on $\tilde{G}/\tilde{\Gamma}$. $\{X_1, \dots, X_d\}$ be a Mal'cev basis for $G/\Gamma$ with complexity at most $M$. Let $\pi: G \to G/\mathrm{ker}(\xi)$ be the quotient map. Let $\overline{X_i}$ be the quotient of $X_i$ in $\log(G/\mathrm{ker}(\xi))$. Then since $X_i$ forms a $M$-rational basis for $G/\Gamma$, it follows from Lemma~\ref{l:weakbases} that $X_i$ is a $M^{O_k(d^{O(1)})}$-weak basis. Pick a maximal linearly independent subset $\mathcal{Y} = \{Y_j: j = 1, \dots, d - d_{(s)} + 1\}$ of the $\overline{X_i}$. It follows that since $\mathrm{ker}(\xi)$ is $(ML)^{O_k(d^{O(1)})}$-rational, each $\overline{X_k}$ is a $(ML)^{O_k(d^{O(1)})}$-rational combination of elements in $\mathcal{Y}$. Hence, $\mathcal{Y}$ is an $(ML)^{O_k(d^{O(1)})}$-weak basis. Invoking Lemma~\ref{l:ConstructingMalcev}, we obtain the desired Mal'cev basis for $\mathcal{Y}$. \\\\
To show that the vertical component is one-dimensional, note that if $G$ is $s$-step, let $\overline{g_1}, \dots, \overline{g_s} \in \tilde{G}$ and suppose that they lift to $g_1, \dots, g_s$, respectively. Then since the projection to $\mathrm{ker}(\xi)$ is a homomorphism, $[\overline{g_1}, \dots, \overline{g_s}] = \pi([g_1, \dots, g_s])$ which is nonzero if and only if $\xi([g_1, \dots, g_s])$ is nonzero. Hence, the vertical component is one-dimensional. The horizontal tori are seen to be isomorphic as $\mathrm{ker}(\xi) \subseteq [G, G]$. \\\\
In addition, if $\eta$ is a horizontal torus on $\tilde{G}/\tilde{\Gamma}$ of size at most $Q$, then $\eta(\exp(Y_j)) \le Q$ for each $j$. On the other hand, $\eta$ descends to $G/[G, G]$, so $\eta(\exp(\pi_{\mathrm{horiz}}(Y_j))) \le Q$. Each $\pi_{\mathrm{horiz}}(X_j)$ can be written as a $(LM)^{O_k(d^{O_k(1)})}$ rational combination of $\pi_{\mathrm{horiz}}(Y_j)$, so defining $\eta(\exp(X_i)) = \eta(\exp(\pi_{\mathrm{horiz}}(X_i)))$, it follows that $\eta(\exp(X_i)) \le Q(LM)^{O_k(d^{O_k(1)})}$ as desired. \\\\
Since $F$ is invariant under the kernel of $\xi$, $F$ descends to a map $\tilde{F}$ on $\tilde{G}$ via $\pi$. Thus, defining $\tilde{g} = \pi \circ g$, it follows that $\tilde{F}(\tilde{g}(n)\tilde{\Gamma}) = F(g(n)\Gamma)$. To show the Lipschitz bound on $\tilde{F}$, we see that (writing $x \sim z$ if $xz^{-1} \in \mathrm{ker}(\xi)$)
$$|F(x) - F(y)| = \inf_{x \sim x', y \sim y'} |F(x') - F(y')| $$
$$\le \inf_{x \sim x', y \sim y'} \|F\|_{\mathrm{Lip}(G)}d_G(x', y') \le \|F\|_{\mathrm{Lip}(G)}\inf_{x \sim x', y \sim y'} \inf_{(x_j)} \sum_j \min(|\psi(x_j x_{j + 1}^{-1})|, |\psi(x_{j + 1}x_j^{-1})|)$$
But since $G_{(s)}$ commutes with all of $G$, it follows that
$$\inf_{x \sim x', y \sim y'} \inf_{(x_j)} \sum_j \min(|\psi(x_j x_{j + 1}^{-1})|, |\psi(x_{j + 1}x_j^{-1})|) $$
$$\le (ML)^{O_k(d^{O_k(1)})}\inf_{(\pi(x_i))} \sum_j \min(|\psi(x_j x_{j + 1}^{-1})|, |\psi(x_{j + 1}x_j^{-1})|) \le (ML)^{O_k(d^{O_k(1)})}d_{\tilde{G}}(x, y).$$
In addition, we have
$$\|F\|_{\mathrm{Lip}(\tilde{G})} = \sup_{x, y} \frac{|F(x) - F(y)|}{d_{\tilde{G}}(xG', yG')} \le (ML)^{O_k(d^{O_k(1)})}\|F\|_{\mathrm{Lip}(G)}.$$
Since $\|F\|_{\mathrm{Lip}(G)} = \|F\|_{\mathrm{Lip}(G/\Gamma)}$ and similarly $\|F\|_{\mathrm{Lip}(\tilde{G})} = \|F\|_{\mathrm{Lip}(\tilde{G}/\tilde{\Gamma})}$, we have the desired inequality.
\end{proof}
\begin{remark}
We note that it is not necessarily true that if $H$ lies in the center of $G$, the horizontal component of $G/H$ agrees with the horizontal component of $G$. This is due to $[G/H, G/H] = [G, G]H/H$, so $(G/H)/([G/H, G/H]) \cong G/(H[G, G])$.
\end{remark}
%$$\mathbb{E}_{n \in [N]} F(g(n)\Gamma) = \mathbb{E}_{n \in [N]} F(g(Pn)\Gamma_1)$$
\subsection{Asymptotic notation}
We will specify asymptotic notation here. We say that $f = O(g)$ if there exists some absolute constant $C$ such that $|f| \le C|g|$. If $h$ is a variable, we say that $f = O_h(g)$ if there exists a constant $C_h$ depending on $h$ such that $|f| \le C_h|g|$. We shall also adopt Vinogradov's notation and to $f = O(g)$ as $f \ll g$ and $f = O_h(g)$ as $f \ll_h g$. In this paper, $s$ will often denote ``step" of a nilmanifold, and $k$ the ``degree" of the nilmanifold. Since we are in the setting where they are bounded, $O(g)$ will actually often be $O_k(g)$. Since in applications to arithmetic combinatorics, $s$ and $k$ are often constant, we make no effort to specify the explicit losses in terms of $s$ and $k$, though we anticipate since there is an iteration in $s$ and $k$ that losses are double exponential in those parameters. In addition, the case of $d = 1$ is somewhat degenerate, so quantities like $d^{O(1)}$ or $d^{O(d)}$ will actually denote $(d + 1)^{O(1)}$, and $(d + 1)^{O(d)}$, respectively. In an effort to shorthand a lot of the exponentials and quantities in \cite[Appendix A]{TT21}, the authors adopted the use of ``$\mathrm{poly}_m(\delta)$" (or denoted in our notation as $\mathrm{poly}_d(\delta)$) as any quantity lower bounded by $\gg \exp(\exp(-d^{O_k(1)}))\delta^{\exp(-d^{O_k(1)})}.$ Since many of our quantities are bounded above by a similarly cumbersome quantity $(\delta/M)^{-O_k(d^{O_k(1)})}$, we shall adopt a similar practice and instead denote $c(\delta)$ as any quantity lower bounded by $\gg (\delta/M)^{O_{k, \ell}(d^{O_{k, \ell}(1)})}$.

\subsection{Parameters used}
Given a nilpotent Lie group $G$ with a discrete cocompact subgroup $\Gamma$, we will specify some accompanying parameters that will be often used. The dimension of $G$ will be denoted $d$, and the dimension of $G_i$ will be denoted $d_i$. We will also specify that $d_{(j)}$ is the dimension of $G_{(j)}$ where as we defined above $G_{(j)}$ is the $j$th element of the standard filtration. As stated above, the step of the nilpotent Lie group will often be denoted $s$ and the degree of the filtration will often be denoted $k$. The complexity of the nilpotent Lie group will be bounded by $M$. We will work with the convention of $M \ge 2$.

\section{Refined bracket polynomial lemmas}\label{s:refbrack}
In this section, we deduce the refined bracket polynomial lemmas. This is, in the author's opinion, the crux of the argument of Theorem~\ref{t:mainresult1} and the main source of quantitative gain over \cite{GT12}. \\\\
We begin with the single variable refined bracket polynomial lemma.
\begin{lemma}\label{l:refined1}
Let $\delta \in (0, 1/10)$, $K \ge 2$ and $N \ge 2$ an integer. Suppose $\alpha, a \in \mathbb{R}^{d}$, $|a| \le M$, and
    $$\|\beta + a \cdot \{\alpha h\}\|_{\mathbb{R}/\mathbb{Z}} < K/N$$
    for $\delta N$ many $h \in [N]$. Then at least one of the following holds.
    \begin{itemize}
        \item We have $N \ll (\delta/d^{d}M)^{-O(1)}$ or $K/N > 1/100$;
        \item or there exists linearly independent $w_1, \dots, w_r \in \mathbb{Z}^d$, with $|w_1||w_2| \cdots |w_r| \le (\delta/d^{d} M)^{-O(1)}$ and linearly independent $\eta_1, \dots, \eta_{d - r} \in \mathbb{Z}^{d}$ with size at most $(\delta/d^{d} M)^{-O(1)}$ such that $\langle w_i, \eta_j \rangle = 0$ and
        $$\|\eta_j \cdot \alpha\|_{\mathbb{R}/\mathbb{Z}}, |w_i \cdot a| \le \frac{(\delta/M)^{-O(1)}d^{O(d)}K}{N}.$$
    \end{itemize}
\end{lemma}
This can be generalized to the multiparameter refined bracket polynomial lemma.
\begin{lemma}\label{l:refined2}
Let $\delta, M, K, d, r > 0$ fixed with $\delta \in (0, 1/10)$, $M > 1$, $d \in \mathbb{Z}$, $a_1, \dots, a_r, \alpha_1, \dots, \alpha_r$ be such that, $(a_1, \dots, a_r), (\alpha_1, \dots, \alpha_r) \in (\mathbb{R}^d)^r$, with $|a| \le M$, and $\beta \in \mathbb{R}^r$. Let $\vec{N} = (N_1, \dots, N_r)$ be an element in $\mathbb{Z}^r_{>0}$. Suppose that
$$\left\|\beta \cdot \vec{n} +  \sum_{j = 1}^r n_j a_j \cdot \left\{\sum_{i = 1}^r \alpha_i h_i\right\}\right\|_{C^\infty[\vec{N}]} \le K$$
    for $\delta |\vec{N}|$ many $h \in [\vec{N}]$. Then at least one of the following holds.
    \begin{itemize}
        \item We have $N_i \ll (\delta/(dr)^{d + r}M)^{-O(1)}$ for some $i$ or $K/N_i > \frac{1}{100r}$ for some $i$;
        \item or there exists linearly independent $w_1, \dots, w_r \in \mathbb{Z}^d$ and linearly independent $\eta_1, \dots, \eta_{d - r} \in \mathbb{Z}^{d}$ with size at most $(\delta/drM)^{-O(d + r)}$ in $\mathbb{Z}^{d}$ such that for each $(i, k)$, $\langle w_i, \eta_k \rangle = 0$ and for each $i$ and $j$
    $$\|\eta_i \cdot \alpha_j\|_{\mathbb{R}/\mathbb{Z}}, |w_i \cdot a_j| \le \frac{(\delta/drM)^{-O(d + r)}K}{N_j}.$$
    \end{itemize}
\end{lemma}
Although the multiparameter refined bracket polynomial lemma is ultimately the one we'll need for Theorem~\ref{t:mainresult1}, we have included the single variable one as the notation is less cumbersome and the argument is easier to absorb. We give a brief review of results from the geometry of numbers.
\subsection{The geometry of numbers}
We begin with the following Minkowski's second theorem below. Before we state it, we shall need some terminology. Given a lattice $\Gamma$ of $\mathbb{R}^d$ and a convex body $X$, the \emph{successive minima} of $X$ with respect to $\Gamma$, denoted $\lambda_i$, are defined as
$$\lambda_k := \inf\{\lambda > 0: \lambda \cdot X \text{ contains } k \text{ independent vectors of } \Gamma\}.$$
We have $\lambda_1 \le \lambda_2 \le \lambda_3 \le \lambda_4 \le \cdots \le \lambda_d < \infty.$
Minkowski's second theorem \cite[Theorem 3.30]{TV10} states that the successive minima satisfy the following property.
\begin{theorem}[Minkowski's second theorem]
Let $\Gamma$ be a lattice of full rank in $\mathbb{R}^d$, $B$ a symmetric convex body with successive minima $\lambda_1, \dots, \lambda_d$. Then there exists vectors $v_1, \dots, v_d$ such that
\begin{itemize}
    \item For each $1 \le j \le d$, $v_j$ lies on the boundary of $\lambda_j \cdot B$, but $\lambda_j \cdot B$ does not contain any vectors in $\Gamma$ outside the span of $v_1, \dots, v_{j - 1}$. 
    \item The cross-polytope with vertices $\pm v_i$ contain no elements of $\Gamma$ other than $0$.
    \item We have
    $$\frac{2^d |\Gamma/(\mathrm{Span}_{\mathbb{Z}}(v_1, \dots, v_d))|}{d!} \le \frac{\lambda_1 \cdots \lambda_d\mathrm{vol}(B)}{\mathrm{vol}(\mathbb{R}^d/\Gamma)} \le 2^d.$$
\end{itemize}
\end{theorem}
The next two statements below show, as a corollary of Minkowski's second theorem that such an intersection must contain a \emph{generalized arithmetic progression} (see below for definition) which has size roughly $d^{-O(d)}$ times the intersection of the convex body and the lattice. The next lemma we'll need, a statement and proof of which can be found in \cite[Lemma 3.14]{TV10}, is Ruzsa's covering lemma.
\begin{lemma}[Ruzsa's covering lemma]\label{l:Ruzsa}
For any bounded subsets $A$ and $B$ of $\mathbb{R}^d$ with positive measure, we may cover $B$ by at most $\min\left(\frac{\mathrm{vol}(A + B)}{\mathrm{vol}(A)}, \frac{\mathrm{vol}(A - B)}{\mathrm{vol}(A)}\right)$ many translates of $A - A$. 
\end{lemma}
We use Ruzsa's covering lemma in the various proofs of the refined bracket polynomial lemma to convex subsets of $\mathbb{R}^d$. In the case of convex and symmetric $A$ and $B$, Ruzsa's covering lemma is extra powerful, since $A + A$ and $A - A$ are both just dilates of $A$.\\\\
A \emph{generalized arithmetic progression} (GAP) in $\mathbb{R}^n$ is a subset of the form $\{\ell_1v_1 + \cdots + \ell_d v_d: \ell_i \in [N_i]\}$ where $v_1, \dots, v_d \in \mathbb{R}^n$. The generalized arithmetic progression is \emph{proper} if each of the elements $\ell_1 v_1+ \cdots + \ell_d v_d$ are distinct. The \emph{rank} of a proper generalized arithmetic progression is the quantity $d$. As a consequence of Minkowski's second theorem and Ruzsa's covering lemma, we have the following, which is \cite[Lemma 3.33]{TV10}.
\begin{proposition}\label{p:properprogression}
    Let $B$ be a symmetric convex body and let $\Gamma$ be a lattice in $\mathbb{R}^d$. Then there exists a GAP $P = \{\ell_1 v_1 + \cdots + \ell_{d'} v_{d'}: \ell_i \in [N_i]\}$ in $B \cap \Gamma$ with $v_1, \dots, v_{d'}$ linearly independent such that $|P| \ge O(d)^{-7d/2}|B \cap \Gamma|$.
\end{proposition}
\begin{remark}
\cite[Lemma 3.33]{TV10} do not require $v_1, \dots, v_{d'}$ to be linearly independent; however, their proof does give linearly independent $v_1, \dots, v_{d'}$.
\end{remark}
We are now ready to prove Lemma~\ref{l:refined1} and Lemma~\ref{l:refined2}.
\subsection{Deduction of Lemma~\ref{l:refined1}}

We deduce Lemma~\ref{l:refined1} from the following.
\begin{lemma}\label{l:easierrefined1}
    Let $\delta \in (0, 1/10)$, and $N$ be an odd prime. Suppose $\alpha, a \in \mathbb{R}^{d}$ and $\alpha$ is of denominator $N$, $|a| \le M$, and
    $$\|\beta + a \cdot \{\alpha h\}\|_{\mathbb{R}/\mathbb{Z}} < K/N$$
    for $\delta N$ many $h \in [N]$. Then one of the following holds.
    \begin{itemize}
        \item We have $N \ll (\delta/d^{d}M)^{-O(1)}$ or $K/N > 1/10$;
        \item or there exists linearly independent $w_1, \dots, w_r$, with $|w_1||w_2| \cdots |w_r| \le (\delta/d^{d} M)^{-O(1)}$ and linearly independent $\eta_1, \dots, \eta_{d - r}$ in $\mathbb{Z}^{d}$ with size at most $(\delta/d^{d} M)^{-O(1)}$ such that $\langle w_i, \eta_j \rangle = 0$ and
    $$\|\eta_j \cdot \alpha\|_{\mathbb{R}/\mathbb{Z}} = 0, \hspace{0.1in} |w_i \cdot a| \le \frac{(\delta/M)^{-O(1)}d^{O(d)}K}{N}.$$
    \end{itemize}
\end{lemma}
\begin{proof}[Proof of Lemma~\ref{l:refined1} assuming Lemma~\ref{l:easierrefined1}]
We begin with
$$\|a \cdot \{\alpha h\} + \gamma\|_{\mathbb{R}/\mathbb{Z}} < \frac{K}{N}$$
for $\delta N$ many $h \in [N]$. By replacing $N$ with prime $N'$ with $10N \le N' \le 20N$ and $\delta$ with $\delta/20$, we may assume that all such $h \in [N/10]$. Now take $\alpha'$ with $N\alpha'\in\mathbb{Z}^d$ such that $|\alpha - \alpha'| \le \frac{1}{2N}$ and denoting $v = \alpha - \alpha'$. Then
$$a \cdot \{\alpha h\} \equiv a \cdot \{\alpha' h\} + a \cdot \{v h\} + a \cdot (\{\alpha h\} - \{\alpha' h\} - \{v h\}).$$
However, we observe that since $h \in [N/10]$, it follows that $\{vh\} = vh$, and since $\alpha - \alpha' - v = 0$, it follows that $\{\alpha h\} - \{\alpha' h\} - \{vh\} \in \mathbb{Z}^{d}$, and since $|\{\alpha h\} - \{\alpha' h\} - \{vh\}| \le \frac{3}{2}$, it follows that each coordinate of that vector is either $0, 1$, or $-1$. By the Pigeonhole principle, there exists some vector $w \in \mathbb{Z}^{d}$ such that for $\delta/3^{d}$ many elements $h \in [N]$ that
$$\|a \cdot \{\alpha' h\} + a \cdot w + \gamma + a \cdot vh \|_{\mathbb{R}/\mathbb{Z}} < \frac{K}{N}.$$
Finally, rewriting $a \cdot v h = a \cdot v N\{h/N\}$ (here we are using the fact that $h \in [N/10]$ so $h/N = \{h/N\}$), and replacing $\gamma$ with $\gamma + a \cdot w$, and letting $b = Na \cdot v$, it follows that
$$\|a \cdot \{\alpha' h\} + b\{1/N\} + \gamma\|_{\mathbb{R}/\mathbb{Z}} < \frac{K}{N}.$$
Applying Lemma~\ref{l:easierrefined1} gives $w_1, \dots, w_r$ and $\eta_1, \dots, \eta_{d + 1 - r}$ such that if $\eta_i = (\eta_i', \eta_i^{d+1})$ and $w_i = (w_1', w_1^{d+1})$, then 
$$\eta_i' \cdot \alpha' + \eta_i^{d+1} /N \equiv 0 \pmod{1}$$
$$|w_i' \cdot a + w_i^{d+1}| \le \frac{(\delta/M)^{-O(1)}d^{O(d)}K}{N}.$$
Suppose $w_1^{d+1} \neq 0$ and denote $\tilde{w}_i := w_i^{d+1} w_1' - w_i'w_1^{d+1}$ for $2 \le i \le d + 1 - r$. We see that 
$$|\tilde{w_i} \cdot a| \le \frac{(\delta/M)^{-O(1)}d^{O(d)}K}{N}.$$
We claim that $\tilde{w}_i$'s are linearly independent of each other. Suppose there exists $c_i$ such that
$$\sum_{i \neq 1} c_i(w_i^{d+1} w_1' - w_i'w_1^{d+1}) = 0.$$
Rearranging, we have
$$w_1'\left(\sum_{i \neq 1} c_i w_i^{d+1} \right) + \sum_{i \neq 1} (-c_i w_1^{d+1})w_i' = 0.$$
In addition,
$$\sum_{i \neq 1} w_1^{d+1}\left(\sum_{i \neq 1} c_i w_i^{d+1} \right) + \sum_{i \neq 1} (-c_i w_1^{d+1})w_i^{d+1} = 0$$
and therefore
$$w_1\left(\sum_{i \neq 1} c_i w_i^{d+1} \right) + \sum_{i \neq 1} (-c_i w_1^{d+1})w_i = 0.$$
By linear independence of the $w_i$'s and since $w_1^{d+1}$ is nonzero, it follows that $c_i = 0$. Thus, the $\tilde{w}_i$ are independent of each other. To see that $\tilde{w}_i$'s are orthogonal to the $\eta_j'$'s, note that
\begin{align*}
\tilde{w}_i \cdot \eta_j' &= w_i^{d+1} w_1' \cdot \eta_j' - w_i' \cdot \eta_j' w_1^{d+1} \\
\eta_j \cdot w_i &= \eta_j^{d+1} w_i^{d+1} + \eta_j' \cdot w_i' = 0 \\
\eta_j \cdot w_1 &= \eta_j^{d+1} w_1^{d+1} + \eta_j' \cdot w_1' = 0.
\end{align*}
Subtracting the second and third equations gives that the first expression is zero. Finally, we claim that the $\eta_j'$'s are linearly independent of each other. To see this, we note that $(\eta_j', \eta_j^{d+1})$ are orthogonal to $(\tilde{w}_i, 0)$ and $(w_1', w_1^{d+1})$. Since $(0, 1)$ is not orthogonal to $(w_1', w_1^{d+1})$, it follows that $(\eta_j', \eta_j^{d+1})$ cannot span $(0, 1)$. Thus, $(\eta_j', \eta_j^{d+1})$ and $(0, 1)$ are linearly independent, which implies that $\eta_j'$ are linearly independent. \\\\
If $w_1^{d+1} = 0$ and $w_i^{d+1} \neq 0$ for some $i$, we let $w_i^{d+1}$ play the role of $w_1^{d+1}$ in the above argument. If $w_i^{d+1} = 0$ for all $i$, we find that
$$|w_i' \cdot a| \le \frac{(\delta/M)^{-O(1)}d^{O(d)}K}{N}$$
and that
$$\|\eta_j' \cdot \alpha\|_{\mathbb{R}/\mathbb{Z}} \le \frac{(\delta/M)^{-O(1)}d^{O(d)}K}{N}.$$
Picking a linearly independent subset of the $\eta_j'$ completes the proof.    
\end{proof}

We now prove Lemma~\ref{l:easierrefined1}.
\begin{proof}[Proof of Lemma~\ref{l:easierrefined1}]
\textbf{Step 1: Reducing to the case of when $\beta = 0$}. Using the Pigeonhole principle, there exists at least $\delta N/M$ many $h$'s such that
$$|\beta + a \cdot \{\alpha h\} - k| < K/N$$
for some $k \in \mathbb{Z}$. By using the Pigeonhole principle again, there exists a sign pattern in $\{-1, 0, 1\}^{d}$ such that for $\delta N/(3^{d}Md)$ many $h \in [N]$ all satisfying the sign pattern (i.e., $\{\alpha h\}$ are all the same sign for all such $h$) and such that
$$|\beta + a \cdot \{\alpha h\} - k| < K/N.$$
Note that if $\{x\}$ and $\{y\}$ have the same sign then $\{x - y\} = \{x\} - \{y\}$. Thus, taking the difference of any two such $h$ and $h$',
$$|a \cdot \{\alpha (h - h')\}| < 2K/N.$$

\textbf{Step 2: Invoking Minkowski's second theorem}. Let 
$$B= \left\{|a \cdot x| < \frac{2K}{N}, |x_i| < \frac{1}{2} \text{ for all } i \in [d] \right\}$$
and $\Gamma =\alpha \mathbb{Z} + \mathbb{Z}^{d}$. It follows from Step 1 that at least $\delta \frac{N}{3^{d}dM}$ many elements of the lattice lie in $B$. Since $B$ is a convex body and $\Gamma$ is a lattice of $\mathbb{R}^d$, we have by Proposition~\ref{p:properprogression}, there exists linearly independent vectors $v_1, \dots, v_{d'}$ of the lattice and $N_1, \dots, N_{d'}$ with $N_1 \cdots N_d = \frac{\delta N}{d^{O(d)}M}$ and 
$$P := \{\ell_1v_1 + \cdots \ell_{d'} v_{d'}: |\ell_i| \le N_i\} \subseteq B.$$
Since $N$ is an odd prime, each $v_i$ is of the form $\{\alpha h_i\}$ since that is unique in the coset $\alpha h_i + \mathbb{Z}^d$ where each coordinate lies in $(-1/2, 1/2)$. \\\\
Let 
$$\Gamma' = v_1\mathbb{Z} + v_2\mathbb{Z} + \cdots v_{d'}\mathbb{Z} \subseteq \Gamma, \text{ and } V = \mathrm{span}(v_1, \dots, v_{d'}).$$
For $x \in \mathbb{R}^d$ and $r > 0$, let $B_{r, \infty}(x)$ denote the $\ell^\infty$ ball of radius $r$ around $x$ and $B_r(x)$ the $\ell^2$ ball of radius $r$ around $x$. Since $N$ is prime, $B_{1/2, \infty}(0)$ contains at most $N$ points of $\Gamma$ and therefore at most $N$ points of $\Gamma'$. Thus, by placing a fundamental parallelepiped dilated by two of $\Gamma'$ around each lattice point in the ball of radius $1/2$, we see that the volume of the ball of radius $1/2$ in $V$ is at most $N/2^d$ times the the volume of the fundamental parallelepiped of $\Gamma'$ in $V$. We have the upper bound
$$\mathrm{vol}(B_{1/2, \infty}(0)\cap V) \le 2^dN\mathrm{vol}(V/\Gamma').$$
Let $O$ denote the interior of the generalized arithmetic progression $P$ inside $V$. Note that we have the lower bound of
$$\mathrm{vol}(O) \ge N(\delta/(d^{O(d)}Mr)) \mathrm{vol}(V/\Gamma').$$

By Ruzsa's covering lemma (using the fact that $O - O = 2O$ since $O$ is the interior of a generalized arithmetic progression), we may cover the ball of radius $1/2$ inside $V$ with at most $T := \delta^{-1}Md^{O(d)}$ many translations of $O$. Since the ball of radius $1/2$ inside $V$ is convex and connected, it must be contained in a dilation of $T$ of $O$. This implies that $B_{T^{-1}, \infty}(0) \cap V \subseteq O \subseteq B \cap V$.\\\\
Let 
$$P_L = \{n_1v_1 + \cdots + n_{d'} v_{d'}: |n_i| \le LN_i\}$$
This is contained in $B_{dL/2}(0) \cap V$. Letting $v_j = \{h_j \alpha\}$, we see that if $n_1h_1 + n_2h_2 + \cdots n_{d'}h_{d'} \equiv 0 \pmod{N}$, then $n_1v_1 + \cdots + n_{d'}v_{d'}$ is a point in $\mathbb{Z}^{d}$. By the Pigeonhole principle, there exists a residue class $q$ such that $n_1h_1 + n_2h_2 + \cdots n_{d'}h_{d'} \equiv q \pmod{N}$ for at least $L^{d'} \delta/M d^{-O(d)}$ many elements in $P_L$. Thus, taking the difference of any two such elements, it follows that $P_L$ contains at least $L^{d'}\frac{\delta}{M} d^{-O(d)}$ many points in $\mathbb{Z}^{d}$. Consequently, letting $L$ go to infinity, $\mathbb{Z}^{d} \cap \Gamma'$ is a lattice in $V$ with covolume at most $\delta^{-1}M d^{O(d)}$. 

This implies, via Minkowski's second theorem that the product of the successive minima of $\mathbb{Z}^d \cap V$ with respect to $B_1(0) \cap V$ is at most $\delta^{-1} M d^{O(d)}$. Let $w_1, \dots, w_{d'}$ be those linearly independent vectors that correspond to the $d'$ successive minima. Then $|w_1| \cdots |w_{d'}| \le \delta^{-1} M d^{O(d)}$. Since the ball of radius $\delta/(Md^{O(d)})$ inside $V$ is contained in $B$, it follows that
$$|w_i \cdot a| \le \frac{(\delta/M)^{-O(1)}d^{O(d)}K}{N}.$$

\textbf{Step 3: Finishing the proof}. Now, letting $e_i$ be the coordinate vector in $\mathbb{R}^{d}$ with a $1$ in the $i$th coordinate and $0$ everywhere else, there exists $e_{i_1}, \dots, e_{i_{d - d'}}$ such that $w_1, \dots, w_{d'}, e_{i_1}, \dots, e_{i_{d - d'}}$ are linearly independent. Consider the $d \times d$ matrix $A = (w_1, \dots, w_{d'}, e_{i_1}, \dots, e_{i_{d - d'}})^{T}$. By Cramer's rule, $A^{-1} = \text{adj}(A)\det(A)^{-1}$ (where $\text{adj}(A)$ is the adjugate matrix). It follows by taking $\eta_1, \dots, \eta_{d - d'}$ be the last $d - d'$ columns of $\text{adj}(A)$, it follows (using Hadamard's inequality of $|v_1 \wedge v_2 \wedge \cdots \wedge v_d| \le d^{O(d)}|v_1| \cdots |v_d|$) that for all $i$,
\begin{itemize}
    \item $\eta_i \in \mathbb{Z}^d$,
    \item $|\eta_i| \le (\delta/M)^{-O(1)}d^{O(d)}$,
    \item $\langle \eta_i, w_j \rangle = 0$ for all $j$,
    \item and for $\delta N/(d^{O(d)}M)$ many $h \in [N]$,
    $\|\eta_i \cdot \{\alpha h\}\|_{\mathbb{R}/\mathbb{Z}} = 0.$
    Namely, $h\in\{\ell_1h_1+\cdots\ell_{d'}h_d:|\ell_i|\le N_i\}$ work since $\eta_i \in \mathbb{Z}^d$ and $\eta_i$ annihilates $P$.
\end{itemize}

Since $N$ is prime and sufficiently larger than each $|\eta_i|$, and $\alpha$ has denominator $N$, it follows that
$$\|\eta_i \cdot \alpha\|_{\mathbb{R}/\mathbb{Z}} = 0.$$
\end{proof}
\subsection{Deduction of Lemma~\ref{l:refined2}}
As in the one-dimensional refined bracket polynomial lemma, we invoke the following lemma to deduce Lemma~\ref{l:refined2}.
\begin{lemma}\label{l:easierrefined2}
Let $\delta, M, K, d, r > 0$ fixed with $\delta \in (0, 1/10)$, $M \ge 2$, $d \in \mathbb{Z}$, $a_1, \dots, a_r, \alpha_1, \dots, \alpha_r$ be such that, $(a_1, \dots, a_r), (\alpha_1, \dots, \alpha_r) \in (\mathbb{R}^d)^r$, and $\beta \in \mathbb{R}^r$, with $|a| \le M$. Let $\vec{N} = (N_1, \dots, N_r)$ be an element in $\mathbb{Z}^r$ with $N_i$ being distinct odd primes. Suppose that $\alpha_j$ has denominator $N_j$ and that there are at least $\delta |\vec{N}|$ many $\vec{h} \in [\vec{N}]$ such that
$$\left\|\beta \cdot \vec{n} +  \sum_{j = 1}^r n_j a_j \cdot \left\{\sum_{i = 1}^r \alpha_i h_i\right\}\right\|_{C^\infty[\vec{N}]} \le K.$$
Then at least one of the following holds.
    \begin{itemize}
        \item We have $N_i \ll (\delta/(dr)^{d + r}M)^{-O(1)}$ for some $i$ or $K/N_i > \frac{1}{10r}$ for some $i$;
        \item or there exists linearly independent $w_1, \dots, w_r$ and linearly independent $\eta_1, \dots, \eta_{d - r}$ in $\mathbb{Z}^{d}$ with size at most $(\delta/drM)^{-O(d + r)}$ in $\mathbb{Z}^{d}$ such that for each $(i, k)$, $\langle w_i, \eta_k \rangle = 0$ and for each $i$ and $j$
    $$\|\eta_i \cdot \alpha_j\|_{\mathbb{R}/\mathbb{Z}} = 0 \text{ and } |\eta_i \cdot w_j| \le \frac{(\delta/drM)^{-O(d + r)}K}{N_j}.$$
    \end{itemize}
\end{lemma}
\begin{proof}[Proof of Lemma~\ref{l:refined2} assuming Lemma~\ref{l:easierrefined2}]
We have that for at least $\delta |\vec{N}|$ many $\vec{h} \in [\vec{N}]$,
$$\left\|\beta \cdot \vec{n} +  \sum_{j = 1}^r n_j a_j \cdot \left\{\sum_{i = 1}^r \alpha_i h_i\right\}\right\|_{C^\infty[\vec{N}]} \le K.$$
We pick distinct odd primes $N_1', \dots, N_r'$ such that $10rN_i \le N_i' \le 20rN_i$. \\\\
Replacing $N_i$ with $N_i'$ we may assume that each $h_i \in [N_i/(10r)]$. We now approximate $\alpha_i$ with $\alpha_i'$ with $N_i\alpha_i' \in \mathbb{Z}^d$ such that $\|\alpha_i - \alpha_i'\|_\infty \le 1/(2N_i)$. We see that
$$\left\|\beta \cdot \vec{n} + \sum_{j = 1}^r n_j a_j \cdot \left\{\sum_{i = 1}^r \alpha_i' h_i\right\} + \sum_{j = 1}^r n_j a_j \cdot \left\{\sum_{i = 1}^r (\alpha_i - \alpha_i')h_i\right\} + P(\vec{n}, \vec{h})\right\|_{C^\infty[\vec{N}]} \le K$$
where
$$P(\vec{n}, \vec{h}) = \sum_{j = 1}^r n_j a_j \cdot \left\{\sum_{i = 1}^r \alpha_i h_i\right\} - \sum_{j = 1}^r n_j a_j \cdot \left\{\sum_{i = 1}^r \alpha_i' h_i\right\} - \sum_{j = 1}^r n_j a_j \cdot \left\{\sum_{i = 1}^r (\alpha_i - \alpha_i')h_i\right\}.$$
The point is that $\left\{\sum_{i = 1}^r \alpha_i h_i\right\} - \left\{\sum_{i = 1}^r \alpha_i' h_i\right\} - \left\{\sum_{i = 1}^r (\alpha_i - \alpha_i')h_i\right\} \in \mathbb{Z}^d$ and this can take at most $(2r + 1)^d$ many values. Pigeonholing in one of those values, and observing that $\left\{\sum_{i = 1}^r (\alpha_i - \alpha_i')h_i\right\} = \sum_{i = 1}^r (\alpha_i - \alpha_i')h_i$, there exists some $\beta' \in \mathbb{R}^d$ such that for at least $\delta|\vec{N}|/(2r + 1)^{d}$ many values $h \in [\vec{N}]$ satisfying 
$$\left\|\beta' \cdot \vec{n} +  \sum_{j = 1}^r n_j a_j \cdot \left\{\sum_{i = 1}^r \alpha_i' h_i\right\} + \sum_i \left(\sum_{j} \tilde{a}_{ij} n_j\right)\{h_i/N_i\}\right\|_{C^\infty[\vec{N}]} \le 11Kr$$
where $\tilde{a}_{ij} = N_i(\alpha_i - \alpha_i')a_j$. Here, we note that since $h_i \in [N_i/(10r)]$, we have $h_i = N_i\{h_i/N_i\}$. \\\\
Applying Lemma~\ref{l:easierrefined2} gives $w_1, \dots, w_{d'}$ and $\eta_1, \dots, \eta_{d + r - d'}$ such that denoting $\eta_i = (\eta_i', \eta_i^{d+1}, \dots, \eta_i^{d + r})$ and $w_i = (w_i', w_i^{d+1}, \dots, w_i^{d + r})$, we have
$$\eta_i' \cdot \alpha'_j + \eta_i^{d+j} /N_j \equiv 0 \pmod{1}$$
$$|w_i' \cdot a_j + \sum_k w_i^{d+k} \cdot \tilde{a}_{jk}| \le 11Kr/N_j.$$
Now taking $u_1, \dots, u_m$ to be linearly independent and orthogonal to the subspace generated by the $\eta_1', \dots, \eta_{d + r - d'}'$, we see that $(u_1, 0), \dots, (u_m, 0)$ is orthogonal to $\eta_1, \dots, \eta_{d + r - r'}$, so is in the subspace spanned by $w_1, \dots, w_{d'}$. Furthermore, by Cramer's rule, we may choose $u_1, \dots, u_m$ to be integer vectors with size bounded by $(\delta/drM)^{-O(d + r)}$, so that $u_1, \dots, u_m$ are each spanned by $(\delta/drM)^{-O(d + r)}$-integer combination of $w_1, \dots, w_{d'}$. Thus,
$$\|\eta_i' \cdot \alpha_j\|_{\mathbb{R}/\mathbb{Z}} \le (\delta/drM)^{-O(d + r)}/N_j$$
$$|u_i \cdot a_j| \le \frac{K(\delta/drM)^{-O(d + r)}}{N_j}$$
as desired. 
\end{proof}
We now prove Lemma~\ref{l:easierrefined2}.
\begin{proof}[Proof of Lemma~\ref{l:easierrefined2}]
\textbf{Step 1: Reducing to the case of when $\beta = 0$}. We have for $\delta|\vec{N}|$ many $h \in [\vec{N}]$ that
$$\left\|\beta_j + a_j \cdot \left\{\sum_{i = 1}^r \alpha_i h_i\right\}\right\|_{\mathbb{R}/\mathbb{Z}} < \frac{K}{N_i}$$
for each $j = 1, \dots, r$. By the Pigeonhole principle, there exists $k_j \in \mathbb{Z}$ such that for a proportion of $\delta/(2Mdr)$ many $h_1, \dots h_r$, we have
$$\left|\beta_j + a_j \cdot \left\{\sum_{i = 1}^r \alpha_i h_i\right\} - k_j\right| < K/N_j.$$
Pigeonholing again in the sign pattern of $\{\sum_{i = 1}^r \alpha_i h_i\} \in \{0, 1, -1\}^d$ to obtain a proportion of $\delta/(2Mdr3^{d})$ many $h_1, \dots, h_r$ such that
$$\left|\beta_j + a_j \cdot \left\{\sum_{i = 1}^r \alpha_i h_i\right\} - k_j \right| < K/N_j$$
and all of $\{\sum_{i = 1}^r \alpha_i h_i\}$ have the same sign as $\vec{h}$ ranges over the set we Pigeonholed in. Taking the difference of any two such $\vec{h}$'s, we obtain
$$\left|a_j \cdot \left\{\sum_{i = 1}^r \alpha_i h_i\right\}\right|  < 2K/N_j$$
for $\delta|\vec{N}|/(2Mdr3^{d})$ many $\vec{h}$ in $[\vec{N}]$. \\\\
Let
$$B_j = \{|a_j \cdot x| < 2K/N_j, |x_i| < 1/2 \text{ for all } i \in [d]\}, \text{ } B = \bigcap_{i = 1}^r B_i, \text{ and }\Gamma = \alpha_1 \mathbb{Z} + \alpha_2\mathbb{Z} + \cdots \alpha_r\mathbb{Z} + \mathbb{Z}^{d}.$$
Without a loss of generality, let $\alpha_1, \dots, \alpha_{r'} \not\in \mathbb{Z}^d$ and let $\alpha_{r' + 1}, \dots, \alpha_r \in \mathbb{Z}^d$. Let $[\vec{N'}] = \prod_{i = 1}^{r'} [N_i]$. We now claim that $|B \cap \Gamma| \ge \delta/(2Mdr3^{d}) |\vec{N'}|$. To see this, suppose
$$\sum_i \alpha_i h_i \equiv \sum_i \alpha_i h_i' \pmod{1}$$
for $h_i, h_i'$ integers. Then
$$\sum_i \alpha_i (h_i - h_i') \equiv 0 \pmod{1}.$$
Clearing denominators,
$$\sum_i N_1 \cdots N_{r'} \alpha_i (h_i - h_i') \equiv 0 \pmod{|\vec{N'}|}$$
which implies
$$N_1 \cdots N_{r'} \alpha_i (h_i - h_i') \equiv 0 \pmod{N_i}.$$
As $\alpha_i$ is not integral for $1 \le i \le r'$ we have $h_i \equiv h_i' \pmod{N_i}$. \\\\
\textbf{Step 2: Invoking Minkowski's second theorem}. By Proposition~\ref{p:properprogression}, there exists a generalized arithmetic progression 
$$P := \{m_1v_1 + m_2v_2 + \cdots + m_{d'} v_{d'}: |a_i| \le M_i\} \subseteq B \cap \Gamma$$ 
with $v_1, \dots, v_{d'}$ linearly independent and $M_1M_2 \cdots M_{d'} = \frac{\delta}{rMd^{O(d)}} |\vec{N'}|$. Let 
$$\Gamma' = v_1\mathbb{Z} + v_2\mathbb{Z} + \cdots v_{d'}\mathbb{Z} \text{ and } V = \text{span}(v_1, \dots, v_{d'}).$$
By placing a fundamental parallelepiped dilated by two of $\Gamma'$ around each lattice point in the ball of radius $1/2$, we see that the volume of the ball of radius $1/2$ in $V$ is at most $N/2^d$ times the the volume of the fundamental parallelepiped of $\Gamma'$ in $V$. We have the upper bound:
$$\mathrm{vol}(B_{1/2, \infty}(0)) \le 2^dN\mathrm{vol}(V/\Gamma').$$
Let $O$ denote the interior of the generalized arithmetic progression $P$ inside $V$. Note that we have the lower bound of
$$\mathrm{vol}(O) \ge N(\delta/(d^{O(d)}Mr)) \mathrm{vol}(V/\Gamma').$$
By Ruzsa's covering lemma (using the fact that $O - O = 2O$ since $O$ is the interior of a generalized arithmetic progression), we may cover the ball of radius $1/2$ inside $V$ with at most $T := \delta^{-1}Mrd^{O(d)}$ many translations of $O$. Since the ball of radius $1/2$ inside $V$ is convex and connected, it must be contained in a dilation of $T$ of $O$. This implies that $B_{T^{-1}, \infty}(0) \cap V \subseteq O \subseteq B\cap V$. \\\\
Let $P_L = \{m_1v_1 + m_2v_2 + \cdots + m_{d'} v_{d'}: |a_i| \le LM_i\}$. This lies inside $B_{dL/2}(0) \cap V$. Since $N_1, \dots, N_{r'}$ are odd, there exists some $\vec{h_i} = (h_i^1, \dots, h_i^{r'}) \in [\vec{N'}]$ such that $v_i = \{\vec{h_i} \cdot (\alpha_1, \dots, \alpha_{r'})\}$. By the Pigeonhole principle, there exists a residue classes $(q_1, \dots, q_{r'})$ such that for at least $L^{d'}\frac{\delta}{rMd^{O(d)}}$ many $m_1, \dots, m_{d'}$,
$$m_1h_1^j + m_2h_2^j + \cdots + m_{d'} h_{d'}^j \equiv q_j \pmod{N_j}.$$
Taking the difference of any two such $m_1, \dots, m_{d'}$, we obtain that for at least $L^{d'}\frac{\delta}{rMd^{O(d)}}$ many $m_1, \dots, m_{d'}$, that
$$m_1h_1^j + m_2h_2^j + \cdots + m_{d'} h_{d'}^j \equiv 0 \pmod{N_j}$$
and so there are at least $L^{d'}\frac{\delta}{rMd^{O(d)}}$ many elements in $P_{2L}$ that lie in $\mathbb{Z}^{d}$. Consequently, $B_{dL}(0) \cap V$ contains at least $L^{d'}\frac{\delta}{rMd^{O(d)}}$ many elements in $\mathbb{Z}^d$. \\\\
By Minkowski's second theorem, $V \cap \mathbb{Z}^{d}$ has $d'$ linearly independent vectors $w_1, \dots, w_{d'}$ such that $|w_1| \cdots |w_{d'}| \le (\delta/(Mrd^{O(d)}))^{-1}$. Since the ball of radius $\frac{\delta}{Mrd^{O(d)}}$ lies inside $B$, it follows that
$$|w_i \cdot a_j| \le \frac{K(\delta/(Mrd^{O(d)}))^{-1}}{N_j}.$$
\textbf{Step 3: Finishing the argument}. Now, letting $e_i$ be the coordinate vector in $\mathbb{R}^{d}$ with a $1$ in the $i$th coordinate and $0$ everywhere else, there exists $e_{i_1}, \dots, e_{i_{d - d'}}$ such that $w_1, \dots, w_{d'}, e_{i_1}, \dots, e_{i_{d - d'}}$ are linearly independent. Consider the $d \times d$ matrix $A = (w_1, \dots, w_{d'}, e_{i_1}, \dots, e_{i_{d - d'}})^{T}$. By Cramer's rule, $A^{-1} = \text{adj}(A)\det(A)^{-1}$ (where $\text{adj}(A)$ is the adjugate matrix). By taking $\eta_1, \dots, \eta_{d - d'}$ be the last $d - d'$ columns of $\text{adj}(A)$, it follows (using Hadamard's inequality of $|v_1 \wedge v_2 \wedge \cdots \wedge v_d| \le d^{O(d)}|v_1| \cdots |v_d|$) that for each $i$,
\begin{itemize}
    \item $\eta_i \in \mathbb{Z}^d$
    \item $|\eta_i| \le (\delta/(Mrd^{O(d)}))^{-1}$
    \item $\langle \eta_j, w_j \rangle = 0$ for each $i, j$
    \item and for $\delta N/(d^{O(d)}Mr)$ many $\vec{h} \in [\vec{N}']$, $\|\eta_i \cdot \{(\alpha_1, \dots, \alpha_{r'}) \cdot \vec{h}\}\|_{\mathbb{R}/\mathbb{Z}} = 0$ for the same reason as the one-dimensional case.
\end{itemize}
Since each $N_i$ is prime and sufficiently larger than each $|\eta_i|$, and $\alpha$ has denominator $N$, and since $\alpha_{r' + 1}, \dots, \alpha_r \in \mathbb{Z}^d$ it follows that
$$\|\eta_i \cdot \alpha_j\|_{\mathbb{R}/\mathbb{Z}} = 0.$$
\end{proof}
\subsection{Corollary of Lemma~\ref{l:refined2}}
We now deduce the following corollary of Lemma~\ref{l:refined2} which will be used in our main result.
\begin{corollary}\label{c:bracketcorollary2}
Let $N, \delta, M, K, d, r > 0$ fixed with $\delta \in (0, 1/10)$, $d \in \mathbb{Z}$, $a_1, \dots, a_r, \alpha_1, \dots, \alpha_r$ be such that, $(a_1, \dots, a_r), (\alpha_1, \dots, \alpha_r) \in (\mathbb{R}^d)^r$, and $\beta \in \mathbb{R}^r$, $\gamma \in \mathbb{R}^{r^2}$ with $|a| \le M$. Suppose that for $\delta |\vec{N}|$ many elements $\vec{h} \in [\vec{N}]$,
$$\left\|\beta \cdot \vec{n} +  \sum_{j = 1}^r n_j a_j \cdot \left\{\sum_{i = 1}^r \alpha_i h_i\right\} + \sum_{i, j \in [r]} \gamma_{ij} n_ih_j\right\|_{C^\infty[\vec{N}]} \le K.$$
Then one of the following holds.
\begin{itemize}
    \item We have $N_i \ll (\delta/(dr)^{d + r}M)^{-O(1)}$ for some $i$ or $K/N_i > 1/(100r)$ for some $i$;
    \item or there exists linearly independent $w_1, \dots, w_{d'}$ and $\eta_1, \dots, \eta_{d - d'}$ in $\mathbb{Z}^{d}$ with size at most $(\delta/drM)^{-O(dr)^{O(1)}}$ such that for each $(i, k)$, $\langle w_i, \eta_k \rangle = 0$ and for each $i$ and $j$
    $$\|\eta_i \cdot \alpha_j\|_{\mathbb{R}/\mathbb{Z}}, \|w_i \cdot a_j\|_{\mathbb{R}/\mathbb{Z}} \le \frac{(\delta/drM)^{-O(dr)^{O(1)}}K}{N_j}.$$
\end{itemize}
\end{corollary}
\begin{proof}
We write $\sum_{i, j} \gamma_{ij} n_ih_j = \sum_i n_i \left\{\sum_j \gamma_{ij} h_j \right\} \pmod{1}.$ Define $\tilde{a}_j = (a_j, e_{d+1}, e_{d+2}, \dots, e_{d+r})$ where $e_j$ is the $j$-th elementary basis vector and $\tilde{\alpha}_i = \left(\alpha_i, \gamma_{i}\right)$ where $\gamma_i = \sum_{\ell \in [r]}\gamma_{i\ell} e_{i+d}$. We have for $\delta|\vec{N}|$ many $\vec{h} \in [\vec{N}]$ that
$$\|\beta \cdot \vec{n} + \sum_{j = 1}^r \tilde{a}_j n_j \cdot \{\sum_{i = 1}^r \tilde{\alpha}_i h_i\} \|_{C^\infty[\vec{N}]} \le K.$$
By Lemma~\ref{l:refined2}, we find $w_1, \dots, w_{d'}$ and $\eta_1, \dots, \eta_{d + r - d'}$ of size at most $(\delta/drM)^{-O(d + r)}$ such that
$$\|\eta_i \cdot \tilde{\alpha}_j\|_{\mathbb{R}/\mathbb{Z}}, |w_i \cdot \tilde{a}_j| \le \frac{(\delta/drM)^{-O(d + r)}K}{N_j}$$
for each $j \in [d + r]$. \\\\
We denote $\eta_i = (\eta_i', \eta_i^{d+1}, \dots, \eta_i^{d + r})$ and $w_i = (w_i', w_i^{d+1}, \dots, w_i^{d + r})$. Fixing $j$ (and noting that $w_i^{d + k}e_{d+k} \cdot e_{d+k} \in \mathbb{Z}$), we see that
$$\|\eta_i' \cdot \alpha_j+\sum_k\eta_i^{d+k}\gamma_{ik}\|_{\mathbb{R}/\mathbb{Z}} \le \frac{(\delta/M)^{-O(r)}(dr)^{O(dr + r^2)}K}{N_j}$$
and
$$\|w_i' \cdot a_j\|_{\mathbb{R}/\mathbb{Z}} \le \frac{(\delta/M)^{-O(r)}(dr)^{O(dr + r^2)}K}{N_j}.$$
Consider $u_1, \dots, u_m$ linearly independent which span the orthogonal complement of the $w_i'$'s. It follows that $(u_i, 0)$ lies in the orthogonal complement of the span of the $w_i$'s, so it lies in the span of $\eta_i$'s. By Cramer's rule we may choose $u_1, \dots, u_m$ to be integer vectors with size at most $(\delta/drM)^{-O(dr^{O(1)})}$, so $(u_1, 0), \dots, (u_i, 0)$ is in an $(\delta/drM)^{-O(dr^{O(1)})}$-integer span of the $\eta_i'$. It follows that
$$\|u_i \cdot \alpha_j\|_{\mathbb{R}/\mathbb{Z}} \le \frac{(\delta/drM)^{-O(dr^{O(1)})}K}{N_j}.$$
\end{proof}

\section{Equidistribution for multiparameter nilsequences}\label{s:equidistribute}
We shall recall the main theorem here.

\mainresulta*

We begin with the following lemma.
\begin{lemma}\label{l:commutatorlemma}
Let $G$ be a group and $G_2$ a subgroup of $G$ containing $[G, G]$. Then any $t - 1$-fold commutator containing $t - 1$ copies of $G$ and one copy of $G_2$ lies inside the $t - 1$-fold commutator $[G_2, G, \dots, G]$.
\end{lemma}
\begin{proof}
We claim by induction that for $0 \le r \le m - 1$ that if $g_m \in G_2$ and $g_j \in G$ for all other $j$ that
$$[[g_1, \dots, g_{m - r - 1}], [g_m, g_{m - 1}, \dots, g_{m - r}]] \in [G_2, G, \dots, G]$$
where there are $m - 1$ copies of $G$ in above relation. Here, we would like to remind the reader that as indicated in Definition~\ref{d:lowercentraseries}, $G_{(u)}$ is the $u - 1$-fold commutator $[G, G, \dots, G]$. Our main tool for proving this claim is the Hall-Witt identity, which states that if $x^y = y^{-1}xy = x[x, y]$, then
$$[[x, y], z^x][[z, x], y^z][[y, z], x^y] = \mathrm{id}_G.$$
We proceed by induction on $m$. The cases $m = 1, 2$ are trivial. For fixed $m$, we proceed by backwards induction on $r$. The case of $r = m$ follows straight from the fact that $[g_m, \dots, g_1]$ lies inside $[G_2, \dots, G]$. Now, assuming the case for $\ge r$, we wish to show
$$[[g_1, \dots, g_{m - r}], [g_k, g_{m - 1}, \dots, g_{m - r + 1}]] \in [G_2, G, \dots, G].$$
By the Hall-Witt identity applied with
$$x = [g_1, \dots, g_{m - r - 1}], y = g_{m - r}, \text{ and } z = [g_m, g_{m - 1}, \dots, g_{m - r + 1}],$$
and noticing that $G_{(m + 1)}$ lies in the $m - 1$-fold commutator $[G_2, G, \dots, G]$ so the conjugation is irrelevant, it suffices to show that
$$[[[g_1, \dots, g_{m - r - 1}], [g_m, g_{m - 1}, \dots, g_{m - r + 1}]], g_{m - r}] \in [G_2, G, \dots, G]$$
and
$$[[g_1, \dots, g_{m - r - 1}], [g_m, \dots, g_{m -r}]] \in [G_2, G, \dots, G].$$
Both of these are true by the induction hypothesis and result follows.
\end{proof}
We next define the following two notions.
\begin{definition}[Linear and nonlinear component]\label{d:gling2}
Given a polynomial sequence $g \in \mathrm{poly}(\mathbb{Z}^\ell, G)$ with $g(0) = \mathrm{id}_G$, we define
$$g_{\mathrm{lin}}(\vec{n}) = g(e_1)^{n_1} g(e_2)^{n_2} \cdots g(e_\ell)^{n_\ell} \text{ and } g_2(\vec{n}) = g(\vec{n})g_{\mathrm{lin}}(\vec{n})^{-1}.$$
\end{definition}
\begin{definition}\label{d:joining}
For a group $G$ and a normal subgroup $H$ of $G$, we define
$$G \times_{H} G := \{(g, g') \in G^2: g'g^{-1} \in H\}.$$
\end{definition}

We finally require a lemma which given that allows one to compile information regarding various $C^{\infty}[\vec{N}]$ norms of various shifts into information regarding the polynomials themselves; this is the analog of \cite[Lemma~7.6]{GT12}. 
\begin{lemma}\label{l:polynomialseparation}
Fix $\rho\in (0,1/2)$, an integer $\ell$, and integers $N_1,\ldots,N_{\ell}$. Let $P$ and $Q$ be polynomials in $n_1,\ldots,n_{\ell}$ of degree at most $k$ and $\sigma(h):[\vec{N}]\to \mathbb{R}^{\ell}$ an arbitrary map. 

Suppose that $P(0) = 0$, $Q(\vec{e}_i) = 0$ for $i\in [\ell]$, $Q(0) = 0$, and for at least $\rho N_1\ldots N_{\ell}$ many $\vec{h} \in [\vec{N}]$ we have 
\[\|P(\vec{n}) + Q(\vec{n} + \vec{h}) - Q(\vec{n}) + \sigma(\vec{h}) \cdot \vec{n}\|_{C^{\infty}[\vec{N}]}\le \rho^{-1}.\]

Then there there exists a nonzero integer $q$ with $|q|\le \rho^{-O_{k,\ell}(1)}$ such that for $\rho N_1\ldots N_{\ell}$ shifts $\vec{h} \in [\vec{N}]$ we have
\begin{align*}
\left\|q\left(\sigma(\vec{h}) \cdot \vec{n} + \sum_{i\in [\ell]}P(\vec{e}_i) n_i+ \sum_{i, j \in [\ell]} (\partial_{e_i}\partial_{e_j}Q(0)) n_ih_j\right)\right\|_{C^\infty[\vec{N}]} \le \rho^{-O_{k,\ell}(1)}.
\end{align*}

If furthermore, $\sigma$ is identically zero we have that there exists a nonzero integer $q$ with $|q|\le \rho^{-O_{k,\ell}(1)}$ such that 
\[\|qP(\vec{n})\|_{C^{\infty}[\vec{N}]}\le \rho^{-O_{k,\ell}(1)}\text{ and } \|qQ(\vec{n})\|_{C^{\infty}[\vec{N}]}\le \rho^{-O_{k,\ell}(1)}.\]
\end{lemma}
\begin{proof}
We prove this via a trick. Given a vector $\vec{h}$ and a polynomial $R$, we define 
\[\partial_{\vec{h}}R(\vec{n}) = R(\vec{n}+\vec{h}) - R(\vec{n}).\]
Take $\tau = \rho^{O_{k,\ell}(1)}$ sufficiently small. We have for each $\vec{h} \in [\vec{N}]$ such that 
\[\|P(\vec{n}) + Q(\vec{n} + \vec{h}) - Q(\vec{n}) + \sigma(\vec{h}) \cdot \vec{n}\|_{C^{\infty}[\vec{N}]}\le \rho^{-1}\]
that for all $\vec{h}'\in [\vec{N}]$, 
\[\|\partial_{\vec{h}'}(P(\vec{n}) + Q(\vec{n} + \vec{h}) - Q(\vec{n}))\|_{C^{\infty}[\vec{N}]}=\|\partial_{\vec{h}'}(P(\vec{n}) + Q(\vec{n} + \vec{h}) - Q(\vec{n}) + \sigma(h) \cdot \vec{n})\|_{C^{\infty}[\vec{N}]} \le \tau^{-1/4}\]
by (say) Lemma~\ref{l:multiparameterextrapolation} and the triangle inequality.
This implies that for at least $\rho|\vec{N}|^2$ elements in $(\vec{h}',\vec{h})\in [\vec{N}]\times [\vec{N}]$, we have  
\begin{align*}
&\| (\partial_{\vec{n}}\partial_{\vec{h}'}\partial_{\vec{h}}Q)(0)+ (\partial_{\vec{n}}\partial_{\vec{h}'}P)(0)\|_{C^{\infty}[\vec{N}]}\le \tau^{-1/4}.
\end{align*}

Thus for at least $\rho$ fraction of $(\vec{h},\vec{h}',\vec{n})\in [\vec{N}]\times [\vec{N}]\times \tau \cdot [\vec{N}]$, we have that 
\begin{align*}
&\|(\partial_{\vec{n}}\partial_{\vec{h}'}\partial_{\vec{h}}Q)(0)+ (\partial_{\vec{n}}\partial_{\vec{h}'}P)(0)\|_{\mathbb{R}/\mathbb{Z}}\le \tau^{1/2}.
\end{align*}

This implies via Lemma~\ref{l:vinogradov} that there exists a nonzero integer $q$ such that $|q|\le \rho^{-O_{k,\ell}(1)}$ and 
 \begin{align*}
&\|q((\partial_{\vec{n}}\partial_{\vec{h}'}\partial_{\vec{h}}Q)(0)+ (\partial_{\vec{n}}\partial_{\vec{h}'}P)(0))\|_{C^{\infty}[\vec{N},\vec{N},\vec{N}]}\le \rho^{-O_{k,\ell}(1)}.
\end{align*}

Taking $\vec{h} = \vec{0}$ and $\vec{n} = \vec{h}'$, we have
\begin{align*}
&\|q(P(2\vec{n}) - 2P(\vec{n}) + P(0))\|_{C^{\infty}[\vec{N}]}\le \rho^{-O_{k,\ell}(1)}.
\end{align*}
Let $P(\vec{n}) = \sum_{|\vec{i}|\le k}\alpha_{\vec{i}}\vec{n}^{\vec{i}}= \sum_{|\vec{i}|\le k}\alpha_{\vec{i}}'\binom{\vec{n}}{\vec{i}}$. This implies that there exists $q_1 = C_k \cdot q$ and $q_2 = C_k^{2} \cdot q$ such that 
\[\|q_1\cdot \alpha_{\vec{i}}\|_{\mathbb{R}/\mathbb{Z}} \le \rho^{-O_{k,\ell}(1)}\vec{N}^{-\vec{i}}\]
for $|\vec{i}|\ge 2$ and via Lemma~\ref{l:cinfinity} that
\[\|q_2\cdot \alpha_{\vec{i}}'\|_{\mathbb{R}/\mathbb{Z}} \le \rho^{-O_{k,\ell}(1)}\vec{N}^{-\vec{i}}\]
for $|\vec{i}|\ge 2$. Thus 
 \begin{align*}
&\|q_2((\partial_{\vec{n}}\partial_{\vec{h}'}\partial_{\vec{h}}Q)(0))\|_{C^{\infty}[\vec{N},\vec{N},\vec{N}]}\le \rho^{-O_{k,\ell}(1)}.
\end{align*}
Taking $\vec{n} = \vec{h} = \vec{h}'$, and an analogous argument to above, there exists a nonzero integer $|q_3| \le \rho^{-O_{k,\ell}(1)}$ such that if $Q(\vec{n}) = \sum_{|\vec{i}|\le k}\beta_{\vec{i}}\binom{\vec{n}}{\vec{i}}$ then 
\[\|q_3\cdot \beta_{\vec{i}}\|_{\mathbb{R}/\mathbb{Z}} \le \rho^{-O_{k,\ell}(1)}\vec{N}^{-\vec{i}}\]
for $|\vec{i}|\ge 3$ and $|\vec{i}| = 1$ (using the vanishing conditions on $Q$). We thus have that for at least $\rho|\vec{N}|$ values of $\vec{h}\in [\vec{N}]$ that 
\[\left\|q_3\left(P(\vec{n}) + Q(\vec{n} + \vec{h}) - Q(\vec{n}) + \sigma(\vec{h}) \cdot \vec{n}\right)\right\|_{C^\infty[\vec{N}]}\le \rho^{-O_{k,\ell}(1)}.\]
For such $\vec{h}$ and plugging in the conditions obtained, we have that 
\begin{align*}
&\bigg\|q_3\bigg(\sum_{i\in[\ell]}P(\vec{e}_i)n_i + \sum_{1\le i<j\le \ell}(\partial_{e_i}\partial_{e_j}Q)(0)((n_i+h_i)(n_j+h_j)-n_in_j)\\
&\qquad+ \sum_{i\in \ell}(\partial_{e_i}\partial_{e_i}Q)(0)\bigg(\binom{n_i+h_i}{2}-\binom{n_i}{2}\bigg) +  \sigma(\vec{h}) \cdot \vec{n}\bigg)\bigg\|_{C^\infty[\vec{N}]}\le \rho^{-O_{k,\ell}(1)}
\end{align*}
and simplifying gives the first part of the conclusion.

When $\sigma(\vec{h}) = \vec{0}$, we have that for $\rho |\vec{N}|$ many $\vec{h} \in [\vec{N}]$ has
\begin{align*}
\left\|q_3\left(\sum_{i\in [\ell]}P(\vec{e}_i) n_i+ \sum_{i, j \in [\ell]}(\partial_{e_i}\partial_{e_j}Q)(0) n_ih_j\right)\right\|_{C^\infty[\vec{N}]} \le \rho^{-O_{k,\ell}(1)}.
\end{align*}
This implies taking $\tau_2 = \rho^{-O_{k,\ell}(1)}$ that at least a $\rho$-fraction of $(\vec{n},\vec{h})\in [\vec{N}]\times \tau_2\cdot [\vec{N}]$ that 
\begin{align*}
\left\|q_3\left(\sum_{i\in [\ell]}P(\vec{e}_i) n_i+ \sum_{i, j \in [\ell]} (\partial_{e_i}\partial_{e_j}Q)(0) n_ih_j\right)\right\|_{\mathbb{R}/\mathbb{Z}} \le\tau_2^{1/2}.
\end{align*}
Via Lemma~\ref{l:vinogradov}, we then have that there exists a nonzero integer $q_4$ with $|q_4|\le \rho^{-O_{k,\ell}(1)}$ such that 
\[\left\|q_4P(\vec{e}_i)\right\|_{\mathbb{R}/\mathbb{Z}}\le \rho^{-O_{k,\ell}(1)}N_i^{-1}\text{ and } \left\|q_4(\partial_{e_i}\partial_{e_j}Q)(0)  \right\|_{\mathbb{R}/\mathbb{Z}}\le \rho^{-O_{k,\ell}(1)}(N_iN_j)^{-1}.\]
This gives the remaining coefficients to imply the final part of the lemma.    
\end{proof}

The proof of Theorem~\ref{t:mainresult1} will eventually split into whether or not the $s - 1$-fold commutator
$$\xi([G_2, G, \dots, G]) = 0$$
holds. For more motivation and simpler cases of this argument, we encourage the reader to consult \cite{LenNew}. We now turn to the proof of Theorem~\ref{t:mainresult1}.
\subsection{Proof of Theorem~\ref{t:mainresult1}}
If $s = 1$, then we may write $F(g(\vec{n})\Gamma)$ as $e(\xi(g(\vec{n})))$. A Weyl-type exponential sum estimate (e.g., \cite[Proposition 7]{Tao12}) immediately gives us that there exists some integer $q \le c(\delta)^{-1}$ such that
$$\|q \xi \circ g\|_{C^\infty[\vec{N}]} \le c(\delta)^{-1}.$$
Let $\eta = q\xi$ be the horizontal character. We see that $|\eta| \le c(\delta)^{-1}$ and defining $G' = \mathrm{ker}(\eta)$, we see that if $g \in G'$, then $\xi(g) = 0$. This proves the result for $s = 1$. \\\\
We now assume that $s > 1$. We make a set of preliminary reductions. We first reduce via Lemma~\ref{l:trick1} to the case of when $g(0) = \mathrm{id}_G$ and $|\psi(g(e_i))| \le \frac{1}{2}$ for all $i$ and apply Lemma~\ref{l:trick2} to quotient out by the kernel of $\xi$ to have a one-dimensional vertical torus. We next make a slight modification to the filtration: we replace $G_\ell$ for $\ell \ge s$ by $G_\ell G_{(s)}$. This preserves normality and the filtration property since $G_{(s)}$ lies in the center of $G$. By Lemma~\ref{l:ConstructingMalcev}, there exists a complexity $M^{O_k(d^{O(1)})}$ Mal'cev basis adapted to this filtration. \\\\
We now apply Lemma~\ref{l:nilcharacters} to Fourier expand
$$F = \sum_{|\zeta| \le c(\delta)^{-1}} F_\zeta + O(c(\delta))$$
with $F_\zeta$ being a $G_k$-character of frequency $\zeta$. Integrating the expression
$$\int_{G_{(s)}/G_{(s)} \cap \Gamma} F(g_{(s)}x\Gamma) e(-\xi(g_{(s)})) dg_{(s)} = \sum_{|\zeta| \le c(\delta)^{-1}} \int_{G_{(s)}/G_{(s)} \cap \Gamma} F_\zeta(g_{(s)}x\Gamma) e(-\xi(g_{(s)})) dg_{(s)} + O(c(\delta))$$
it follows that we may assume that $(\zeta - \xi)(G_{(s)}) = 0$. By the Pigeonhole principle, we may assume that $F = F_\zeta$ for such a $\zeta$ and then apply van der Corput's inequality to obtain
$$|\mathbb{E}_{\vec{n} \in [\vec{N}]} F(g(\vec{n})\Gamma)\overline{F(g(\vec{n} + \vec{h})\Gamma)}| \ge c(\delta)$$
for $c(\delta)|\vec{N}|$ many $\vec{h} \in [\vec{N}]$.\\\\
Note that the Taylor coefficients of $g_2(\vec{n})$ are of the form $g_{\vec{j}}^{\binom{\vec{n}}{\vec{j}}}$ with $|j| \ge 2$ and $g_{\vec{j}} \in G_{|\vec{j}|}$. We may rewrite the above as
$$|\mathbb{E}_{\vec{n} \in [\vec{N}]} F_{\vec{h}}(g_{\vec{h}}(\vec{n})\Gamma^2)| \ge c(\delta)$$
where
$$F_{\vec{h}}(x, y) = F(\{g_{\mathrm{lin}}(\vec{h})\} x)\overline{F(y)}$$
$$g_{\vec{h}}(\vec{n}) := (\{g_{\mathrm{lin}}(\vec{h})\}^{-1}g_2(\vec{n} + \vec{h})g_{\mathrm{lin}}(\vec{n} + \vec{h})[g_{\mathrm{lin}}(\vec{h})]^{-1}, g(\vec{n})).$$
Define $G^\square := G \times_{G_2} G$. By Lemma~\ref{l:annoyingcomplexitybounds}, this is an $\le s$-step nilpotent Lie group and has a natural filtration of $(G^\square)_i = G_i \times_{G_{i + 1}} G_i$ and by Lemma~\ref{l:vandercorputpolynomial}, $g_{\vec{h}}$ is a polynomial sequence on $G^\square$ with respect to this filtration. Since $F$ has $G_k$-frequency $\zeta$, and since $G_k$ lies in the center of $G$ and $F_{\vec{h}}$ is invariant under $G_k^\triangle :=\{(g_k,g_k) : g_k\in G_k\}$, it follows that $F_{\vec{h}}$ descends to a function on $\overline{G^\square} := G^\square/G_k^\triangle$. By Lemma~\ref{l:annoyingcomplexitybounds}, it follows that $G^\square$ and $\overline{G^\square}$ have complexity at most $M^{O_{k}(d^{O(1)})}$.\\\\
Now by Lemma~\ref{l:annoyingcomplexitybounds}, given a horizontal character $\eta\colon G^\square \to \mathbb{R}$, we may decompose it as $\eta(g_2g, g) = \eta_1(g) + \eta_2(g_2)$ where $g \in G$ and $g_2 \in G_2$ where $\eta_1$ is a horizontal character on $G$ and $\eta_2$ is a horizontal character on $G_2$ which annihilates $[G, G_2]$. The same lemma tells us that if $\eta$ has size at most $c(\delta)^{-1}$, then $\eta_1$ and $\eta_2$ have size at most $c(\delta)^{-1}$. We now compute $\eta \circ g_{\vec{h}}$, following \cite[p. 13-14]{GT14}. Given a horizontal character $\eta$ on $G^\square$, we have
$$\eta(g_{\vec{h}}(\vec{n})) = \eta_1(g(\vec{n})) + \eta_2(\{g_{\mathrm{lin}}(\vec{h})\}^{-1}g_2(\vec{n} + \vec{h})g_{\mathrm{lin}}(\vec{n} + \vec{h})[g_{\mathrm{lin}}(\vec{h})]^{-1}g(\vec{n})^{-1}).$$
As $\eta_2$ vanishes on $[G, G_2]$, the above is equal to
$$\eta_1(g(\vec{n})) + \eta_2(g_2(\vec{n} + \vec{h})) - \eta_2(g_2(\vec{n})) + \eta_2(\{g_{\mathrm{lin}}(\vec{h})\}^{-1}g_{\mathrm{lin}}(\vec{n} + \vec{h})[g_{\mathrm{lin}}(\vec{h})]^{-1}g_{\mathrm{lin}}(\vec{n})^{-1}).$$
We may write $[g_{\mathrm{lin}}(\vec{h})]^{-1} = g_{\mathrm{lin}}(\vec{h})^{-1}\{g_{\mathrm{lin}}(\vec{h})\}$, so
\begin{align*}
\eta_2&(\{g_{\mathrm{lin}}(\vec{h})\}^{-1}g_{\mathrm{lin}}(\vec{n} + \vec{h})[g_{\mathrm{lin}}(\vec{h})]^{-1}g_{\mathrm{lin}}(\vec{n})^{-1}) \\
&= \eta_2(\{g_{\mathrm{lin}}(\vec{h})\}^{-1}g_{\mathrm{lin}}(\vec{n} + \vec{h})g_{\mathrm{lin}}(\vec{h})^{-1}\{g_{\mathrm{lin}}(\vec{h})\}g_{\mathrm{lin}}(\vec{n})^{-1}) \\
&= \eta_2([\{g_{\mathrm{lin}}(\vec{h})\}, g_{\mathrm{lin}}(\vec{h})g_{\mathrm{lin}}(\vec{n} + \vec{h})^{-1}] g_{\mathrm{lin}}(\vec{n} + \vec{h}) g_{\mathrm{lin}}(\vec{h})^{-1} g_{\mathrm{lin}}(\vec{n})^{-1}) \\
&= \eta_2([\{g_{\mathrm{lin}}(\vec{h})\}, g_{\mathrm{lin}}(\vec{h})g_{\mathrm{lin}}(\vec{n} + \vec{h})^{-1}]) + \eta_2(g_{\mathrm{lin}}(\vec{n} + \vec{h})g_{\mathrm{lin}}(\vec{h})^{-1}g_{\mathrm{lin}}(\vec{n})^{-1})
\end{align*}
By Baker-Campbell-Hausdorff, we see that
$$g_{\mathrm{lin}}(\vec{n} + \vec{h})g_{\mathrm{lin}}(\vec{h})^{-1}g_{\mathrm{lin}}(\vec{n})^{-1} = \prod_{i < j} [g(e_i)^{h_i}, g(e_j)^{n_j}] \pmod{[G, G, G]}$$
$$[\{g_{\mathrm{lin}}(\vec{h})\}, g_{\mathrm{lin}}(\vec{h})g_{\mathrm{lin}}(\vec{n} + \vec{h})^{-1}] = \prod_{i = 1}^\ell [\{g_{\mathrm{lin}}(\vec{h})\}, g(e_i)^{-n_i}] \pmod{[G, G,G]}.$$
Hence,
\begin{align*}
\eta(g_{\vec{h}}(\vec{n})) &= \eta_1(g(\vec{n})) + \eta_2(g_2(\vec{n} + \vec{h})) - \eta_2(g_2(\vec{n})) \\
&+\sum_{i = 1}^\ell n_i\eta_2([g(e_i), \{g_{\mathrm{lin}}(\vec{h})\}]) + \sum_{1 \le i< j \le \ell} h_i n_j\eta_2([g(e_i), g(e_j)]).
\end{align*}
We will now proceed to Case 1 of the argument.
\subsection{Case 1}
Suppose first that the $s - 1$-fold commutator satisfies
$$\xi([G_2, G, G, \cdots, G]) = 0.$$
As we reduced to the case where $G_{(s)}$ is one-dimensional, $[G_2, G, G, \cdots, G] = 0$. Applying induction for our argument requires an understanding of vertical frequencies on $\overline{G^\square}$, which requires understanding the various $t - 1$-fold commutators $[\overline{G^\square}, \dots, \overline{G^{\square}}]$. First note that since $\xi([G_2, G, G, \cdots, G]) = 0$, it follows from Lemma~\ref{l:commutatorlemma} and $G_{(s - 1)}^\triangle$ not being central in $G^\square$ that $G_{(s - 1)}^\triangle/G_k^\triangle \neq \mathrm{Id}_{\overline{G^\square}}$ so $\overline{G^\square}$ is in fact $s - 1$-step. Thus, $F_{\vec{h}}$ descends to an $s - 1$-step nilsequence $\overline{F_h}$ on $\overline{G^\square}$. By Lemma~\ref{l:commutatorlemma}, the vertical component of $\overline{G^\square}$ is $G_{(s - 1)} \times_{K} G_{(s - 1)}/G_{k}^\triangle$
where $K$ is the $s - 2$-fold commutator $K = [G_2, G, G, \cdots, G]$. This group contains $G_{(s)}^2/G_k^\triangle$. We define the homomorphism $\xi \otimes \overline{\xi} \colon G_{(s)}^2 \to \mathbb{R}$ via $\xi \otimes \overline{\xi}(g, g') := \xi(g) - \xi(g')$; since $\xi \otimes \overline{\xi}$ annihilates $G_k^\triangle \cap G_{(s)}^2 = G_{(s)}^\triangle$, (and abusing notation) this descends to a homomorphism $\xi \otimes \overline{\xi} \colon G_{(s)}^2/G_k^\triangle \to \mathbb{R}$. In addition, note that $\xi \otimes\overline{\xi}((\Gamma \cap G_{(s)})^2/G_k^\triangle) \subseteq \mathbb{Z}$. \\\\
Observe that $\overline{F_{\vec{h}}}$ is a $G_{(s)}^2/G_k^\triangle$-vertical character with frequency $\xi \otimes \overline{\xi}$. Since $\xi$ plays a crucial role in the conclusion of Theorem~\ref{t:mainresult1}, we shall wish to preserve this property even after we Fourier expand $\overline{F_{\vec{h}}}$ into vertical characters. We do so as follows: Fourier expanding $F_{\vec{h}}$ via Lemma~\ref{l:nilcharacters}, we obtain
$$\overline{F_{\vec{h}}} = \sum_{|\alpha| \le c(\delta)^{-1}} F_{\vec{h}, \alpha} + O(c(\delta))$$
where $F_{\vec{h}, \alpha}$ are vertical characters of frequency $\alpha$. Now integrating both sides against $\xi \otimes \overline{\xi}$, we obtain
$$\int_{G_{(s)}^2/G_{k}^\triangle \Gamma} \overline{F_{\vec{h}}}(g_{(s)}x) e(-\xi \otimes \overline{\xi}(g_{(s)})) dg_{(s)} = \sum_{|\alpha| \le c(\delta)^{-1}} \int_{G_{(s)}^2/G_{k}^\triangle \Gamma} F_{\vec{h}, \alpha}(g_{(s)}x) e(-\xi \otimes \overline{\xi}(g_{(s)})) dg_{(s)} + O(c(\delta))$$
We rewrite this as
$$\overline{F_{\vec{h}}} = \sum_{|\beta| \le c(\delta)^{-1}} F_{\vec{h}, \beta} + O(c(\delta))$$
where $F_{\vec{h}, \beta}$ has frequency $\beta$. The upshot is that $F_{\vec{h},\beta}$ has frequency $\xi \otimes \overline{\xi}$ on $G_{(s)}^2/G_k^\triangle$, so (abusively) lifting $\beta$ to $G_{(s - 1)} \times_K G_{(s - 1)}$, $(\beta - \xi \otimes \overline{\xi})(G_{(s)}^2) = 0$. Thus, by the Pigeonhole principle, we may assume that $\overline{F_{\vec{h}}} = F_{\vec{h}, \beta}$ and by Pigeonholing in $h$, we may assume that $\beta$ is $h$-independent. \\\\
Let $\overline{g_{\vec{h}}}$ denote the projection of $g_{\vec{h}}$ to $\overline{G^\square}$. By our induction hypothesis and Pigeonholing in $\vec{h}$, there exists horizontal characters (independent of $\vec{h}$) $\eta_1, \dots, \eta_t$ such that
$$\|\eta_i \circ \overline{g_{\vec{h}}}\|_{C^\infty[\vec{N}]} \le c(\delta)^{-1}$$
and that the subgroup $H := \bigcap_{i = 1}^t \mathrm{ker}(\eta_i) \subseteq \overline{G^\square}$ satisfies
$\beta([H, H, \dots, H]) = 0$ where there are $s - 1$ copies of $H$. Abusing notation, we may lift $\eta_i$ to a horizontal character on $G^\square$ which annihilates $G_k^\triangle$ so that
$$\|\eta_i \circ g_{\vec{h}}\|_{C^\infty[\vec{N}]} \le c(\delta)^{-1}$$
Replacing $H$ with $HG_k^\triangle$, we see since $\eta_i$ annihilates $G_k^\triangle$ that (abusing notation) $H = \bigcap_{i = 1}^t \mathrm{ker}(\eta_i) \subseteq G^\square$ and that $\beta([H, H, \dots, H]) = 0$. Here, we would like to emphasize for the sake of clarity that as written, $\beta$ is a character on $G_{(s - 1)} \times_K G_{(s - 1)}$ that annihilates $G_k^\triangle$. Thus, if $\overline{a} \in \overline{G^\square}$, $\beta(\overline{a})$ is defined by lifting $\overline{a}$ to $a$ in $G_{(s - 1)} \times_{K} G_{(s - 1)}$ and then evaluating $\beta$ on $a$. In addition, if $\overline{a}_1, \dots, \overline{a}_{s - 1} \in H$ lifts to $a_1, \dots, a_{s - 1}$ in $G_{(s - 1)} \times_K G_{(s - 1)}$, then $[\overline{a}_1, \dots, \overline{a}_{s - 1}] \equiv [a_1, \dots, a_{s - 1}] \pmod{G_k^\triangle}$.\\\\
By Lemma~\ref{l:annoyingcomplexitybounds}, we may write
$\eta_i(g', g) = \eta_i^1(g) + \eta_i^2(g'g^{-1})$ where $\eta_i^1$ is a horizontal character on $G$ which annihilates $G_k$ and $\eta_i^2$ is a horizontal character on $G_2$ which annihilates $[G_2, G]$; by Lemma~\ref{l:annoyingcomplexitybounds}, both horizontal characters are size at most $c(\delta)^{-1}$. We have from the above procedure that
\begin{align*}
\eta_i(g(\vec{n})) &= \eta_i^1(g(\vec{n})) + \eta_i^2(g_2(\vec{n} + \vec{h})) - \eta_i^2(g_2(\vec{n}))  \\
&+\sum_{j} n_j\eta_i^2([g(e_j), \{g_{\mathrm{lin}}(\vec{h})\}]) + \sum_{1 \le i< j \le \ell} h_i n_j\eta_i^2([g(e_i), g(e_j)]).
\end{align*}
Let
$$P(\vec{n}) = \eta_i^1(g(\vec{n})), \text{ } Q(\vec{n}) = \eta_i^2(g_2(\vec{n}))$$
and $\sigma\colon [\vec{N}]^\ell \to \mathbb{R}^\ell$ be defined via
$$\sigma(\vec{h})\cdot \vec{n} = \sum_{j} n_j\eta_i^2([g(e_j), \{g_{\mathrm{lin}}(\vec{h})\}]) + \sum_{1 \le i< j \le \ell} h_i n_j\eta_i^2([g(e_i), g(e_j)]).$$
Applying Lemma~\ref{l:polynomialseparation} (and possibly scaling $\eta_i$ by a nonzero integer bounded by $c(\delta)^{-1}$) with this choice of $P, Q, \sigma$ gives for some $\gamma = (\gamma_{ij})_{i, j = 1}^\ell \in \mathbb{R}^{\ell^2}$ and $\beta \in \mathbb{R}^\ell$ that
$$\left\|\beta \cdot \vec{n} + \sum_{j} n_j\eta_i^2([g(e_j), \{g_{\mathrm{lin}}(\vec{h})\}]) + \sum_{i, j \in [\ell]} \gamma_{ij} n_ih_j\right\|_{C^\infty[\vec{N}]} \le c(\delta)^{-1}.$$
Let $\alpha$ be the first $d - d_2$ coordinates $\psi(g_{\mathrm{lin}}(\vec{h}))$ and $a$ the unique vector in $(\mathbb{R}^{d - d_2})^\ell$ that represents $(\eta_i^2([g(e_j),\cdot]))_{j = 1}^\ell$. By Lemma~\ref{l:multiplication} and Lemma~\ref{l:transition}, we see that $|a| \le c(\delta)^{-1}$. In addition, note that $\eta([\cdot, \cdot])$ descends to a map which only depends on the first $d - d_2$ Mal'cev coordinates of its arguments. Thus, we may apply Corollary~\ref{c:bracketcorollary2} to this choice of $a$, $\alpha$, $\beta$, and $\gamma$, obtaining linearly independent horizontal characters $x_{1, i}, \dots, x_{r_i, i}$ which annihilate $G_2$ and orthogonal elements to each of the independent horizontal characters $y_{1, i}, \dots, y_{d - d_{2} - r_i, i} \in \Gamma/(G_2\cap \Gamma)$ all of size at most $c(\delta)^{-1}$ such that for each $i, j$,
$$\|x_{j, i} \circ g\|_{C^\infty[\vec{N}]} \le c(\delta)^{-1}$$
$$\|\eta_i^2([y_{j, i}, g])\|_{C^\infty[\vec{N}]} \le c(\delta)^{-1}.$$

Let $\pi\colon \overline{G^\square} \to G/G_2$ be the projection map $(g', g) \mapsto g \pmod{G_2}$ and $V = \pi(H)$. Then since $H$ is $c(\delta)^{-1}$-rational, $V$ can be spanned by $c(\delta)^{-1}$-rational vectors with respect to the basis $(\bar{X}_1, \dots, \bar{X}_{d - d_2})$ (with $X \mapsto \bar{X}$ being the map from $\mathfrak{g} \to \mathfrak{g}/\mathfrak{g}_2$). By Cramer's rule (Lemma~\ref{l:Cramer} and Lemma~\ref{l:Cramer2}), there exists linearly independent horizontal characters $z_1, \dots, z_u$ of size at most $c(\delta)^{-1}$ that annihilate $G_2$ such that $V = \{x \in G_1/G_2: z_i(x) = 0 \text{ for all } i \in [u]\}$. Each $z_i$ lifts to a horizontal character $\tilde{z_i}$ of size at most $c(\delta)^{-1}$ on $G^\square$ that annihilates $G_2^2$ via $\tilde{z_i}(g', g) = z_i(g)$. Note that for these lifted characters that $\tilde{z_i}(H) = z_i(\pi(H)) = 0$. Thus $\tilde{z_i}$ is a $c(\delta)^{-1}$-rational combination of the $\eta_j$'s (as $H$ is defined as the joint kernel of $\eta_j$'s) and hence possibly scaling $z_i$ by an appropriate integer $q \le c(\delta)^{-1}$,
$$\|z_i \circ g\|_{C^\infty[\vec{N}]} \le c(\delta)^{-1}.$$
We now let 
$$H' = \{g \in G: x_{j, i} \circ g = 0 , \eta_i^2([y_{j, i}, g]) = 0, z_m \circ g = 0 \text{ for all } i, j, m\}.$$
Then $H'$ is a $c(\delta)^{-1}$-rational subgroup of $G$, and invoking Lemma~\ref{l:factorization1} on $g$, we may write $g = \eps g_1\gamma$ with $\eps$ $(c(\delta)^{-1}, \vec{N})$-smooth, $\gamma$ $c(\delta)^{-1}$-rational, and $g_1$ a polynomial sequence on $H'$ with the filtration $H_i' = G_i \cap H'$. It remains to check that the $s - 1$-fold commutator
$$\xi([H', \dots, H']) = 0.$$

Let $W = H'/G_2$. By Lemma~\ref{l:commutatorlemma}, we have that $\xi([W, W, \dots, W])$ is well-defined and 
$$\xi([H', \dots, H']) = \xi([W, W, \dots, W]).$$ 
So it suffices to show that $\xi([W, W, \dots, W]) = 0$. Let $a_1, \dots, a_s \in H'$ and denote $\overline{a}_i = a_i \pmod{G_2}$. First note that $(\mathrm{id}_G,[a_{1}, a_2])$ lies inside $H$. This is because 
$$\eta_i((\mathrm{id}_G, [a_1, a_2])) = \eta_i^1([a_1, a_2]) - \eta_i^2([a_1, a_2]) = -\eta_i^2([a_{1}, a_2]) = 0$$
for each $i$. To see this, note that $\overline{a_1}$ can be spanned by the elements $y_{j, i} \pmod{G_2}$, since it is orthogonal to all $x_{j,i}$, and then we use $\eta_i^2([y_{j,i},a_{2}]) = 0$.

Next note, denoting $\overline{a_i}^\triangle = (\overline{a_i}, \overline{a_i})$, that $\overline{a_3}^\triangle, \dots, \overline{a_{s}}^\triangle$ lie inside $H/G_2^2$. This is because $a_i$ lies inside $V$. Thus, there exists elements $\tilde{a_i} \in H$ such that $\tilde{a_i} = a_i^\triangle \pmod{G_2^2}$, and 
$$\beta([(\mathrm{id}_G, [a_1, a_2]), \tilde{a}_3, \dots, \tilde{a}_{s - 2}]) = 0.$$
By Lemma~\ref{l:commutatorlemma}, we see that the commutator $[(\mathrm{id}_G, [a_1, a_2]), a_3^\triangle, a_4^\triangle, \dots, a_{s}^\triangle]$ is invariant under multiplying $a_i^\triangle$ by an element of $G_2^2$. Since $\tilde{a}_i \equiv a_i^\triangle \pmod{G_2^2}$, we have that
$$\beta([(\mathrm{id}_G, [a_1, a_2]), [\tilde{a}_3, \dots, \tilde{a}_{s}]]) = \beta((\mathrm{id}_G, [a_1, \dots, a_s])) = -\xi([a_1, \dots, a_s])$$
where in the last equality we used that $(\beta - \xi \otimes \overline{\xi})(G_{(s)}^2) = 0$. Hence, $\xi([a_1, \dots, a_s]) = 0$, completing the proof that $\xi([W, W, \dots, W]) = 0$ and Case 1 of the argument.
\subsection{Case 2}
Now suppose that we have
$$\xi([G_2, G, G, \dots, G]) \neq 0.$$
Our goal is to reduce to Case 1. To be precise, the induction on $s$ proceeds by handling first Case 1 and only then Case 2. In this case, while $\overline{G^\square}$ is no longer $s - 1$-step, $F_{\vec{h}}$ is invariant under the action of $(G \times_{G_2} G)_k = G_k^\triangle$ and $F_{\vec{h}}(g_{\vec{h}}(n)\Gamma)$ descends to a degree $k - 1$ nilsequence on $\overline{G^\square}/\overline{\Gamma^\square}$. Here, we define $\Gamma^\square = \Gamma \times_{G_2} \Gamma$ and $\overline{\Gamma^\square} = \Gamma \times_{G_2} \Gamma/G_k^\triangle$. By Lemma~\ref{l:commutatorlemma}, the vertical component of $\overline{G^\square}$ is $(G_{(s)} \times_{K} G_{(s)})/G_k^\triangle$
where $K$ is the $s - 1$-fold commutator $[G_2, G, \dots, G]$, so $F_{\vec{h}}$ is a vertical character on $G^\square$ of frequency $\xi \otimes \overline{\xi}$ where we define $\xi \otimes \overline{\xi}(g, g') = \xi(g) - \xi(g')$; this horizontal character on $G_{(s)} \times_K G_{(s)}$ descends to a character on $(G_{(s)} \times_K G_{(s)})/G_k^\triangle$ since $\xi \otimes \overline{\xi}(G^\triangle_{(s)}) = 0$. \\\\
Now let 
$$\overline{g_{\vec{h}}} = g_{\vec{h}} \pmod{G_k^\triangle}.$$
By induction on $k$ and pigeonholing in $\vec{h}$, and applying Lemma~\ref{l:factorization1}, we have
$$\overline{g_{\vec{h}}} = \varepsilon_{\vec{h}} g_{\vec{h}}' \gamma_{\vec{h}}$$
where $\varepsilon_{\vec{h}}$ is $(c(\delta)^{-1}, \vec{N})$-smooth, $\gamma_{\vec{h}}$ is $c(\delta)^{-1}$-rational, and $g_{\vec{h}}'$ lies in a subgroup $H$ of $\overline{G^\square}$ with the $s - 1$-fold commutator $\xi \otimes \overline{\xi}([H, H, \dots, H]) = 0$.\footnote{We would like to emphasize once again that as written, for $\overline{a} \in G_{(s)} \times_K G_{(s)}/G_k^\triangle$, $\xi \otimes \overline{\xi}(\overline{a})$ is evaluated by first lifting $\overline{a}$ to some $a \in G_{(s)} \times_K G_{(s)}$. In addition, if $\overline{a_1}, \dots, \overline{a_s} \in \overline{G^\square}$ lifts to $a_1, \dots, a_s \in G^\square$, then $[\overline{a}_1, \dots, \overline{a}_s] = [a_1, \dots, a_s] \pmod{G_k^\triangle}$.} Note here that we may Pigeonhole in $\vec{h}$ to make $H$ independent of $\vec{h}$. Pigeonholing $\vec{h}$ in the rationality of $\gamma_{\vec{h}}$, we may also assume that each $\gamma_{\vec{h}}$ is a fixed rationality of height at most $c(\delta)^{-1}$. \\\\
Let $\pi\colon \overline{G^\square} \to \overline{G^\square}/[G, G]^2 = G^\square/[G, G]^2G_k^\triangle$ denote the projection map and let $V = \pi(H)$. Since $H$ is $c(\delta)^{-1}$-rational, it follows that $V$ is $c(\delta)^{-1}$-rational with respect to a spanning subset of the projection of the Mal'cev basis of $G^\square/[G, G]^2G_k^\triangle$. Hence, by Cramer's rule (Lemma~\ref{l:Cramer} and Lemma~\ref{l:Cramer2}), there exists linearly independent horizontal characters $\eta_1, \dots, \eta_t$ on $G^\square$ which annihilate $[G, G]^2G_k^\triangle$ and are size at most $c(\delta)^{-1}$ such that
$$V = \{g \in G^\square \pmod{[G, G]^2G_k^\triangle}: \eta_i(g) = 0 \text{ for all } i \in [t]\}.$$
Here, we are abusively descending $\eta_i$ to $G^\square/[G, G]^2G_k^\triangle$. We now define
$$W = \{g \in G^\square \pmod{[G, G]^2}: \eta_i(g) = 0 \text{ for all } i \in [t]\}.$$
A key claim is that the $s - 1$-fold commutator 
$$\xi \otimes \overline{\xi}([W, W, \dots, W])$$ 
vanishes. To prove this, note that each coordinate of the commutator $[\cdot,\cdot, \dots, \cdot]$ is invariant under multiplying by an element of $G_k^\triangle$. Thus, this repeated commutator map descends to a map on $W/G_k^\triangle$. However, each element $g \in W/G_k^\triangle$ satisfies $\eta_i(g) = 0$, so $\pi(g)$ belongs to $V$. As $\xi \otimes \overline{\xi}([V, V, \dots, V]) = 0$, it follows that $\xi \otimes \overline{\xi}([W, W, W, \dots, W]) = 0$.

Now, by scaling each $\eta_i$ by some integer $q \le c(\delta)^{-1}$, we have 
$$\|\eta_i \circ g_{\vec{h}}\|_{C^\infty[\vec{N}]} \le c(\delta)^{-1}.$$
Decomposing $\eta_i(g', g) = \eta_i^1(g) + \eta_i^2(g'g^{-1})$, and using the fact that $\eta_i$ annihilates $[G, G]^2$ we see  using the earlier expansion of $\eta_i(g_{\vec{h}})$ that
$$\eta_i(g_{\vec{h}}(\vec{n})) = \eta_i^1(g(\vec{n})) + \eta_i^2(g_2(\vec{n} + \vec{h})) - \eta_i^2(g_2(\vec{n})).$$

Thus, applying Lemma~\ref{l:polynomialseparation} to $\eta_i^1 \circ g = P$, $\eta_i^2 \circ g_2 = Q$, and $\sigma$ identically zero for each $i$, we have (possibly scaling $\eta_i$ by a nonzero integer bounded by $c(\delta)^{-1}$) that
$$\|\eta_i^1 \circ g\|_{C^\infty[\vec{N}]}, \|\eta_i^2 \circ g_2\|_{C^\infty[\vec{N}]} \le c(\delta)^{-1}.$$
We now define
\begin{align*}
\tilde{G} &= \{g \in G: \eta_i^1(g) = 0 \text{ for all } 1 \le i \le t\} \\
\tilde{G}_2 &= \{g \in \tilde{G} \cap G_2: \eta_i^2(g) = 0 \text{ for all } 1 \le i \le t\}
\end{align*}
First note that $[\tilde{G}, \tilde{G}] \subseteq \tilde{G}_2$. This is because $\eta_i^2$ vanishes on $[G, G]$ so it vanishes on $[\tilde{G}, \tilde{G}]$. We claim that the following $s - 1$-fold commutator identity holds: 
$$\xi([\tilde{G_2}, \tilde{G}, \dots, \tilde{G}]) = 0.$$
It suffices to show that for any $y \in \tilde{G}_2$ and $x_1, \dots, x_{s - 1}$ be inside $\tilde{G}$, 
$$\xi([y, x_1, \dots, x_{s - 1}]) = 0.$$
Let $\pi'\colon G \to G/[G, G]$. Note that both $(\pi'(y), \mathrm{id}_{G/[G,G]})$ and $(\pi'(x_j), \pi'(x_j))$ lie inside $W$. This is because for each $i$, $\eta_i((\pi'(y), \mathrm{id}_{G/[G,G]}])) = \eta_i((\pi'(x_j), \pi'(x_j))) = 0$ since $\eta_i((\pi'(y), \mathrm{id}_{G/[G, G]})) = \eta_i^2(\pi'(y)) = 0$ and $\eta_i(\pi'(x_j), \pi'(x_j)) = \eta_i^1(\pi'(x_j)) = 0$. Thus as $\xi \otimes \overline{\xi}([W, W, \dots, W]) = 0$, we have
\begin{align*}
0 &= \xi \otimes \overline{\xi}([(\pi'(y), \mathrm{id}_{G/[G,G]}), (\pi'(x_1), \pi'(x_1)), \dots, (\pi'(x_{s - 1}), \pi'(x_{s - 1}))]) \\
&= \xi \otimes \overline{\xi}([(y, \mathrm{id}_G), (x_1, x_1), \dots, (x_{s - 1}, x_{s - 1})]) \\
&= \xi([y, x_1, \dots, x_{s - 1}]).
\end{align*}

Now using Lemma~\ref{l:factorization1}, we may write $g(\vec{n}) = \varepsilon'(\vec{n})g'(\vec{n})\gamma'(\vec{n})$ where $\varepsilon'$ is $(c(\delta)^{-1}, \vec{N})$-smooth, $\gamma'$ is $c(\delta)^{-1}$-rational, and $g' \in \mathrm{poly}(\mathbb{Z}^\ell, G')$ with the filtration $(G')_i = \tilde{G} \cap G_i$. We thus have $g'(\vec{n}) = \varepsilon'(\vec{n)}^{-1}g(\vec{n})\gamma'(\vec{n})^{-1}$. Since $\pi_{\mathrm{horiz}}$ is a homomorphism and $G/G_{(2)}$ is abelian, we have
$$g_{\mathrm{lin}}'(\vec{n}) \equiv \varepsilon_{\mathrm{lin}}'(\vec{n})^{-1} g_{\mathrm{lin}}(\vec{n}) \gamma_{\mathrm{lin}}'(\vec{n})^{-1} \pmod{[G, G]}.$$
Hence
$$g_{2}'(\vec{n}) \equiv \varepsilon_2'(\vec{n})^{-1}g_2(\vec{n})\gamma_2'(\vec{n})^{-1} \pmod{[G, G]}.$$
Since $\eta_i^2$ annihilates $[G, G]$, by scaling $\eta_i^2$ up by an appropriate integer $q \le c(\delta)^{-1}$, we have that
$$\|\eta_i^2 \circ g_{2}'\|_{C^\infty[\vec{N}]} \le c(\delta)^{-1}.$$
Hence, by Lemma~\ref{l:factorization2}, we may write
$$g'(\vec{n}) = \varepsilon^*(\vec{n})g_1(\vec{n})\gamma^*(\vec{n})$$
$\varepsilon^*$ is $(c(\delta)^{-1}, \vec{N})$-smooth, $\gamma^*$ is $c(\delta)^{-1}$-rational, and $g_1 \in \mathrm{poly}(\mathbb{Z}^\ell, \tilde{G})$ with the filtration $\tilde{G}_i = \tilde{G}_2 \cap G_i$ for $i \ge 2$. Letting $\varepsilon = \varepsilon' \varepsilon^*$ and $\gamma = \gamma'\gamma^*$, we have
$$g(\vec{n}) = \varepsilon(\vec{n})g_1(\vec{n})\gamma(\vec{n})$$
where by Lemma~\ref{l:multiplyrational}, $\gamma$ is $c(\delta)^{-1}$-rational. To show that $\varepsilon$ is smooth, we have from Lemma~\ref{l:multiplication} that $d(\varepsilon(0), \mathrm{id}_G) \le c(\delta)^{-1}$. In addition, by right invariance , we have
$$d(\varepsilon(\vec{n}), \varepsilon(\vec{n} + e_i)) = d(\varepsilon'(\vec{n}), \varepsilon'(\vec{n} + e_i)\varepsilon^*(\vec{n} + e_i) \varepsilon^*(\vec{n})^{-1})$$
so by Lemma~\ref{l:leftinvariance} and Lemma~\ref{l:distancecomparison} and since $d(\varepsilon'(\vec{n} + e_i), \mathrm{id}_G) \le c(\delta)^{-1}$ by smoothness, we have
\begin{align*}
d(\varepsilon'(\vec{n}), \varepsilon'(\vec{n} + e_i)\varepsilon^*(\vec{n} + e_i) \varepsilon^*(\vec{n})^{-1}) &\le d(\varepsilon'(\vec{n} + e_i)^{-1}\varepsilon'(\vec{n}), \varepsilon^*(\vec{n} + e_i)\varepsilon^*(\vec{n})^{-1}) \\
&\le c(\delta)^{-1}(|\psi(\varepsilon'(\vec{n} + e_i)^{-1}\varepsilon'(\vec{n}))| + |\psi(\varepsilon^*(\vec{n} + e_i)\varepsilon^*(\vec{n})^{-1})|) \\
&\le \frac{c(\delta)^{-1}}{N_i}.
\end{align*}
Hence, $\varepsilon$ is $(c(\delta)^{-1}, \vec{N})$-smooth. \\\\
Let $P$ be the period of $\gamma$. By Lemma~\ref{l:rationalpolynomialsequence}, we have $P \le c(\delta)^{-1}$. Since $\varepsilon$ is $(c(\delta)^{-1}, \vec{N})$-smooth, on subprogressions $P \cdot [\vec{N'}] + r$ of $[\vec{N}]$ with common difference $P \le c(\delta)^{-1}$ and size $c(\delta)|[\vec{N}]|$, we may approximate $F(\varepsilon(P \cdot \vec{n} + r)g_1(P \cdot \vec{n} + r)\gamma(P \cdot \vec{n} + r)\Gamma)$ with $F(\varepsilon_0 g_1(P \cdot \vec{n} + r) \gamma_0\Gamma)$ for some elements $\varepsilon_0, \gamma_0$ of $G$. Define $\tilde{F}(\cdot) := F(\epsilon_0 \{\gamma_0\} \cdot )$. We see that $\tilde{F}$ is a vertical character with frequency $\xi$ and
$$F(\varepsilon_0 g_1(P \cdot \vec{n} + r) \gamma_0\Gamma) = \tilde{F}(\{\gamma_0\}^{-1}g_1(P \cdot \vec{n} + r)\{\gamma_0\}\Gamma).$$
Since $g_1$ is a polynomial sequence with respect to the filtration on $\tilde{G}$, the same holds for $g_1(P \cdot \vec{n} + r)$. In addition, since $[G, G] \subseteq \tilde{G}_2$, we see that $\tilde{G}_i$ is normal in $G$ and thus $\{\gamma_0\}^{-1} g_1(P \cdot \vec{n} + r) \{\gamma_0\}$ also defines a polynomial sequence in $\mathrm{poly}(\mathbb{Z}^\ell, \tilde{G})$. \\\\
Pigeonholing in a subprogression of common difference $P$ and length $c(\delta)|\vec{N}|$ and pplying Case 1 and Lemma~\ref{l:factorization1} to $\{\gamma_0\}^{-1} g_1(P \cdot \vec{n} + r)\{\gamma_0\}$, we obtain a factorization $\{\gamma_0\}^{-1} g_1(P \cdot \vec{n} + r)\{\gamma_0\} = \varepsilon'(\vec{n}) g_1'(\vec{n}) \gamma'(\vec{n})$ with $\varepsilon'$ $(c(\delta)^{-1}, \vec{N})$-smooth, $\gamma'$ $c(\delta)^{-1}$-rational, and $g_1' \in \mathrm{poly}(\mathbb{Z}^\ell, H)$ where $H$ is at most $s - 1$-step. Since $H$ is $c(\delta)^{-1}$-rational in $\tilde{G}$, it is $c(\delta)^{-1}$-rational in $G$. Applying Lemma~\ref{l:lowerstep} to $H$, we obtain horizontal characters $\alpha_1, \dots, \alpha_r$ on $G$ such that for each $j \in \{1, \dots, r\}$,
$$\|\alpha_j \circ g_1'\|_{C^\infty[\vec{N}']} = 0$$
and thus
$$\|\alpha_j \circ g_1'\|_{C^\infty[\vec{N}]} = 0$$
where $P \cdot [\vec{N}'] + r$ was the subprogression that we Pigeonholed to and for any $\beta_1, \dots, \beta_s$ orthogonal to $\alpha_1, \dots, \alpha_r$, we have
$$\xi([\beta_1, \dots, \beta_s]) = 0.$$
From Lemma~\ref{l:multiparameterextrapolation}, there exists some nonzero integer $|q| \le c(\delta)^{-1}$ such that
$$\|q\alpha_i \circ g_1\|_{C^\infty[\vec{N}]} \le c(\delta)^{-1}$$
and thus for some nonzero integer $|q'| \le c(\delta)^{-1}$
$$\|q'\alpha_i \circ g\|_{C^\infty[\vec{N}]} \le c(\delta)^{-1}$$
which proves the desired result. This completes the proof of Theorem~\ref{t:mainresult1}.

\section{Application to linear equations in primes}\label{s:linearprimes}
In this section, we deduce Theorem~\ref{t:quantthm}. Define $\mu_{\mathrm{Siegel}}$ and $\Lambda_{\mathrm{Siegel}}$ as follows: let $\beta$ be a possible Siegel zero of level $Q = \exp(\log^{1/10}(N))$ with conductor $q_{\mathrm{Siegel}} \le Q$. We let
$$P(Q) := \prod_{\substack{p < Q \\ p \text{ prime}}} p,$$
$\chi_{\mathrm{Siegel}}$ is the character for the Siegel zero $\beta$, and
\begin{align*}
\alpha &:= \frac{1}{L'(\beta, \chi_{\mathrm{Siegel}})}\prod_{p < Q} \left(1 - \frac{1}{p}\right)^{-1}\left(1 - \frac{\chi_{\mathrm{Siegel}}(p)}{p^\beta}\right)^{-1} \\
\mu_{\mathrm{Siegel}}(n) &:= (1_{n | P(Q)} \mu(n)) * (\alpha n^{\beta - 1} \chi_{\mathrm{Siegel}}(n)1_{(n, P(Q)) = 1}) \\
\Lambda_Q(n) &:= \frac{P(Q)}{\phi(P(Q))}1_{(n, P(Q)) = 1} \\
\Lambda_{\mathrm{Siegel}}(n) &:= \Lambda_Q(n)(1 - n^{\beta - 1} \chi_{\mathrm{Siegel}}(n)).
\end{align*}
Our main estimate is the following; the reader can find the definition and basic properties of the $U^{s + 1}[N]$ norm in \cite[Chapter 1.3]{Tao12}.
\begin{theorem}\label{t:mobiusuniformity}
There exists some real $c_s > 0$ such that
$$\|\mu - \mu_{\mathrm{Siegel}}\|_{U^{s + 1}([N])} \ll \exp(-\log(N)^{c_s})$$
$$\|\Lambda - \Lambda_{\mathrm{Siegel}}\|_{U^{s + 1}([N])} \ll \exp(-\log(N)^{c_s}).$$
\end{theorem}
We remark that as in \cite{TT21}, the bounds in Theorem~\ref{t:mobiusuniformity} are effective. We will defer the proof of Theorem~\ref{t:mobiusuniformity} and Theorem~\ref{t:quantthm} in Section~\ref{s:theoremmobius}. In order to prove Theorem~\ref{t:mobiusuniformity}, we record \cite[Proposition 7.2]{TT21} below, which states we may decompose $\mu$, $\Lambda$, $\mu_{\mathrm{Siegel}}$, $\Lambda_{\mathrm{Siegel}}$, and $\Lambda_Q$ into type I, twisted type I, and type II sums. 
\begin{definition}
We say a sequence $a_d$ is divisor-bounded if there exists some $k$ such that $a_d \ll \tau(d)^k\log(N)^k$ where $\tau(d)$ is the sum of the divisors of $d$.
\end{definition}

\begin{lemma}\label{l:type1type2decomposition}
Any of the five functions $\mu, \mu_{\mathrm{Siegel}}, \Lambda, \Lambda_{\mathrm{Siegel}}, \Lambda_Q$ on $[N]$ can be expressed as a convex linear combination of functions of one of the following four classes (with uniform constants in the bounds).
\begin{itemize}
    \item[(i)] (Type I sum) A function of the form
    $$n \mapsto \sum_{d \le N^{2/3}} a_d 1_{d|n} 1_{[N']}(n)$$
    where the coefficients $a_d$ are divisor-bounded and $1 \le N' \le N$.
    \item[(ii)] (Twisted type I sum) A function of the form
    $$n \mapsto \sum_{d \le N^{2/3}} a_d 1_{d|n} \chi_{\mathrm{Siegel}}(n/d)1_{[N']}(n)$$
    where the coefficients $a_d$ are divisor-bounded and $1 \le N' \le N$.
    \item[(iii)] (Type II sum) A function of the form
    $$n \mapsto \sum_{d, w > N^{1/3}} a_d b_w 1_{dw = n}$$
    for some divisor-bounded coefficients $a_d, b_w$.
    \item[(iv)] (Negligible sum) A divisor-bounded function $n \mapsto f(n)$ with
    $$\sum_{n \in [N]} |f(n)| \ll N \exp(-\log^{1/2}(N)).$$
\end{itemize}
\end{lemma}
Thus, to study the correlation between a nilsequence with $\mu, \mu_{\mathrm{Siegel}}, \Lambda, $ or $\Lambda_{\mathrm{Siegel}}$, it suffices to study the correlation of the nilsequence with a type I, twisted type I, and a type II sum. The advantage of working with these sums is evident in the following.
\begin{proposition}\label{p:type1type2correlation}
Let $g(n)$ be a type I sum, $g_1(n)$ be a twisted type I sum and $g_2(n)$ be a type II sum. Suppose $f\colon [N] \to \mathbb{C}$ is a one-bounded function and $\delta \in (0, 1/10)$.
\begin{itemize}
    \item Suppose $\left|\sum_{n \in [N]} f(n)g(n)\right| \ge \delta N$. Then there exists $c > 0$ such that if $\delta \gg \exp(-c\log^{1/3}(N))$, $N' \le N$, there exists $L \le N^{2/3}$ such that for $\delta \log^{-O(1)}(N)L$ many $\ell \in [L/2, L]$,
    $$\left|\sum_{n \in [N'/\ell]} f(\ell n)\right| \ge \frac{\delta N}{L\log^{O(1)}(N)}.$$
    \item Suppose $\left|\sum_{n \in [N]} f(n)g_1(n)\right| \ge \delta N $. Then there exists some $c > 0$ such that if $\delta \gg \exp(-c\log^{1/3}(N))$ and $N' \le N$, there exists $L \le N^{2/3}$ such that $\delta \log^{-O(1)}(N)L$ many $\ell \in [L/2, L]$,
    $$\left|\sum_{n \in [N'/\ell]} f(\ell n)\chi_{\mathrm{Siegel}}(n)\right| \ge \frac{\delta N}{L\log^{O(1)}(N)}.$$
    \item Suppose $\left|\sum_{n \in [\delta N, N]} f(n)g_2(n)\right| \ge \delta N$. Then there exists some $c > 0$ such that if $\delta \gg \exp(-c\log^{1/3}(N))$ and $N^{1/3} \le L \le N^{2/3}$ and $M \in [\delta N/L, N/L]$ then
    $$\sum_{\ell, \ell' \in [L/2, L]} \sum_{m, m' \in [M/2, M]} f(\ell m)\overline{f(\ell'm)f(\ell m')}f(\ell'm') \ge \frac{\delta^{O(1)}N^2}{\log^{O(1)}(N)}.$$
\end{itemize}
\end{proposition}
This motivates us to prove estimates for type I, twisted type I, and type II estimates, which occupy the next two subsections.

\subsection{The type I and twisted type I case}
We first record the type I and twisted type I estimates we must prove.
\begin{proposition}[Type I estimate]\label{p:type1}
Let $G/\Gamma$ be a nilmanifold with dimension $d$, step $s$, degree $k$, complexity $M$, and one-dimensional vertical component. Let $F\colon G/\Gamma \to \mathbb{C}$ be a $1$-Lipschitz function and a vertical character with frequency $\xi$ such that $|\xi| \le M/\delta$. Let $P$ be an arithmetic progression in $[N]$ and $F_P(g(n)\Gamma) = F(g(n)\Gamma)1_P(n)$. Suppose for more than $\delta^{O(1)}L$ many $\ell \in [L/2, L]$ with $L \le N^{1/3}$,
\begin{equation}\label{e:maintype1inequality}
\left|\sum_{\substack{n \in [N] \\ n \equiv 0 \nmod{\ell}}} F_P(g(n)\Gamma)\right| \ge \delta N/L.    
\end{equation}
Then there exists some $c_k > 0$ such that if $\delta \gg \exp(-\log^{c_k}(N))$, then there exists a factorization $g(n) = \varepsilon(n)g_1(n)\gamma(n)$ such that $g_1$ takes values in a $(\delta/M)^{-O_k(d^{O_k(1)})}$-rational subgroup of $G$ whose step is strictly smaller than the step of $G$, $\varepsilon$ is $((\delta/M)^{-O_k(d^{O_k(1)})}, N)$-smooth, and $\gamma$ is $(\delta/M)^{-O_k(d^{O_k(1)})}$-rational.
\end{proposition}

\begin{proposition}[Twisted type I estimate]\label{p:twistedtype1}
Let $G/\Gamma$ be a nilmanifold with dimension $d$, step $s$, degree $k$, complexity $M$, and one-dimensional vertical component. Let $F\colon G/\Gamma \to \mathbb{C}$ be a $1$-Lipschitz function and a vertical character with frequency $\xi$ such that $|\xi| \le M/\delta$. Let $P$ be an arithmetic progression in $[N]$ and $F_P(g(n)\Gamma) = F(g(n)\Gamma)1_P(n)$. Suppose for more than $\delta^{O(1)}L$ many $\ell \in [L/2, L]$ with $L \le N^{1/3}$,
\begin{equation}\label{e:maintwistedtype1inequality}
\left|\sum_{\substack{n \in [N]\\ n \equiv 0 \nmod{\ell}}}  \chi_{\mathrm{Siegel}}(n/d)F_P(g(n)\Gamma)\right| \ge \delta N/L.    
\end{equation}
Then there exists some $c_k > 0$ such that if $\delta \gg \exp(-\log^{c_k}(N))$ and $q_{\mathrm{Siegel}} \le (\delta/M)^{-O_k(d^{O_k(1)})}$, then there exists a factorization $g(n) = \varepsilon(n)g_1(n)\gamma(n)$ such that $g_1$ takes values in a $(\delta/M)^{-O_k(d^{O_k(1)})}$-rational subgroup of $G$ whose step is strictly smaller than the step of $G$, $\varepsilon$ is $((\delta/M)^{-O_k(d^{O_k(1)})}, N)$-smooth, and $\gamma$ is $(\delta/M)^{-O_k(d^{O_k(1)})}$-rational.
\end{proposition}
\begin{proof}[Proof of Proposition~\ref{p:type1} and Proposition~\ref{p:twistedtype1}.]
First, observe we can take $|P| \ge c(\delta)N$, for otherwise, we may upper-bound \eqref{e:maintype1inequality} or \eqref{e:maintwistedtype1inequality} by $c(\delta)N/L$, which is a contradiction to the lower bound of $\delta N/L$.  \\\\
We now prove Proposition~\ref{p:type1}. Write $P = q \cdot [N'] + r$. Splitting $[\delta N/L, N/L]$ into subprogressions of common difference $q$, we see by the Pigeonhole principle that there exists a residue class $b$ modulo $q$ such that for $c(\delta)L$ many $\ell \in [L/2, L]$, 
$$\left|\sum_{\substack{m \in [N/L] \\ m \equiv b \nmod{q}}} 1_P(m\ell)F(g(m\ell)\Gamma)\right| \ge \delta \frac{N}{Lq}.$$
Note that since $m\ell \equiv b\ell \pmod{q}$, the indicator function of $P$ can be dropped after restricting to a sufficiently short interval. Letting $m = qm' + b$, we have intervals $I_{\ell}$ of size at least $c(\delta)N/L$ such that for at least $c(\delta)L$ many $\ell \in [L, 2L]$,
$$\left|\mathbb{E}_{m' \in I_\ell} F(g((qm' + b)\ell)\Gamma)\right| \ge c(\delta).$$
Let $\tilde{g}_\ell = g((q\cdot + b)\ell)$ define a polynomial sequence in $\mathrm{poly}(\mathbb{Z}, G)$. Applying Theorem~\ref{t:mainresult1} (and shifting $I_\ell$ by its least element), there exists linearly independent horizontal characters $\eta_1, \dots, \eta_r$ of size at most $c(\delta)^{-1}$ such that for $c(\delta)L$ many elements $\ell \in [L/2, L]$, we have
$$\|\eta_i \circ \tilde{g}_\ell\|_{C^\infty[|I_\ell|]} \le c(\delta)^{-1}$$
which implies since $|I_\ell| \ge c(\delta)N/L$,
$$\|\eta_i \circ \tilde{g}_\ell\|_{C^\infty[N/L]} \le c(\delta)^{-1}.$$ 
Furthermore, the subgroup
$$G^* := \{g \in G: \eta_i(g) = 0 \text{ for all } i\}$$
has step strictly less than $s$. Now let $g_\ell = g(\ell \cdot)$; by Lemma~\ref{l:multiparameterextrapolation}, we have by scaling $\eta_i$ appropriately that
\begin{equation}\label{e:type1inequality}
\|\eta_i \circ g_\ell\|_{C^\infty[N/L]} \le c(\delta)^{-1}.
\end{equation}
Inequality \eqref{e:type1inequality} implies via Lemma~\ref{l:cinfinity}, Lemma~\ref{l:vinogradov}, and the fact that $L \le N^{1/3}$ that for some integer $q = O_k(1)$,
$$\|q\eta_i \circ g\|_{C^\infty[N]} \le c(\delta)^{-1}.$$
By Lemma~\ref{l:factorization1}, it follows that we may decompose $g = \varepsilon g_1\gamma$ with $g_1 \in \mathrm{poly}(\mathbb{Z}, G^*)$ with $G^*$ equipped with the filtration $G^*_i = G_i \cap G^*$. \\\\
We now turn to Proposition~\ref{p:twistedtype1}. We use the hypothesis of $q_{\mathrm{Siegel}} \le c(\delta)^{-1}$ to Pigeonhole $\chi_{\mathrm{Siegel}}$ in a progression modulo $q_{\mathrm{Siegel}}$, obtaining for at least $c(\delta)L$ many $\ell \in [L/2, L]$ 
$$|\sum_{m \in [N/\ell]} 1_{m\ell \equiv a\nmod{q_{\mathrm{Siegel}}}}1_P(m\ell) F(g(m\ell)\Gamma)| \ge c(\delta)N/\ell$$
Note that $1_{\cdot \equiv a \nmod{q_{\mathrm{Siegel}}}}1_P$ is the indicator function of yet another progression $1_{P'}$ with $|P'| \ge c(\delta)N$. Thus, we may now proceed exactly as in the previous part.
\end{proof}

\subsection{The type II case}
We now need the following type II estimate.
\begin{proposition}[Type II estimate]\label{p:type2}
Let $G/\Gamma$ be a nilmanifold with dimension $d$, step $s$, degree $k$, complexity $M$, and one-dimensional vertical component. Let $F\colon G/\Gamma \to \mathbb{C}$ be a $1$-Lipschitz function and vertical character with frequency $\xi$ such that $|\xi| \le M/\delta$. Let $P$ be an arithmetic progression in $[N]$ and $F_P(g(n)\Gamma) = F(g(n)\Gamma)1_P(n)$. Suppose there exists $L$ and $M$ with $N^{1/3}\le L \le N^{2/3}$, $\delta N \le LM \le N/4$, and
\begin{equation}\label{e:maintype2inequality}
 \left|\sum_{\substack{\ell, \ell' \in [L/2, L] \\ m, m' \in [M/2, M]}}  F_P(g(m\ell)\Gamma)F_P(g(m'\ell')\Gamma)\overline{F_P(g(m'\ell)\Gamma)F_P(g(m\ell')\Gamma)} \right| \ge \delta L^2M^2.   
\end{equation}
Then there exists some $c_k > 0$ such that if $\delta \gg \exp(-\log^{c_k}(N))$, then there exists a factorization $g(n) = \varepsilon(n)g_1(n)\gamma(n)$ such that $g_1$ takes values in a $(\delta/M)^{-O_k(d^{O_k(1)})}$-rational subgroup of $G$ whose step is strictly smaller than the step of $G$, $\varepsilon$ is $((\delta/M)^{-O_k(d^{O_k(1)})}, N)$-smooth, and $\gamma$ is $(\delta/M)^{-O_k(d^{O_k(1)})}$-rational.
\end{proposition}
\begin{proof}
We first prove that we can take $|P| \ge c(\delta)N$. If not, then \eqref{e:maintype2inequality} can be upper bounded by $c(\delta)N^2$, which contradicts the lower bound of $\delta L^2M^2$. \\\\
By hypothesis, we have
$$\mathbb{E}_{\ell, \ell' \in [L/2, L]} |\mathbb{E}_{m \in [M/2, M]}  F_P(g(m\ell)\Gamma)\overline{F_P(g(m\ell')\Gamma)}|^2 \ge \delta.$$
This implies that there exists $\delta^{O(1)}L^2$ many $\ell, \ell' \in [L, 2L]$ such that
$$|\mathbb{E}_{m \in [M/2, M]}  F_P(g(m\ell)\Gamma)\overline{F_P(g(m\ell')\Gamma)}| \ge \delta^{O(1)}.$$
Write $P = q \cdot [N'] + r$. By the pigeonhole principle, there exists some residue class $b \pmod{q}$ such that for $c(\delta)L^2$ many $\ell, \ell' \in [L/2, L]$,
$$|\mathbb{E}_{\substack{m \in [M/2, M] \\ m \equiv b \nmod{q}}} 1_P(m\ell)1_P(m\ell') F(g(m\ell)\Gamma)\overline{F(g(m\ell')\Gamma)}| \ge \delta^{O(1)}.$$
Note that since $m\ell \equiv b\ell \pmod{q}$ and $m\ell' \equiv b\ell' \pmod{q}$, the indicator function of $P$ can be dropped after restricting to a sufficiently short interval. Letting $m = qm' + b$, we have intervals $I_{\ell, \ell'}$ of size at least $c(\delta)N/L$ such that for at least $c(\delta)L^2$ many $\ell, \ell' \in [L/2, L]$, we have
$$|\mathbb{E}_{m' \in I_{\ell, \ell'}} F(g((pm' + b)\ell)\Gamma)\overline{F(g((qm' + b)\ell')\Gamma)}| \ge c(\delta).$$
Let $\tilde{g}_{\ell, \ell'}(m') = (g((qm' + b)\ell), g((qm' + b)\ell'))$ be a polynomial sequence in $\mathrm{poly}(\mathbb{Z}, G^2)$. \\\\
Applying Theorem~\ref{t:mainresult1} and Lemma~\ref{l:multiparameterextrapolation}, it follows that there exists horizontal characters $\eta_1, \dots, \eta_r$ on $G \times G$ of size at most $c(\delta)^{-1}$ such that the subgroup $\tilde{G}_1 := \{g \in G_1: \eta_j(g) = 0 \text{ for all } j\}$ has step at most $s - 1$ and for those values of $\ell, \ell'$, we have
$$\|\eta_j \circ \tilde{g}_{\ell, \ell'}\|_{C^\infty[|I_{\ell, \ell'}|]} \le c(\delta)^{-1}$$
and since $|I_{\ell, \ell'}| \ge c(\delta)N/L$,
$$\|\eta_j \circ \tilde{g}_{\ell, \ell'}\|_{C^\infty[N/L]} \le c(\delta)^{-1}.$$
Let $g_{\ell, \ell'} = g(\cdot \ell, \cdot \ell')$. Applying Lemma~\ref{l:multiparameterextrapolation} and scaling up $\eta_j$ appropriately, we thus have
$$\|\eta_j \circ g_{\ell, \ell'}\|_{C^\infty[N/L]} \le c(\delta)^{-1}.$$
Writing $\eta_j(g, g') = \eta_j^0(g) + \eta_j^1(g')$, an application of Lemma~\ref{l:cinfinity} and Lemma~\ref{l:vinogradov}  gives that (upon scaling $\eta_j^i$ appropriately)
$$\|\eta_j^i \circ g\|_{C^\infty[N]} \le c(\delta)^{-1}.$$
We now show that the subgroup $\tilde{G} := \{g \in G: \eta_j^i(g) = 0 \text{ for all } i, j\}$ is $\le s - 1$-step, from which the type II case would follow from an application of Lemma~\ref{l:factorization1}. We see that if $\eta_j^0$ annihilates $w$, then $\eta_j$ annihilates $(w, \mathrm{id}_G)$. Thus, since $\tilde{G}$ is the subgroup generated by the annihilators of $\eta_j^0$, we see that $\tilde{G} \times \mathrm{id}_G$ is a subgroup of the group generated by the annihilators of $\eta_j$, which is $\tilde{G}_1$. Hence, $\tilde{G}$ has step strictly less than $s$.
\end{proof}

\subsection{Deducing Theorem~\ref{t:mobiusuniformity} and Theorem~\ref{t:quantthm}}\label{s:theoremmobius}
We are now ready to prove Theorem~\ref{t:mobiusuniformity}.
\begin{proof}[Proof of Theorem~\ref{t:mobiusuniformity}]
We will now prove Theorem~\ref{t:mobiusuniformity}. We work with either hypothesis
$$\|\mu - \mu_{\mathrm{Siegel}}\|_{U^{s + 1}([N])} \ge \delta$$
$$\|\Lambda - \Lambda_{\mathrm{Siegel}}\|_{U^{s + 1}([N])} \ge \delta$$
and optimize in $\delta$. We will assume throughout that $\delta \le \log(N)^{-1}$. We show that if $\delta$ is sufficiently large, then these hypotheses cannot hold. By \cite[Lemma 2.4]{TT21}, we have the bounds
$$\mu, \mu_{\mathrm{Siegel}} \ll 1 \text{ and } \Lambda, \Lambda_{\mathrm{Siegel}} \ll \log(N).$$
Thus, we may apply \cite[Theorem 1.2]{LSS24b} directly to obtain a nilsequence $F(g(\cdot)\Gamma)$ such that
\begin{equation}\label{e:beginningmobius}
|\mathbb{E}_{n \in [N]} (\mu - \mu_{\mathrm{Siegel}})(n) F(g(n)\Gamma) | \ge \exp(-\log(1/\delta)^{O(1)})   
\end{equation}
or 
\begin{equation}\label{e:beginningvonmango}
|\mathbb{E}_{n \in [N]} (\Lambda - \Lambda_{\mathrm{Siegel}})F(g(n)\Gamma)| \ge \exp(-\log(1/\delta)^{O(1)})    
\end{equation}
where $G/\Gamma$ has degree $k \le s$, dimension $d \le \log(1/\delta)^{O(1)}$, and complexity $M \le \exp(\log(1/\delta)^{O(1)})$. By Fourier expanding on the vertical torus via Lemma~\ref{l:nilcharacters}, we may further assume that $F$ is a vertical character with frequency $\xi$ with $|\xi| \le \delta^{-O(d)^{O(1)}}$. Finally, by Lemma~\ref{l:trick2}, we may assume that the vertical component is one-dimensional. \\\\
We now iterate the conclusions of Proposition~\ref{p:type1}, Proposition~\ref{p:twistedtype1}, and Proposition~\ref{p:type2} in order to one-by-one lower the step of $G$ until $g \equiv \mathrm{id}_G$, maintaining that $\|F\|_{\mathrm{Lip}(G/\Gamma)} \le c(\delta)^{-1}$, $\delta \ge \exp(-\log(N)^{c_k})$, $M \le c(\delta)^{-1}$, and $|P| \ge c(\delta)N$. We begin with 
\begin{equation} \label{e:beginningprogressionmobius}
|\mathbb{E}_{n \in [N]} (\mu - \mu_{\mathrm{Siegel}})(n) F_P(g(n)\Gamma)| \ge c(\delta)
\end{equation}
or
\begin{equation}\label{e:beginningprogressionvonmango}
|\mathbb{E}_{n \in [N]} (\Lambda - \Lambda_{\mathrm{Siegel}})(n) F_P(g(n)\Gamma)| \ge c(\delta).
\end{equation}
Applying Lemma~\ref{l:type1type2decomposition} and Proposition~\ref{p:type1type2correlation}, we obtain the hypotheses of \eqref{e:maintype1inequality} or \eqref{e:maintwistedtype1inequality} or \eqref{e:maintype2inequality} with $c(\delta)$ in place of $c(\delta)$.  \\\\
We now take care of the case of when $q_{\mathrm{Siegel}} \ge c(\delta)^{-1}$ where $c(\delta)$ is sufficiently large. By \cite[Theorem~2.5]{TT21} and the converse to the inverse theorem of Gowers norm (see \cite[Lemma~B.5]{LSS24b}), we have
$$|\mathbb{E}_{n \in [N]} \mu_{\mathrm{Siegel}}(n) F_P(g(n)\Gamma)| \ll c(\delta)^{-1} q_{\mathrm{Siegel}}^{1/d^{O_s(1)}}$$
or 
$$|\mathbb{E}_{n \in [N]} (\Lambda_{\mathrm{Siegel}}(n) - \Lambda_Q(n))F_P(g(n)\Gamma)| \ll c(\delta)^{-1} q_{\mathrm{Siegel}}^{1/d^{O_s(1)}}.$$
Thus, if $q_{\mathrm{Siegel}}$ is sufficiently large, then by the triangle inequality and \eqref{e:beginningprogressionmobius} or \eqref{e:beginningprogressionvonmango}, we have
\begin{equation}
|\mathbb{E}_{n \in [N]} \mu(n)F_P(g(n)\Gamma)| \ge c(\delta)
\end{equation}
or
\begin{equation}
|\mathbb{E}_{n \in [N]} (\Lambda(n) - \Lambda_Q(n))F_P(g(n)\Gamma)| \ge c(\delta).    
\end{equation}
Each of the three functions $\mu$, $\Lambda$, and $\Lambda_Q$ can be decomposed into purely type I and type II sums. Thus, when considering twisted type I sums, we may assume that $q_{\mathrm{Siegel}} \le c(\delta)^{-1}$. \\\\
We have now verified the conditions to apply Proposition~\ref{p:type1}, Proposition~\ref{p:twistedtype1}, and Proposition~\ref{p:type2}. Each of these lemmas states that we may write $g(n) = \varepsilon(n)g_1(n)\gamma(n)$ where $g_1$ is a polynomial sequence on a $c(\delta)^{-1}$-rational subgroup of step at most $s - 1$. Thus, either
$$|\mathbb{E}_{n \in [N]} F_P(\varepsilon(n)g_1(n)\gamma(n)\Gamma)(\mu(n) - \mu_{\mathrm{Siegel}}(n))| \ge c(\delta)$$
or
$$|\mathbb{E}_{n \in [N]} F_P(\varepsilon(n)g_1(n)\gamma(n)\Gamma)(\Lambda(n) - \Lambda_{\mathrm{Siegel}}(n))| \ge c(\delta).$$
Let $Q$ be the period of $\gamma$, so $\gamma(Qn + b)\Gamma$ is constant in $n$ for each $b$. Thus, there exists an arithmetic progression $P' \subseteq P$ of common difference $Q$ and size $c(\delta)N$ such that for any two elements $m, m' \in P'$, $d(\varepsilon(m), \varepsilon(m')) \le c(\delta)$, and
$$|\mathbb{E}_{n \in [N]} F_{P'}(\varepsilon(n)g_1(n)\gamma(n)\Gamma)(\mu(n) - \mu_{\mathrm{Siegel}}(n))| \ge c(\delta)$$
or
$$|\mathbb{E}_{n \in [N]} F_{P'}(\varepsilon(n)g_1(n)\gamma(n)\Gamma)(\Lambda(n) - \Lambda_{\mathrm{Siegel}}(n))| \ge c(\delta).$$
Thus, there exists some $\varepsilon_0$ and $\gamma_0$ such that
$$|\mathbb{E}_{n \in [N]} F_{P'}(\varepsilon_0 g_1(n)\gamma_0\Gamma)(\mu(n) - \mu_{\mathrm{Siegel}}(n))| \ge c(\delta)$$
or
$$|\mathbb{E}_{n \in [N]} F_{P'}(\varepsilon_0 g_1(n)\gamma_0\Gamma)(\Lambda(n) - \Lambda_{\mathrm{Siegel}}(n))| \ge c(\delta).$$
By Lemma~\ref{l:decomposition}, we have the factorization $\gamma_0 = \{\gamma_0\}[\gamma_0]$ where $[\gamma_0] \in \Gamma$ and $\psi(\{\gamma_0\}) \in [-1/2, 1/2)^d$. We may now write $\varepsilon_0 g_1(n)\gamma_0 = \varepsilon_0\{\gamma_0\} (\{\gamma_0\}^{-1}g_1(n)\{\gamma_0\})[\gamma_0]$. Replacing $F$ by $F(\varepsilon_0\{\gamma_0\} \cdot)$, $g$ by $\{\gamma_0\}^{-1}g_1(n)\{\gamma_0\}$, and $P$ by $P'$, we may repeat the iteration. Since the iteration happens at most $O_k(1)$ many times, we can maintain $\delta \ge \exp(-\log^{c_k}(N))$. The iteration terminates when $g(n) = \mathrm{id}_G$ for each $n$, in which case we have that $F = c1_{P'}$ for $0 < |c| \le 1$ and may invoke \cite[Proposition~2.2]{TT21} (a variant of the Siegel--Walfisz Theorem which accounts for a Siegel correction to obtain quasipolynomial bounds)  to finish.    
\end{proof}

We are finally ready to prove Theorem~\ref{t:quantthm}.
\begin{proof}[Proof of Theorem~\ref{t:quantthm}.]
For $\Psi = (\psi_1, \psi_2, \dots, \psi_t)$ as in the statement of Theorem~\ref{t:quantthm}, functions $f_1, \dots, f_t\colon [-N, N] \to \mathbb{Z}$, and $N'$ a prime between $10N$ and $20N$, we may embed $[-N, N]$ inside $\mathbb{Z}/N'\mathbb{Z}$ and extend $1_{\Omega}$, $f_1, \dots, f_t$ to be zero outside $[-N, N]$. We now define the quantities:
\begin{align*}
\Lambda_{\Psi, \Omega}(f_1, \dots, f_t) &= \frac{1}{N^d}\sum_{\vec{n} \in \Omega} \prod_{i = 1}^t f_i(\psi_i(\vec{n}))  \\
\Lambda_{\Psi, N'}(f_1, \dots, f_t) &= \mathbb{E}_{\vec{n} \in (\mathbb{Z}/N'\mathbb{Z})^d } \prod_{i = 1}^t f_i(\psi_i(\vec{n})).
\end{align*}
By \cite[Exercise 1.3.23]{Tao12} and monotonicity of the Gowers norms, we see that for $f_i$ one-bounded and for some integer $s > 1$ depending only on $\Psi$ that
\begin{equation}\label{e:cauchyschwarzcomplexity}
 |\Lambda_{\Psi, N'}(f_1, \dots, f_t)| \le \min_{i \in [t]} \|f_i\|_{U^{s + 1}(\mathbb{Z}/N'\mathbb{Z})} \ll \min_{i \in [t]} \|f_i\|_{U^{s + 1}[N]}   
\end{equation}
In addition, by \cite[Corollary A.3]{GT10} and Lemma~\ref{l:FourierExpansion}, we may Fourier expand
$$1_\Omega = \sum_{i = 1}^J a_i e(\alpha_i \cdot ) + g$$
where $\|g\|_{L^1[N']} \le \log^{-B}(N)$, $\sum_{i = 1}^J |a_i| \ll \log^{O_{d, B}(1)}(N)$, and $\alpha_i$ are rationals with denominator $N'$. If $N'(\alpha_i \cdot n)$ lies in the span of $(\psi_j)_{j = 1}^k$ modulo $N'$, then we may write $N' \alpha_i \cdot n$ into a linear combination of $\psi_j$ so $e(\alpha_i \cdot n)$ gets factored into a multiplicative combination of $e(\psi_j(n)/N')$. These terms may be absorbed in $f_j(\psi_j(n))$. Since $s > 1$, these phases don't contribute to the $U^{s + 1}(\mathbb{Z}/N'\mathbb{Z})$ norm of $f_j$. If $N'(\alpha_i \cdot n)$ does not lie in the span of $(\psi_j)_{j = 1}^t$, then by orthogonality,
$$\mathbb{E}_{n \in (\mathbb{Z}/N'\mathbb{Z})^d} e(\alpha_i \cdot n) \prod_{j = 1}^t f_j(\psi_j(n)) = 0.$$
Thus, by \eqref{e:cauchyschwarzcomplexity}, we have
\begin{equation}\label{e:gowersbound}
|\Lambda_{\Psi, \Omega} (f_j)_{j = 1}^t| \ll \min_{j} \log^{O_{t, B}(1)}(N)\|f_j\|_{U^{s + 1}[N]} + O(\log^{-B}(N)).
\end{equation}
Letting $\Lambda_{\mathrm{Siegel}}$ be the approximant of $\Lambda$ we defined, we may write
$$\Lambda_{\Psi, \Omega}((\Lambda)_{j = 1}^t) - \Lambda_{\Psi, \Omega}((\Lambda_{\mathrm{Siegel}})_{j = 1}^t)$$
as a sum of $\Lambda_{\Psi, \Omega}((g_j)_{j = 1}^t)$ where at least one of $g_j$ is equal to $\Lambda - \Lambda_{\mathrm{Siegel}}$ and the rest of the terms are $O(\log(N))$-bounded. By Theorem~\ref{t:mobiusuniformity} and \eqref{e:gowersbound}, it follows from choosing $A$ to be sufficiently large that each
$$|\Lambda_{\Psi, \Omega}((g_j)_{j = 1}^t)| \ll_{A, t,L} \log^{-A}(N).$$
We may also write $\Lambda_{\mathrm{Siegel}} = \Lambda_Q + n^{\beta - 1}\frac{\phi(P(Q))}{P(Q)} 1_{(n, P(Q)) = 1}$. In addition, we may separate 
$$\Lambda_{\Psi, \Omega}((\Lambda_{\mathrm{Siegel}})_{j = 1}^t) - \Lambda_{\Psi, \Omega}((\Lambda_Q)_{j = 1}^t)$$
into $O_t(1)$ terms of the form $\Lambda_{\Psi, \Omega}((h_j)_{j = 1}^t)$, each of which has at least one $h_j$ that is of the form $\Lambda_{\mathrm{Siegel}} - \Lambda_Q$ with the rest of the terms $O(\log(N))$-bounded. Then \cite[Theorem 2.5]{TT21} and \eqref{e:gowersbound} gives us that each
$$|\Lambda_{\Psi, \Omega}((h_j)_{j = 1}^t)| \ll_{A, t, L} \log^{-A}(N).$$
Here, we invoke Siegel's theorem of $1 - \beta \gg_\varepsilon^{\mathrm{ineff}} q_{\mathrm{Siegel}}^{-\varepsilon}$ rather than the estimate $1 - \beta \gg q_{\mathrm{Siegel}}^{-1/2}\log^{-2}q_{\mathrm{Siegel}}$ used in \cite{TT21}; this is the source of ineffectivity in our bounds. Finally, \cite[Proposition 5.2]{TT21} gives us
$$\Lambda_{\Psi, \Omega}((\Lambda_Q)_{j = 1}^t) = \beta_\infty \prod_{p \text{ prime}} \beta_p + O_{L, d, t}(\exp(-\log^{c_s}(N))).$$
This gives us Theorem~\ref{t:quantthm}.
\end{proof}

\appendix

\section{Auxiliary lemmas}\label{s:aux}
In this section, we shall state auxiliary lemmas we use in the proof of our main theorem. The first two lemmas are slight modifications of \cite[Proposition 9.2]{GT12} and \cite[Lemma 7.9]{GT12}, respectively.
\begin{lemma}[Factorization lemma I]\label{l:factorization1}
Let $G/\Gamma$ be a nilmanifold of step $s$, degree $k$, dimension $d$, and complexity $M$. Let $g \in \mathrm{poly}(\mathbb{Z}^\ell, G)$ and $\eta_1, \dots, \eta_r$ are a set of linearly independent nonzero horizontal characters of size at most $L$. Suppose $\|\eta_j \circ g\|_{C^\infty[\vec{N}]} \le \delta^{-1}$ for each $i$. Then we have a factorization
$$g(\vec{n}) = \varepsilon(\vec{n})g_1(\vec{n})\gamma(\vec{n})$$
where
\begin{itemize}
    \item $\varepsilon$ is $((ML/\delta)^{O_{k, \ell}(d^{O_{k, \ell}(1)})}, \vec{N})$-smooth;
    \item $\gamma$ is $(ML)^{O_{k, \ell}(d^{O_{k, \ell}(1)})}$-rational;
    \item $g_1(\vec{n})$ lies in $\tilde{G} = \bigcap_{j = 1}^r \operatorname{ker}(\eta_j)$
which has complexity at most $(ML)^{O_{k, \ell}(d^{O_{k, \ell}(1)})}$.
\end{itemize} 
Furthermore, if $g(0) = \mathrm{id}_G$, then we can take $\varepsilon(0) = \gamma(0) = \mathrm{id}_G$.
\end{lemma}

\begin{proof}
    We first show that we can reduce to the case that $g(0) = \mathrm{id}_G$. To see this, we write $g'(n) = \{g(0)\}^{-1}g(n)[g(0)]^{-1}$. It follows that $g'(0) = \mathrm{id}_G$, and we may replace $g$ with $g'$ and continue the analysis. \\\\
    We write in coordinates that
    $$\psi(g(\vec{n})) = \sum_{\vec{i} \neq 0} \binom{\vec{n}}{\vec{i}}t_{\vec{i}}$$
    where $t_{\vec{i}}$ are vectors representing the coordinates of the degree $\vec{i}$ component of $g$ in Mal'cev coordinates. By hypothesis, we may select coordinates $u_{\vec{i}}$ such that
    $$|t_{\vec{i}} - u_{\vec{i}}| \le \frac{(ML/\delta)^{O_{k, \ell}(d^{O_{k, \ell}(1)})}}{\vec{N}^i}$$
    with $\eta_j \cdot u_{\vec{i}} \in \mathbb{Z}$ for each $\vec{i}$. By Cramer's rule, we may pick a rational vector $v$ with height at most $(dL)^{O(d)}$ such that $\eta_j \cdot v = \eta_j \cdot u_{\vec{i}}$. We define $\gamma$ and $\varepsilon$ via
    $$\psi(\gamma(\vec{n})) = \sum_{\vec{i} \neq 0} \binom{\vec{n}}{\vec{i}} v_{\vec{i}},~~\text{and}~\psi(\varepsilon(\vec{n})) = \sum_{\vec{i} \neq 0} \binom{\vec{n}}{\vec{i}}(t_{\vec{i}} - u_{\vec{i}}).$$
    We see from here that $\gamma(0) = \varepsilon(0) = \mathrm{id}_G$. Now let
    $$g_1(\vec{n}) := \varepsilon(\vec{n})^{-1}g(\vec{n})\gamma(\vec{n})^{-1}.$$
    We claim that this lies inside $\mathrm{poly}(\mathbb{Z}^\ell, \tilde{G})$. To see this, note that it lies inside $\mathrm{poly}(\mathbb{Z}^\ell, G)$, so it suffices to show that the image of $g_1$ lies inside $\tilde{G}$. To see this, we have that
    $$\eta_j(g_1(\vec{n})) = \sum_{\vec{i} \neq 0} \eta_{j} \cdot (u_{\vec{i}} - v_{\vec{i}}) \binom{\vec{n}}{\vec{i}} = 0.$$
\end{proof}

Recall the definition of $g_{\mathrm{lin}}$ and $g_2$ from Definition~\ref{d:gling2}.
\begin{lemma}[Factorization lemma II]\label{l:factorization2}
Let $G/\Gamma$ be a nilmanifold of step $s$, degree $k$, dimension $d$, and complexity $M$. Let $g \in \mathrm{poly}(\mathbb{Z}^\ell, G)$ and $\eta_1, \dots, \eta_r$ are a set of linearly independent nonzero horizontal characters on $G_2$ which annihilate $[G, G]$ and is size at most $L$. Suppose $\|\eta_j \circ g_2\|_{C^\infty[\vec{N}]} \le \delta^{-1}$ for each $i$. Then we have the factorization
$$g(\vec{n}) = \varepsilon(\vec{n})g_1(\vec{n})\gamma(\vec{n})$$
where
\begin{itemize}
    \item $\varepsilon$ is $((ML/\delta)^{O_{k, \ell}(d^{O_{k, \ell}(1)})}, \vec{N})$-smooth;
    \item $\gamma$ is $(ML)^{O_{k, \ell}(d^{O_{k, \ell}(1)})}$-rational;
    \item $g_1 \in \mathrm{poly}(\mathbb{Z}^\ell, \tilde{G})$ where $\tilde{G}$ is given the filtration
$\tilde{G}_0 = \tilde{G}_1 = G,~\tilde{G}_2 = \bigcap_{j = 1}^r \operatorname{ker}(\eta_j)$, and $\tilde{G}_i = \tilde{G}_2 \cap G_i$ for all $i \ge 2.$
\end{itemize}
Furthermore, if $g(0) = \mathrm{id}_G$, then we may take $\gamma(0) = \varepsilon(0) = \mathrm{id}_G$.
\end{lemma}

\begin{proof}
    As before, we first reduce to the case that $g(0) = \mathrm{id}_G$. We write $g'(n) = \{g(0)\}^{-1}g(n)[g(0)]^{-1}$, so we see that $g'(0) = \mathrm{id}_G$ and we may replace $g$ with $g'$ in the below analysis.\\\\
    Recall the conventions of $g_{\mathrm{lin}}$ and $g_2$ given in Definition~\ref{d:gling2}. We write in coordinates that
    $$\psi(g_2(\vec{n})) = \sum_{|\vec{i}| > 1} \binom{\vec{n}}{\vec{i}}t_{\vec{i}}$$
    where $t_{\vec{i}}$ are vectors representing the coordinates of the degree $\vec{i}$ component of $g$ in Mal'cev coordinates. By hypothesis, we may select coordinates $u_{\vec{i}}$ such that
    $$|t_{\vec{i}} - u_{\vec{i}}| \le \frac{(ML/\delta)^{O_{k, \ell}(d^{O_{k, \ell}(1)})}}{\vec{N}^i}$$
    with $\eta_j \cdot u_{\vec{i}} \in \mathbb{Z}$. By Cramer's rule, we may pick a rational vector $v$ with height at most $(dL)^{O(d)}$ such that $\eta_j \cdot v_{\vec{i}} = \eta_j \cdot u_{\vec{i}}$ for each $\vec{i}$ and the linear components of $v$ is zero. We define $\gamma$ and $\varepsilon$ via
    $$\psi(\gamma(\vec{n})) = \sum_{|\vec{i}| > 1} \binom{\vec{n}}{\vec{i}} v_{\vec{i}},~~\text{and}~\psi(\varepsilon(\vec{n})) = \sum_{|\vec{i}| > 1} \binom{\vec{n}}{\vec{i}}(t_{\vec{i}} - u_{\vec{i}}).$$
    Note that $\gamma(0) = \varepsilon(0) = \mathrm{id}_G$. Now let
    $$g_1(\vec{n}) := \varepsilon(\vec{n})^{-1}g(\vec{n})\gamma(\vec{n})^{-1};$$
    it suffices to prove that $g_1$ lies in $\mathrm{poly}(\mathbb{Z}^\ell, \tilde{G})$. Writing
    $$g(\vec{n}) = g_2(\vec{n})g(e_1)^{n_1} \cdots g(e_\ell)^{n_\ell}$$
    we see that
    $$\varepsilon(\vec{n})^{-1}g(\vec{n})\gamma(\vec{n})^{-1} = \varepsilon(\vec{n})^{-1}g_2(\vec{n})\gamma(\vec{n})^{-1}g(e_1)^{n_1} \cdots g(e_\ell)^{n_\ell}[g(e_1)^{n_1} \cdots g(e_\ell)^{n_\ell}, \gamma(\vec{n})^{-1}].$$
    The sequence $\vec{n} \mapsto [g(e_1)^{n_1} \cdots g(e_\ell)^{n_\ell}, \gamma(\vec{n})^{-1}]$ lies inside $\mathrm{poly}(\mathbb{Z}, \tilde{G})$. This is because $[G, G] \subseteq \tilde{G}$. In addition, we see that
    \begin{align*}
    \eta_j(\varepsilon(\vec{n})^{-1}g_2(\vec{n})\gamma(\vec{n})^{-1}) &= -\eta_j(\varepsilon(\vec{n})) + \eta_j(g_2(\vec{n})) - \eta_j(\gamma(\vec{n})) \\
    &= \sum_{|\vec{i}| > 1} \binom{\vec{n}}{\vec{i}} \eta_j \cdot (u_{\vec{i}} - t_{\vec{i}}) + \eta_j \cdot t_{\vec{i}} - \eta_j \cdot v_{\vec{i}} \\
    &= \sum_{|\vec{i}| > 1} \binom{\vec{n}}{\vec{i}} \eta_j \cdot (u_{\vec{i}} - v_{\vec{i}}) = 0
    \end{align*}
    Hence, the image of $g_1$ lies inside $\tilde{G}$. This completes the proof.
\end{proof}

We recall the conventions for $G \times_H G$ given by Definition~\ref{d:joining}. If $G/\Gamma$ is a nilmanifold equipped with a filtration $(G_i)_{i = 0}^\infty$, the proof of Theorem~\ref{t:mainresult1} studies $G^\square/\Gamma^\square$ where $G^\square := G \times_{G_2} G$ and $\Gamma^\square := \Gamma \times_{\Gamma \cap G_2} \Gamma$. We record various properties from \cite[Proposition 7.2, Lemma 7.4, Lemma 7.5]{GT12} regarding this group in the following lemma; we sketch a proof for convenience of the reader.
\begin{lemma}[Properties of $G^\square$]\label{l:annoyingcomplexitybounds}
The following properties hold for $G^\square$.
\begin{itemize}
    \item The sequence of subgroups $(G^\square)_i = G_i \times_{G_{i + 1}} G_i$ forms a filtration of $G^\square$.
    \item If $G$ has complexity $M$, then $G^\square$ has complexity at most $M^{O_k(d^{O(1)})}$. 
    \item Let $F\colon G/\Gamma \to \mathbb{C}$ with $\|F\|_{\mathrm{Lip}(G/\Gamma)} \le L$. If $F^\square(x, y) = F(gx)\overline{F}(y)$ where $|\psi_G(g)| \le 1/2$, then $\|F^\square\|_{\mathrm{Lip}(G^\square/\Gamma^\square)} \le (ML)^{O_k(d^{O(1)})}$. Furthermore, if $F$ is a $G_k$-character with vertical frequency $\xi$, then $F$ invariant under $G_k^\triangle$.
    \item If $w \in \Gamma^\square$ uniquely decomposes as $(u, vu)$ with $u \in \Gamma$ and $v \in \Gamma \cap G_2$, and if $|w| \le Q$, then $|u|, |v| \le (MQ)^{O_k(d^{O(1)})}$.  
    \item Each $\eta$ on $G^\square$ may be decomposed uniquely as
    $$\eta(g', g) = \eta_1(g) + \eta_2(g'g^{-1})$$
    where $\eta_1$ is a horizontal character on $G$, $\eta_2$ is a horizontal character on $G_2$ which annihilates $[G, G_2]$. Furthermore, if $|\eta|$ is bounded by $K$, then $|\eta_1|, |\eta_2|$ are bounded above by $(KM)^{O_k(d^{O(1)})}$.
\end{itemize}
\end{lemma}
\begin{proof}
    For the first point, note that if $(g_i, g_{i + 1}g_i)$ and $(h_i, h_{i + 1}h_i)$ are elements in $(G \times_{G_2} G)_i$ and $(G \times_{G_2} G)_j$, respectively, then 
    $$[(g_i, g_{i + 1}g_i), (h_i, h_{i + 1}h_i)] = ([g_i, h_i], [g_{i + 1}g_i, h_{i + 1}h_i])$$
    Baker-Campbell-Hausdorff immediately give that
    $$[g_{i + 1}g_i, h_{i + 1}h_i] = [g_i, h_i] \pmod{G_{i + j + 1}}.$$
    Hence, $(G^\square)_i$ forms a filtration. To show the second point, denoting $\{X_1, \dots, X_d\}$ as the Mal'cev basis, consider
    $$\{(X_1, 0), (0, X_1), \dots, (X_d, 0), (0, X_d)\}.$$
    This is a Mal'cev basis for $G/\Gamma \times G/\Gamma$, and by Cramer's rule, $G^\square$ is $(dM)^{O(d)}$-rational with respect to this basis. By Lemma~\ref{l:Subgroupfiltration}, it follows that there exists a Mal'cev basis on $G \times_{G_2} G$ which is an $(dM)^{O_k(d^{O(1)})}$-rational combination of $(X_i, X_j)$. For the third point, note that $F$ restricted to $G/\Gamma \times G/\Gamma$ has Lipschitz constant $3L^2$. We see that $G \times_{G_2} G$ has rationality $(dM)^{O_k(d^{O(1)})}$, so if $x, y \in G^\square$, then $d_{G/\Gamma \times G/\Gamma}(x, y) \le (dM)^{O_k(d^{O(1)})}d_{G \times_{G_2} G} (x, y)$. The third point follows from $F(g_kx\Gamma)\overline{F}(g_{k}y\Gamma) = F(x\Gamma)\overline{F}(y\Gamma)$ where $g_k \in G_{k}$. For the fourth point, by Lemma~\ref{l:Subgroupfiltration} applied to $G \times G$ and $G^\square$, it follows that if $w$ is $Q$-bounded with respect to the Mal'cev basis on $G \times_{G_2} G$, then it is $(MQ)^{O_k(d^{O(1)})}$-bounded with respect to the product Mal'cev basis on $G \times G$. The result follows by projecting to each coordinate and using Lemma~\ref{l:multiplication} to show that the Mal'cev coordinates of the product of two elements with bounded Mal'cev coordinates is bounded. Finally, for the fifth point, we define $\eta_1(g) = \eta(g, g)$ and $\eta_2(h) = \eta(h, \mathrm{id}_G)$. Since $\eta$ annihilates $[G \times_{G_2} G, G \times_{G_2} G]$ which contains $[G_2 \times \mathrm{id}_G, G^\triangle] = [G_2, G] \times \mathrm{id}_G$, we see that $\eta_2$ must annihilate $[G, G_2]$ and $\eta_1$ must annihilate $[G, G]$. We also see that since $\eta(\Gamma^\square) \subseteq \mathbb{Z}$, this contains both $\Gamma^\triangle$ and $(\Gamma \cap G_2) \times \mathrm{id}_G$, so $\eta_1(\Gamma) \subseteq \mathbb{Z}$ and $\eta_2(\Gamma \cap G_2) \subseteq \mathbb{Z}$. \\\\
    To check the boundedness conditions, we see that the Mal'cev coordinates of $G \times_{G_2} G$ are rational combinations of $G \times G$ with coefficients that have denominator at most $(dM)^{O_k(1)}$. It follows that $\eta_1, \eta_2$ are bounded by $K(dM)^{O_k(d^{O(1)})}$.
\end{proof}
We also need the following lemma about polynomial sequences on $G \times_{G_2} G$, where we recall from Definition~\ref{d:joining} and Definition~\ref{d:gling2} the conventions for $G \times_H G$ and $g_{\mathrm{lin}}$ and $g_2$. The proof follows \cite[Proposition 4.2]{GT14}.
\begin{lemma}\label{l:vandercorputpolynomial}
Given a polynomial sequence $g$ in $\mathrm{poly}(\mathbb{Z}^\ell, G)$ with $g(0) = \mathrm{id}_G$, and $\vec{h} \in \mathbb{Z}^\ell$, we can define a polynomial sequence
$$g_{\vec{h}}(\vec{n}) := (\{g_{\mathrm{lin}}(\vec{h})\}^{-1}g_2(\vec{n} + \vec{h})g_{\mathrm{lin}}(\vec{n} + \vec{h})[g_{\mathrm{lin}}(\vec{h})]^{-1}, g(\vec{n}))$$
where $g_{\mathrm{lin}}(\vec{n}) = g(e_1)^{n_1}g(e_2)^{n_2} \cdots g(e_\ell)^{n_\ell}$ and $g_2(\vec{n}) = g(\vec{n})g_{\mathrm{lin}}(\vec{n})^{-1}$. Then $g_{\vec{h}}(\vec{n}) \in \mathrm{poly}(\mathbb{Z}^\ell, G^\square)$. 
\end{lemma}
\begin{proof}
Since $g(\vec{n}) = g_{\mathrm{lin}}(\vec{n}) \pmod{G_2}$ and $g_{\mathrm{lin}}(\vec{n} + \vec{h}) = g_{\mathrm{lin}}(\vec{n})g_{\mathrm{lin}}(\vec{h}) \pmod{G_2}$, it follows that each element of $g_{\vec{h}}$ lies inside $G^\square$. In addition, by conjugating by $(\{g_{\mathrm{lin}}(\vec{h})\}^{-1}, \mathrm{id}_G)$, and noting that $[G_k \times_{G_{k + 1}} G_k, G \times \mathrm{id}_G] \subseteq G_{k + 1} \times \mathrm{id}_G$, it follows that it suffices to show that
$$(g_2(\vec{n} + \vec{h})g_{\mathrm{lin}}(\vec{n} + \vec{h})g_{\mathrm{lin}}(\vec{h}), g(\vec{n}))$$
is a polynomial sequence on $G^\square$. It suffices to show that
$$(g_2(\vec{n} + \vec{h}), g_2(\vec{n})), (g_{\mathrm{lin}}(\vec{n + h}) g_{\mathrm{lin}}(\vec{h})^{-1}, g_{\mathrm{lin}}(\vec{n}))$$
are polynomial sequences on $G^\square$. To show the first one, it suffices to show that
$$(x_i^{\binom{n_i + h_i}{i}}, x_i^{\binom{n_i}{i}})$$
for $x_i \in G_i$. This is obvious by Taylor expansion. To show that $(g_{\mathrm{lin}}(\vec{n} + \vec{h}) g_{\mathrm{lin}}(\vec{h}), g_{\mathrm{lin}}(\vec{n}))$ is a polynomial sequence on $G^\square$, we expand
$$(g_{\mathrm{lin}}(\vec{n} + \vec{h}) g_{\mathrm{lin}}(\vec{h}), g_{\mathrm{lin}}(\vec{n})) = (g(e_1)^{n_1 + h_1} \cdots g(e_\ell)^{n_\ell + h_\ell} g(e_\ell)^{-h_\ell} \cdots g(e_1)^{-h_1}, g(e_1)^{n_1} \cdots g(e_\ell)^{n_\ell}).$$
By conjugating by $(g(e_1)^{h_1}, 1)$ and factoring out $(g(e_1)^{n_1}, g(e_1)^{n_1})$, it suffices to show that
$$(g(e_2)^{n_2 + h_2} \cdots g(e_\ell)^{n_\ell + h_\ell} g(e_\ell)^{-h_\ell} \cdots g(e_2)^{-h_2}, g(e_2)^{n_2} \cdots g(e_\ell)^{n_\ell})$$
is a polynomial sequence on $G^\square$. Iterating this procedure, we have the desired result.
\end{proof}

The next lemma is a standard Fourier expansion lemma for the torus; we include it for completeness.
\begin{lemma}[Fourier expansion lemma]\label{l:FourierExpansion}
There exists a constant $C \ge 1$ such that the following holds. Let $f\colon \mathbb{T}^d \to \mathbb{C}$ be an $L$-Lipschitz function. Then for each $x \in \mathbb{T}^d$,
$$\Bigg|f(x) - \sum_{n_i\in \mathbb{Z}^d: |n_i| \le (C\delta)^{-2d- 1}} a_i e(n_i x)\Bigg| \le 3Ld\delta \text{ and } \sum_{i} |a_i| \le (\delta/CL)^{-3d^2}.$$
\end{lemma}
\begin{proof}
We may reduce to the case of $\|f\|_{\mathrm{Lip}(\mathbb{T}^d)} \le 1$ by dividing by $L$. Let $\phi\colon \mathbb{R} \to \mathbb{R}_{\ge 0}$ be a smooth, one-bounded function supported in $(-1, 1)$ and with integral one. Let $Q_\delta(x) = \prod_{i = 1}^d \delta^{-1}\phi(x_i/\delta)$ and let $K = Q_\delta * Q_\delta$. We may identify $\mathbb{T}$ with the interval $[-1/2, 1/2)$ and thus, $Q_\delta$ defines a smooth function on $\mathbb{T}^d$ for $\delta < 1/4$. Since $|\hat{\phi}(\xi)| \ll_k |\xi|^{-k}$ for any $\xi \in \mathbb{Z}^d$, we have
$$|\hat{K}(\xi)| \le C^d \delta^{-2d} \prod_{i = 1}^d (1 + |\xi_i|^2)^{-1}$$
so
$$\sum_{k \in \mathbb{Z}^d, |k| \ge M} |\hat{K}(\xi)| \le C_1^dM^{-1}\delta^{-2d}$$
for a constant $C_1 \ge 1$. Setting, 
$$f_{\mathrm{Approx}}(x) = \sum_{k \in \mathbb{Z}^d: |k| \le M} \hat{f}(k) \hat{K}(k) e(k \cdot x),$$
we have by Fourier inversion that if $M \ge (C_1/\delta)^{3d}$,
$$\|f_{\mathrm{Approx}} - f * K\|_\infty \le \delta.$$
Thus, it suffices to estimate $\|f - f * K\|_{\infty}$. We have
$$\|f - f * K\|_\infty \le \sup_{x \in \mathbb{T}^d}\int |f(x) - f(y)|K(x - y) dy \le \sum_{i = 1}^d  \int \|z_i\|_{\mathbb{R}/\mathbb{Z}} K(z) dz \le 2d\delta$$
since $K$ has integral one and is supported on $\{x \in \mathbb{T}^d: \|x_i\|_{\mathbb{R}/\mathbb{Z}} \le 2\delta \text{ for all } 1 \le i \le d\}$. It follows that $\|f_{\mathrm{Approx}} - f\|_\infty \le 3d\delta$. The sum of the Fourier coefficients of $f_{\mathrm{Approx}}$ is at most $(C/\delta)^{3d^2}$ and multiplying everything by $L$ gives the desired.
\end{proof}
This allows one to Fourier expand functions on $G/\Gamma$ to functions with a vertical frequency as we show in the below lemma.
\begin{lemma}[Expansion into functions with vertical frequency]\label{l:nilcharacters}
Let $G/\Gamma$ be a nilmanifold with dimension $d$, complexity $M$, degree $k$, and step $s$. Let $0 < \delta < 1/100$, and $F\colon G/\Gamma \to \mathbb{C}$ be $L$-Lipschitz. Given a $Q$-rational subgroup $H \subseteq Z(G)$, we have an approximation
$$\sup_{x \in G} \left|F(x\Gamma) - \sum_{|\xi| \le (\delta/(QL))^{-O_k(d^{O(1)})}} F_\xi(x\Gamma)\right| \le \delta$$
where $F_\xi$ is an $H$-vertical character of frequency $\xi$ and Lipschitz norm at most $(\delta/(QML))^{-O_k(d)^{O(1)}}$.
\end{lemma}
\begin{remark}
$G_{(s)}$ is easily seen to be a $M^{O_k(d^{O(1)})}$-rational subgroup of $G$.
\end{remark}
\begin{proof}
We first divide $F$ by $L$ to assume that it has Lipschitz norm $1$. Next, we pick a Mal'cev basis for $H/\Gamma_H$ where $\Gamma_{H} = \Gamma \cap H$ and set $d_H = \mathrm{dim}(H)$. We may pick a Mal'cev basis $\mathcal{Y}$ of $H$ to be an $(QM)^{O_k(d^{O(1)})}$-linear combinations of the Mal'cev basis of $G/\Gamma$. Thus, under that choice of Mal'cev basis, it follows that we may represent $\xi$ as an integer vector $k$ in $\mathbb{Z}^{d_H}$ and with the property that, then by Lemma~\ref{l:distancecomparison}, $\|k\|_{\infty}$ and $|\xi|$ agree up to a factor of $(QM)^{O_k(d^{O(1)})}$. Let $|\xi|_{H/\Gamma_{H}} := \|k\|_\infty$. Let $K$ denote the kernel constructed in Lemma~\ref{l:FourierExpansion} adapted to $H$ via $\mathcal{Y}$. We have thus identified $H/\Gamma_H$ as $\mathbb{T}^{d_H}$ via $\mathcal{Y}$. Now define
$$\tilde{F}(x\Gamma) = \int_{H/\Gamma_H} F(h x\Gamma) K(h) dh.$$
As in Lemma~\ref{l:FourierExpansion}, we have
$$\sum_{\xi \in \mathbb{Z}^{d_H}, |\xi| \le P} |\hat{K}(\xi)| \le C_1^dP^{-1}\delta^{2d}(QM)^{O_k(d^{O(1)})}$$
for some constant $C_1 > 1$. Setting $M_1 = C^{-d} \delta^{-2d - 1}(QM)^{O_k(d^{O(1)})}$ and
$$\hat{F}(\xi)(x\Gamma) = \int_{H/\Gamma_H} F(h x\Gamma) e(-\xi(h)) dh$$
$$F_{\mathrm{Approx}}(x\Gamma) = \sum_{|\xi| \le M_1} \hat{F}(\xi)(x\Gamma) \hat{K}(k) e(\xi(x)).$$
By Fourier inversion, we have
$$\|\tilde{F} - F_{\mathrm{Approx}}\|_\infty \le \delta.$$
It suffices to estimate $\|F - \tilde{F}\|_\infty$. We have
\begin{align*}
\|F - \tilde{F}\|_\infty &\le \sup_{x \in G} \int_{\mathbb{T}^{d_H}} |F(hx\Gamma) - F(x\Gamma)|K(h) dh \\
&\le (QM)^{O_k(d^{O(1)})}\sum_{i = 1}^{d_H} \int_{\mathbb{T}^{d_H}} \|h_i\|_{\mathbb{R}/\mathbb{Z}}K(g - h) dh \\
&\le (QM)^{O_k(d^{O(1)})}\delta
\end{align*}
since $\mathrm{sup}(K) \subseteq \{x \in H/\Gamma_H \cong \mathbb{T}^{d_H} : \|x_i\|_{\mathbb{R}/\mathbb{Z}} \le 2\delta \text{ for all } 1 \le i \le d_H\}$. Thus, $\|F - F_{\mathrm{Approx}}\|_\infty \le (QM)^{O_k(d^{O(1)})}\delta$ and the result follows.
\end{proof}
The following linear algebraic lemma is used frequently.
\begin{lemma}[Corollary of Cramer's rule]\label{l:Cramer}
    Let $v_1, \dots, v_r\in \mathbb{Z}^d$ be linearly independent vectors of size at most $M \ge 2$. Then there exists $\eta_1, \dots, \eta_{d - r} \in \mathbb{Z}^d$ of size at most $(dM)^{O(d)}$ such that $v_1, \dots, v_r, \eta_1, \dots, \eta_{d - r}$ span $\mathbb{R}^d$ and $\langle v_i, \eta_j \rangle = 0$ for all $i$ and $j$.
\end{lemma}
\begin{proof}
%Applying the Gram-Schmidt process, we may replace $v_1, \dots, v_r$ with orthogonal vectors that 
Let $e_1, \dots, e_d$ be the unit coordinate vectors in $\mathbb{R}^d$. Then there exists a subset, say, $E = \{e_{j_1}, \dots, e_{j_{d - r}}\}$ such that $\text{span}(E) \oplus \text{span}(v_1, \dots, v_r) = \mathbb{R}^d$. Let $A$ be the matrix whose rows are $v_1, \dots, v_r$ and $e_{j_1}, \dots, e_{j_{d - r}}$. Then in the matrix $A^{-1}$ has columns that are linearly independent, and letting $\eta_1, \dots, \eta_{d - r}$ be the last $d - r$ columns, we have $\langle v_i, \eta_j \rangle = 0$. Multiplying $\eta_j$ by some integer bounded by $(dM)^{O(d)}$ gives the result.
\end{proof}
A corollary of this is the following.
\begin{lemma}\label{l:Cramer2}
Let $G/\Gamma$ be a complexity $M$, degree $k$, and dimension $d$ nilmanifold equipped with a Mal'cev basis $(X_i)_{i = 1}^d$. Let $H$ be a $Q$-rational subgroup of $G$ containing $[G, G]$. 

Let $w_1, \dots, w_r$ be elements in $\Gamma$ with size at most $Q$ and let $\tilde{w_i}$ be the projection of $w_i$ to $G/H$. Then the subspace of horizontal characters that annihilate $\tilde{w_1}, \dots, \tilde{w_r}$ (when embedded in the Lie algebra of $\mathfrak{g}$ via the Mal'cev basis) has rationality at most $(MQ)^{O_k(d^{O(1)})}$.
\end{lemma}
\begin{proof}
Let $\psi_{H} : G \to G/H$ denote the quotient map and for $v \in \mathfrak{g}$, let $\overline{v}$ denote the image of $v$ under the induced map between the Lie algebras of $G$ and $G/H$. Since $(X_i)$ span the Lie algebra of $G$,  $(\overline{X_i})$ span the Lie algebra of $G/H$. Pick a subset of $(\overline{X_i})$ which forms a basis for the Lie algebra of $G/H$. Without a loss of generality, we may assume it is of the form $(\overline{X_i})_{i = 1}^r$. The dual space of $G/H$ can be identified with itself via the choice of the basis. 

Since $G/\Gamma$ has complexity $M$ with respect to the Mal'cev basis, it follows that $\langle \tilde{w_1}, \dots, \tilde{w_r} \rangle$ has rationality $(MQ)^{O_k(d^{O(1)})}$. By Lemma~\ref{l:Cramer}, it follows that we may find elements $\eta_1, \dots, \eta_{r*}$ with integer coordinates of at most $(MQ)^{O_k(d^{O(1)})}$ which form a basis for the annihilators of that subspace. Since $G/\Gamma$ has complexity $M$, $\eta_i$ also has size at most $(MQ)^{O_k(d^{O(1)})}$.
\end{proof}

\begin{lemma}\label{l:lowerstep}
    Let $G/\Gamma$ be an $s$-step nilmanifold of dimension $d$ and degree $k$ equipped with a Mal'cev basis of complexity at most $M$. Let $G'$ be a subgroup of $G$ with rationality at most $Q$ and such that $G'$ is nilpotent of step at most $s - 1$. Then there exists some $r \le d$ and linearly independent horizontal characters $\eta_1, \dots, \eta_r$ with size bounded by $(QM)^{O_k(d^{O_k(1)})}$ of $G$ such that $\eta_i(G') = 0$, and if $w_1, \dots, w_s$ are elements of $G$ with $\eta_i(w_j) = 0$ for each $i$ and $j$, then $[w_1, w_2, \dots, w_{s - 1}, w_s] = 0$.
\end{lemma}
\begin{proof}
    This is a consequence of the Lie algebra of $G'$ being a subspace of the Lie algebra of $G$. Since $G'$ is $Q$-rational, it follows by Lemma~\ref{l:Cramer2} that there exists linearly independent horizontal characters $\eta_1, \dots, \eta_r$ of size at most $Q^{O_k(d^{O(1)})}$ such that
    $$\psi_{\operatorname{horiz}}\left(\bigcap_{i = 1}^r \mathrm{ker}(\eta_i)\right) = \psi_{\operatorname{horiz}}(G').$$
    Furthermore, given elements $w_1, \dots, w_s$ inside $\bigcap \mathrm{ker}(\eta_i)$, we see that the expression
    $$[w_1, w_2, \dots, w_{s - 1}, w_s]$$
    only depends on the horizontal components of $w_i$, so we may replace them with elements inside $G'$, in which case the expression above is zero. 
\end{proof}

We also require the following lemmas regarding $C^\infty[\vec{N}]$; this is \cite[Lemma 2.1]{GT12}.
\begin{lemma}[Changing representation of $C^\infty$ norm]\label{l:cinfinity}
There exists a nonzero integer $q = O_k(1)$ such that if $p\colon \mathbb{Z}^\ell \to \mathbb{R}$ is $p(n) = \sum_{0 \le |\vec{i}| \le k} \beta_{\vec{i}} n^{\vec{i}}$
$$\|q\beta_j\|_{\mathbb{R}/\mathbb{Z}} \ll_k N^{-j}\|p\|_{C^\infty[\vec{N}]}$$
for all $j \neq 0$.
\end{lemma}
\begin{proof}
Writing $p(n) = \sum_{\vec{j}} \alpha_{\vec{j}} \binom{n}{\vec{j}}$ it follows from a simple Taylor expansion via the discrete difference operator that each $\beta_{\vec{j}}$ can be written as a rational combination of $(\alpha_{\vec{k}})_{\vec{k} \ge \vec{j}}$ with height at most $O_k(1)$; here $\vec{k} \ge \vec{j}$ denotes that each component of $k_p \ge j_p$ for each $p \in [\ell]$. The result follows via clearing denominators.
\end{proof}
We also require \cite[Lemma 2.3]{GT14}.
\begin{lemma}[Multivariate polynomial Vinogradov-type lemma]\label{l:vinogradov}
Suppose $g\colon \mathbb{Z}^\ell \to \mathbb{R}$ is a polynomial of degree $k$ such that $\|g(\vec{n})\|_{\mathbb{R}/\mathbb{Z}} \le \varepsilon$ for $\delta N_1 \cdots N_\ell$ many elements in $[\vec{N}]$ where $\varepsilon < \delta/10$. Then there is some integer $Q \ll \delta^{-O_{k, \ell}(1)}$ such that $$\|Qg\|_{C^\infty[\vec{N}]} \ll \delta^{-O_{k, \ell}(1)}\varepsilon \text{ and } \|Qg(0)\|_{\mathbb{R}/\mathbb{Z}} \ll \delta^{-O_{k, \ell}(1)}\varepsilon.$$
\end{lemma}
We finally require the following slight extension of \cite[Claim 5.8]{PSS23}.
\begin{lemma}[Multiparameter extrapolation]\label{l:multiparameterextrapolation}
    Suppose that $\ell, Q, N_1, \dots, N_\ell, k$ are integer parameters. Let $(b_i)_{i = 1}^t$ be integers of absolute value at most $Q$ with $b_i \neq 0$ and $a_i \in \mathbb{Z}, i = 1, \dots, t$ with $|a_i| \le QN_i$. Let $p\colon \mathbb{Z}^\ell \to \mathbb{R}/\mathbb{Z}$ be a polynomial map of degree at most $k$ and write
    $$\tilde{p}(\vec{n}) = p(a_1 + b_1n_1, a_2 + b_2n_2, \dots, a_\ell + b_\ell n_\ell).$$
    Then there exists $q \in \mathbb{Z}$ with $|q| \ll_{\ell, k} Q^{O_{\ell, k}(1)}$ such that
    $$\|qp\|_{C^\infty[\vec{N}]} \ll_{\ell, k} Q^{O_{\ell, k}(1)}\|\tilde{p}\|_{C^\infty[\vec{N}]}.$$
\end{lemma}
\begin{proof}
     Pick $a_i' \in [b_i]$ such that $a_i' \equiv a_i \pmod{b_i}$. Then $p(a_1' + b_1n_1, a_2' + b_2n_2, \dots, a_\ell' + b_\ell n_\ell) = \tilde{p}\left(n_1 - \frac{a_i' - a_i}{b_i}, \dots, n_\ell - \frac{a_\ell' - a_\ell}{b_\ell}\right)$. Vandermonde's identity implies that
    $$\binom{n_1 + I_1}{j_1} \cdots \binom{n_\ell + I_\ell}{j_\ell} = \sum_{\substack{i \in [\ell] \\ 0 \le t_i \le j_i}} \prod_{i = 1}^\ell \binom{n_i}{t_i}\binom{I_i}{j_i - t_i}.$$
    Thus, since $\frac{a_i' - a_i}{b_i} \le QN_i$, it follows that 
    $$\left\|\tilde{p}\left(n_1 - \frac{a_i' - a_i}{b_i}, \dots, n_\ell - \frac{a_\ell' - a_\ell}{b_\ell}\right)\right\|_{C^\infty[\vec{N}]} \le Q^{O_k(1)} \|\tilde{p}\|_{C^\infty[\vec{N}]}.$$
    Thus,
    $$\|p(a_1' + b_1n_1, a_2' + b_2n_2, \dots, a_\ell' + b_\ell n_\ell)\|_{C^\infty[\vec{N}]} \le Q^{O_k(1)}\|\tilde{p}\|_{C^\infty[\vec{N}]}.$$
    We may then apply \cite[Lemma 8.4]{GT12} to find an appropriate $q$ to finish.
\end{proof}

\section{Quantitative rationality of nilmanifolds}\label{s:quantnil}
The purpose of this section is to quantify \cite[Appendix A]{GT12} and obtain bounds which are single exponential in the dimension of the nilmanifold. We say the dimension of $G/\Gamma$ is $d$ and degree is $k$ and step is $s$. Throughout this section, we will assume $Q \ge 2$. \\\\
The primary difference between these proofs presented here is the presence of a degree $k$ nesting condition; this ensures that a given basis certifies the nilpotence of a Lie algebra and is required to to guarantee bounds single exponential in dimension. In addition, the proof in \cite[Proposition A.9]{GT12} is slightly incorrect; we correct the proof in Lemma~\ref{l:ConstructingMalcev}. \\\\
We first require the definition of a Mal'cev basis for the first kind and the basis of a second kind.
\begin{definition}
Let $\mathcal{X} = \{X_1, \dots, X_d\}$ be a basis for $\mathfrak{g}$. If $g = \exp(t_1X_1 + \cdots + t_d X_d)$, then we say that $(t_1, \dots, t_d)$ are coordinates of the first kind for $g$ and we write this as $\psi_{\mathcal{X}, \exp}(g) = (t_1, \dots, t_d)$. If 
$$g = \exp(u_1X_1) \cdots \exp(u_d X_d)$$
then we say that $(u_1, \dots, u_d)$ are the coordinates of the second kind for $g$ with respect to $\mathcal{X}$. We write this as $(u_1, \dots, u_d) = \psi_{\mathcal{X}}(g)$.
\end{definition}
We next require the definition of rationality and degree $k$ nesting property of a basis.
\begin{definition}
$\mathcal{X}$ is said to be $Q$-rational if all the structure constants $c_{ij\ell}$ in
$$[X_i, X_j] = \sum_\ell c_{ij\ell} X_\ell$$
are rationals of height at most $Q$. We say that $\mathcal{X}$ has the nesting property if
$$[\mathfrak{g}, X_i] \subseteq \text{Span}(X_{i + 1}, \dots, X_d).$$
We say that it has the \emph{degree $k$ nesting property} if there exists $d_1, \dots, d_k$ with $d_1 = 0$ and $d_k = d$ such that denoting $\mathfrak{g}_\ell = \text{Span}(X_{d_\ell + 1}, \dots, X_{d})$, then $[\mathfrak{g}, \mathfrak{g}_\ell] \subseteq \mathfrak{g}_{\ell + 1}$ and $[\mathfrak{g}, \mathfrak{g}_k] = 0$.
\end{definition}
We have the following three lemmas regarding rational bases, change of coordinates, and comparisons between the metric on a nilpotent Lie group and coordinates induced by these bases. These encapsulate a quantified version of \cite[Lemma A.2--A.4]{GT12}.
\begin{lemma}\label{l:transition}
We have the following quantitative change-of-coordinate relations.
\begin{itemize}
\item Let $\mathcal{X}$ be a basis for $\mathfrak{g}$ with the degree $k$ nesting property. Then $\psi \circ \psi_{\exp}^{-1}$ and $\psi_{\exp} \circ \psi^{-1}$ are both polynomials on $\mathbb{R}^m$ with degree $O_k(1)$. If $\mathcal{X}$ is $Q$-rational, then the coefficients of these polynomials are rational with height at most $Q^{O_k(d^{O(1)})}$.
\item Suppose $G' \subseteq G$ is a closed, connected, and simply connected subgroup of dimension $d'$ with associated Lie algebra $\mathfrak{g}'$. Suppose $\mathcal{X}'$ is a basis of $\mathfrak{g}'$ with the degree $k$ nesting property. Then $\psi \circ (\psi')^{-1}$ and $\psi' \circ \psi^{-1}$ are polynomial maps. If $\mathcal{X}$ and $\mathcal{X}'$ are both $Q$-rational, and each element in $\mathcal{X}'$ is a $Q$-linear combination of elements in $\mathcal{X}$, then all polynomials are degree $O_k(1)$ and have coefficients with height $Q^{O_k(d^{O(1)})}$.
\end{itemize}
\end{lemma}
\begin{proof}
\begin{itemize}
\item We first handle the first item. Let $g = \exp(u_1X_1) \cdots \exp(u_dX_d)$ be a generic element of $G$. Let
$$g = \exp(\sum_i t_i X_i)$$
and note that by the Baker-Campbell-Hausdorff formula that
\begin{align*}
t_1 &= u_1, t_2 = u_2, \dots, t_{d_1} = u_{d_1} \\
t_{d_1 + 1} &= u_{d_1 + 1} + P_{d_1 + 1}(u_1, \dots, u_{d_1}), \dots, t_{d_2} = u_{d_2} + P_{d_2}(u_1, \dots, u_{d_1}) \\
&\dots \\
t_d &= u_d + P_d(u_1, \dots, u_{d_{k - 1}})
\end{align*}
where each polynomial $P_\ell$ is of degree at most $s$ and coefficients are rational with height at most $Q^{O_k(d^{O(1)})}$. To prove this, we use the the generalized Baker-Campbell-Hausdorff-Dynkin formula (see \cite{Str87}). It states that if
$$\exp(z) = \exp(x_1)\exp(x_2)\exp(x_3) \cdots \exp(x_d),$$
then
$$z = \sum_{m = 1}^\infty \sum_{p_{j, \ell}} \frac{(-1)^{m - 1}}{m} \frac{(\operatorname{ad} x_d)^{p_{m, d}} \cdots (\operatorname{ad} x_1)^{p_{m, 1}} \cdots (\operatorname{ad} x_d)^{p_{1, d}} \cdots (\operatorname{ad} x_1)^{p_{1, 1}}}{(\sum_{j = 1}^m \sum_{\ell = 1}^d p_{j, \ell}) \prod_{j = 1}^m \prod_{\ell = 1}^d (p_{j, \ell}!)}$$ 
where $\operatorname{ad} x$ denotes the operator $[x, \cdot]$ with the convention that $(\operatorname{ad} x)^0$ acts as the identity, and $p_{j, k}$ range over all nonnegative integers for each $j = 1, \dots, m$ and $\ell = 1, \dots, d$ with $\sum_{\ell = 1}^d p_{j, \ell} > 0$. Observe that there are at most $s^3\binom{s + d - 1}{s} = d^{O_s(1)}$ many terms (where there are at most $s$ terms in the outermost sum, at most $s$ terms for $j = 1, \dots, m$, and for each $j$ at most $s\binom{s + d - 1}{s}$ many $p_{j, k}$ such that $\sum_{\ell = 1}^d p_{j, \ell} > 0$), and each involving a rational combination of elements of $\mathcal{X}$ with height at most $Q^{O_k(1)}$. This establishes the claim for $\psi_{\exp} \circ \psi^{-1}$. To establish the opposite claim, we simply invert this map. This gives
\begin{align*}
u_1 &= t_1, u_2 = t_2, \dots, \\
u_{d_1 + 1} &= t_{d_1 + 1} - P_{d_1 + 1}(t_1, \dots, t_{d_1}), \dots, \\
u_\ell &= t_\ell - P_\ell(t_1, t_2, \dots, t_{d_1}, t_{d_1 + 1} - P_{d_1 + 1}(t_1, \dots, t_{d_1}), \dots)
\end{align*}
Here, we see that each coefficient of each polynomial has height at most $Q^{O_k(1)}$, and has degree $O_k(1)$. This uses that the nesting of $P_j$'s has depth at most $k$; this completes the proof of the first item.
\item For the second item, we decompose
$$\psi \circ (\psi')^{-1} = (\psi \circ \psi_{\exp}^{-1}) \circ (\psi_{\exp} \circ (\psi_{\exp}')^{-1}) \circ (\psi_{\exp}' \circ (\psi')^{-1}).$$
The first and last map involves polynomials with coefficients of height at most $Q^{O_k(d^{O(1)})}$, absolute value at most $Q^{O_k(d^{O(1)})}$, and degree at most $O_k(1)$. For the middle map, note that if $g \in G'$, then we may write $g = \exp(t_1X_1' + \cdots + t_{d'} X_{d'})$. The middle map then consists of writing $X_i'$ as a $Q$-rational combination of elements of $(X_j)$'s and grouping the coefficients of $X_i$'s together. It follows that the middle map is a linear map with components of height at most $Q$. For the map $\psi' \circ \psi^{-1}$, we argue similarly, picking an orthogonal bases to $\mathcal{X}'$ and writing all elements in $\mathfrak{g}$ in terms of $\mathcal{X}$ and the orthogonal bases. The second point follows.
\end{itemize}
\end{proof}

\begin{lemma}\label{l:multiplication}
Let $\mathcal{X}$ be a basis for $\mathfrak{g}$ with the degree $k$ nesting property and letting $x, y \in G$ be $\psi(x) = t$, $\psi(y) = u$. We have
$$\psi(xy) = (t_1 + u_1, t_2 + u_2 + P_2(t_1, u_1), \dots, t_d + u_d + P_d(u_1, \dots, u_{d - 1}, t_1, \dots, t_{d - 1})$$
where $P_i$ has degree at most $O_k(1)$ and coefficients of height at most $Q^{O_k(d^{O(1)})}$ and absolute value at most $Q^{O_k(d^{O(1)})}$. Furthermore, the inverse map
$$\psi(x^{-1}) = (-t_1, -t_2 + \tilde{P}_2(t_1), \dots, -t_d + \tilde{P}_d(t_1, \dots, t_d))$$
where each $\tilde{P}_i$ has degree at most $O_k(1)$ and has coefficients of height at most $Q^{O_k(d^{O(1)})}$.
\end{lemma}
\begin{proof}
By Lemma~\ref{l:transition}, we have write
$$\psi_{\exp}(x) = (t_1', \dots, t_d')$$
$$\psi_{\exp}(y) = (u_1', \dots, u_d')$$
where $t_i' = t_i + R_{i}(t_1, \dots, t_{i - 1})$, $u_i' = u_i + R_i(u_1, \dots, u_{i - 1})$ with $R_i$ having degree $O_k(1)$, and whose coefficients have height at most $Q^{O_k(d^{O(1)})}$. Thus,
$$\psi_{\exp}(xy) = (t_1' + u_1', t_2' + u_2' + S_2(t_1', u_1'), \dots)$$
with the $i$th coordinate being of the form $t_i' + u_i' + S_i(t_1', \dots, t_{i - 1}', u_1', \dots, u_{i - 1}')$ where $S_i$ has degree $O_k(1)$ and whose coefficients have height at most $Q^{O_k(d^{O(1)})}$. Thus, $\psi_{\exp}(xy)$ has $i$th coordinate of the form $t_i + u_i + T_i(t_1, \dots, t_{i - 1}, u_1, \dots, u_{i - 1})$ where $T_i$ has degree at most $O_k(1)$ and has coefficients rational with height at most $Q^{O_k(d^{O(1)})}$. Thus, inverting again by Lemma~\ref{l:transition}, we have 
$$\psi(xy) = (t_1 + u_1, t_2 + u_2 + P_2(t_1, u_1), \dots, t_d + u_d + P_d(u_1, \dots, u_{d - 1}, t_1, \dots, t_{d - 1})$$
of the desired form. To show the formula for $\psi(x^{-1})$, we simply take $\psi_{\exp}(x^{-1}) = (-t_1', \dots, -t_d')$ as before and then invert using Lemma~\ref{l:transition}.
\end{proof}
Below and throughout the appendix, we will often be working with two sets of bases $\mathcal{X} = \{X_1, \dots, X_d\}$ and $\mathcal{X}' = \{X_1', \dots, X_d'\}$ which are equipped with their distance functions, which will be denoted $d$ and $d'$ (see Definition~\ref{d:metricnilmanifold} for the definition of a metric given a basis), respectively. They will also have their respective Mal'cev coordinates of the first kind and second kind, denoted as $\psi = \psi_{\mathcal{X}}$, $\psi' = \psi_{\mathcal{X}'}$, $\psi_{\exp} = \psi_{\exp, \mathcal{X}}$, and $\psi_{\exp}' = \psi_{\exp, \mathcal{X}'}$. \\\\
The following proof differs very slightly from that of \cite[Lemma A.4]{GT12} since we must utilize the degree $k$ nesting property to obtain bounds single exponential in dimension.
\begin{lemma}\label{l:distancecomparison}
Suppose $Q \ge 2$. Suppose $\mathcal{X}$ and $\mathcal{X}'$ are two $Q$-rational bases for $\mathfrak{g}$, both satisfying the nesting condition of degree $k$. Suppose that each $X_i'$ is given by a $Q$-rational combination of $X_i$ and vice versa. Then for $x, y \in G$ with $|\psi'(x)| \le Q$ and $|\psi'(y)| \le Q$, we have
$$d(x, y) \ll_k Q^{O_k(d^{O(1)})}|\psi'(x) - \psi'(y)|$$
and for $x, y \in G$ with $d(x, \mathrm{id}_G), d(y, \mathrm{id}_G) \le Q$, we have
$$|\psi'(x) - \psi'(y)| \ll_k Q^{O_k(d^{O(1)})} d(x, y).$$
\end{lemma}
\begin{proof}
To prove the first inequality, we note that $d(x, y) \le |\psi(xy^{-1})|$. Writing $\psi'(x) = t$ and $\psi'(y) = u$, it follows that $\psi(xy^{-1})$ can be written as $(P_1(t, u), \dots, P_d(t, u))$ where $P_i$ has degree $O_k(1)$ and rational with height $Q^{O_k(d^{O(1)})}$. Each of these polynomials vanish when $t = u$, so by the factor theorem, $P_i(t, u) = \sum_j (t_j - u_j) R_{i, j}(t, u)$ with $R_{i, j}$ (by Taylor expanding) having coefficients bounded in absolute value by $Q^{O_k(d^{O(1)})}$. \\\\
To prove the second bound, we begin with the case of $\mathcal{X}' = \mathcal{X}$ and $y = \mathrm{id}_G$. We show that
$$|\psi(x)| \ll_k Q^{O_k(d^{O(1)})} d(x, \mathrm{id}_G)$$
whenever $d(x, \mathrm{id}_G) \le Q$. 
We define $\kappa(x, y) := \min(|\psi(xy^{-1})|, |\psi(yx^{-1})|)$. We claim that
$$|\psi(x) - \psi(y)| \le Q^{O_k(d^{O(1)})}\kappa(x, y)(1 + \kappa(x, y) + |\psi(y)|)^{O_k(1)}.$$
To prove this, if $\kappa(x, y) = |\psi(xy^{-1})|$, then set $z = xy^{-1}$, we have $\psi(x) - \psi(y) = \psi(zy) - \psi(y)$ so letting $t = \psi(z)$ and $u = \psi(y)$, it follows that $\psi(zy) - \psi(y)$ can be written as a polynomial in $t$ and $u$ which vanishes when $t = 0$. By the factor theorem, each component of the polynomial can be written as
$$\sum_j t_j R_{i, j}(t, u)$$
with $R_{i, j}$ having degree at most $O_k(1)$ and coefficients of absolute value at most $Q^{O_k(d^{O(1)})}$. Hence, we have
$$|\psi(x) - \psi(y)| \ll_k Q^{O_k(d^{O(1)})} |t|(1 + |t| + |u|)^{O_k(1)}$$
from which the claim follows. An analogous claim holds when $\kappa(x, y) = |\psi(yx^{-1})|$. This implies that if $\kappa(x, y) \le 1$ and $|\psi(y)| \le 1$, then
$$|\psi(x)| \le |\psi(y)| + C_k Q^{O_k(d)}\kappa(x, y).$$
Iterating this, we see that if $x_0, x_1, \dots, x_n$ are elements of $G$ with $x_0 = \mathrm{id}_G$, and $\kappa(x_0, x_1) + \cdots + \kappa(x_{n - 1}, x_n) \le C_k^{-1} Q^{-C_kd}$ then
$$|\psi(x_n)| \le C_k Q^{C_kd} (\kappa(x_0, x_1) + \cdots  + \kappa(x_{n - 1}, x_n)).$$
By definition of $d$, it follows that
$$|\psi(x)| \le C_kQ^{O_k(d^{O(1)})}d(x, \mathrm{id}_G)$$
whenever $d(x, \mathrm{id}_G) \le C_kQ^{O_k(d^{O(1)})}$. Since $d$ is right-invariant, we in fact have
$$|\kappa(x, y)| \le C_k Q^{O_k(d^{O(1)})} d(x, y)$$
whenever $d(x, y) \le C_k Q^{O_k(d^{O(1)})}$. By definition of the distance function, it suffices to show
$$|\psi(x_n)| \ll_k Q^{O_k(d^{O(1)})}(\kappa(x_0, x_1) + \cdots + \kappa(x_{n - 1}, x_n))$$
whenever $\kappa(x_0, x_1) + \cdots + \kappa(x_{n - 1}, x_n) \le 2Q$ and $x_0 = \mathrm{id}_G$. As $n$ may be arbitrarily large, we split $(x_0, \dots, x_n)$ into $Q^{O_k(d^{O(1)})}$ paths $(x_i, \dots, x_j)$ with
$$\kappa(x_i, x_{i + 1}) + \kappa(x_{i + 1}, x_{i + 2}) + \cdots + \kappa(x_{j - 1}, x_j) \le C_k^{-1}Q^{-C_k}$$
and possibly $Q^{O(1)}$ many singleton paths $(x_i, x_{i + 1})$ with $C_k^{-1}Q^{-C_k}\kappa(x_i, x_{i + 1}) \le 2Q$. Thus, there exists a path $y_1, \dots, y_r$ with $r \ll_k Q^{O_k(d^{O(1)})}$ such that $\kappa(y_i, y_{i + 1}) \ll_k Q^{O_k(d^{O(1)})}$. Let $g_i = y_iy_{i - 1}^{-1}$. We see that
$$x_n = g_r \cdots g_1.$$
We claim that if $g_1, \dots, g_r \in G$ are such that $|\psi(g_i)| \le t$, then
$$|\psi(g_1 \cdots g_r)| \ll_k (1 + t)^{O_k(1)} Q^{O_k(d^{O(1)})}.$$
To show this, we use induction on coordinates. The first $d_1$ coordinates are bounded by $rt$. If the first $d_\ell$ coordinates are bounded by $K$, then the coordinates of $d_\ell + 1, \dots, d_{\ell + 1}$ is bounded by $rt + Q^{O_k(d^{O(1)})} K^{O_k(1)}$. Thus,
$$|\psi(g_1, \cdots, g_r)| \ll_k (1 + t)^{O_k(1)} Q^{O_k(d^{O(1)})}.$$
This shows the desired inequality. We now deal with the case of $\mathcal{X}' = \mathcal{X}$ and $y$ arbitrary. If $d(x, \mathrm{id}_G) \le Q$ and $d(y, \mathrm{id}_G) \le Q$, then by Baker-Campbell-Hausdorff, it follows that $|\psi(xy^{-1})| \le_k Q^{O_k(d^{O(1)})}$ so $|\psi(xy^{-1})| \ll_k Q^{O_k(d^{O(1)})} d(xy^{-1}, \mathrm{id}_G) = Q^{O_k(d^{O(1)})} d(x, y).$ Using the fact that
$$|\psi(x) - \psi(y)| \ll_k Q^{O_k(d^{O(1)})} |t|(1 + |t| + |u|)^{O_k(1)},$$
the bound follows. Finally, for the case where $\mathcal{X}$ and $\mathcal{X}'$ are different, we have by the first direction and the case just proved that
$$d'(x, y) \ll_k Q^{O_k(d^{O(1)})}|\psi(x) - \psi(y)| \ll_k Q^{O_k(d^{O(1)})}d(x, y).$$
This implies that $d'(x, \mathrm{id}_G) \ll_k Q^{O_k(d^{O(1)})}$ and $d'(y, \mathrm{id}_G) \ll_k Q^{O_k(d^{O(1)})}$. Thus, we have
$$|\psi'(x) - \psi'(y)| \ll_k Q^{O_k(d^{O(1)})}d'(x, y) \ll_k Q^{O_k(d^{O(1)})} d(x, y)$$
as desired.
\end{proof}
We now get the following weak left invariance of a metric on $G$; this is analogous to \cite[Lemma A.5]{GT12}.
\begin{lemma}\label{l:leftinvariance}
Suppose that $Q \ge 2$ and $\mathcal{X}$ is a $Q$-rational basis for $\mathfrak{g}$ satisfying the degree $k$ nesting condition. Suppose that $g, x, y \in G$ are elements that satisfy $|\psi(x)|, |\psi(y)|, |\psi(g)| \le Q$. Then we have the bound
$$d(gx, gy) \ll_k Q^{O_k(d^{O(1)})}d(x, y).$$
\end{lemma}
\begin{proof}
We first claim that
$$|\psi(gzg^{-1})| \ll_k Q^{O_k(d^{O(1)})}(1 + |\psi(z)| + |\psi(g)|)^{O_k(1)}|\psi(z)|.$$
By Baker-Campbell-Hausdorff, $\psi(gzg^{-1})$ is a polynomial of degree $O_k(1)$ with $Q^{O_k(d^{O(1)})}$-rational coefficients in the coordinates of $v = \psi(g)$ and $w = \psi(z)$. This polynomial vanishes when $w = 0$, so by the factor theorem, we may write the polynomial as a product of coordinates of $w$ and another polynomial, also with degree $O_k(1)$ and rational coefficients of $Q^{O_k(d^{O(1)})}$. The claim follows. \\\\
Let $x_0, \dots, x_r$ be a path with $x_1 = x$ and $x_r = y$. Then
$$d(gx, gy) \le \sum_{i = 0}^{r - 1} d(gx_i, gx_{i + 1}) \le \sum_{i = 0}^{r - 1} |\psi(gx_ix_{i + 1}^{-1}g^{-1})|.$$
By applying the inequality we proved for $z = x_ix_{i + 1}^{-1}$, we can show the lemma, provided that we have the bound that $\min(|\psi(x_{i}x_{i + 1}^{-1})|, |\psi(x_{i + 1} x_i^{-1})|) \ll_k Q^{O_k(d^{O(1)})}$. Since $d(x, y) \ll_k Q^{O_k(d^{O(1)})}$, the paths we must consider in the distance function must satisfy that condition among consecutive elements.
\end{proof}
We next show that passing to rational subgroups does not substantially alter the metric. This is analogous to \cite[Lemma A.6]{GT12}.
\begin{lemma}\label{l:subgroupcomparison}
Suppose $G' \subseteq G$ is a closed, connected, and simply connected subgroup and $\mathcal{X}, \mathcal{X}'$ are bases for $\mathfrak{g}, \mathfrak{g}'$ which have the degree $k$ nesting property. Let $Q \ge 2$ and suppose that each $X_i'$ is a $Q$-rational combination of the $X_i$. Then
$$d'(x, y) \ll_k Q^{O_k(d^{O(1)})} d(x, y)$$
uniformly for all $x, y \in G'$ with $|\psi(x)|, |\psi(y)| \le Q$. Furthermore,
$$d(x, y) \ll_k Q^{O_k(d^{O(1)})} d'(x, y)$$
uniformly for all $x, y \in G'$ with $|\psi'(x)|, |\psi'(y)| \le Q$.
\end{lemma}

\begin{proof}
We now prove the first inequality. By Lemma~\ref{l:transition}, it follows that $\psi' \circ \psi^{-1}$ is a polynomial of degree $O_k(1)$ with coefficients at most $O_k(Q)^{O_k(1)}$. Furthermore, letting $\psi(z) = v$, $\psi'(z)$ as a polynomial in $v$ vanishes at $v = 0$. Thus, by the factor theorem, we may factor out components of $v$ out of the polynomial $\psi'(z)$ in $v$. Thus, we have
$$|\psi'(z)| \ll_k Q^{O_k(d^{O(1)})}(1 + |\psi(z)|)^{O_k(1)} |\psi(z)|.$$
It suffices to show that for any path $x_0, \dots, x_r$ from $x$ to $y$, that
$$|\psi(xy^{-1})| \ll_k Q^{O_k(d^{O(1)})} \sum_{i = 0}^{r - 1} \min(|\psi'(x_ix_{i + 1}^{-1})|, |\psi'(x_{i + 1}x_i^{-1})|).$$
Since $d'(x, y) \ll_k Q^{O_k(d^{O(1)})}$ (by Lemma~\ref{l:distancecomparison}), it follows that we may restrict ourselves to paths such that $\min(|\psi'(x_ix_{i + 1}^{-1})|, |\psi'(x_{i + 1}x_i^{-1})|) \ll_k Q^{O_k(d^{O(1)})}$. By the above inequality, we have
$$\min(|\psi'(x_ix_{i + 1}^{-1})|, |\psi'(x_{i + 1}x_i^{-1})|) \ll_k Q^{O_k(d^{O(1)})} \min(|\psi'(x_ix_{i + 1}^{-1})|, |\psi'(x_{i + 1}x_i^{-1})|).$$
Summing over $i$ yields the desired first inequality. To prove the second inequality, we argue analogously, using $\psi \circ (\psi')^{-1}$ instead.
\end{proof}
We next discuss factoring an element into a small part and a rational part. This appears as \cite[Lemma A.14]{GT12} and we include a proof for completeness.

\begin{lemma}\label{l:decomposition}
Let $\mathcal{X}$ be a Mal'cev basis with the nesting property on $G/\Gamma$. Suppose $g \in G$ and let $I = [a, a + 1)$ or equal to $(a, a + 1]$. Suppose that $g \in G$. Then we may write $g = \{g\}[g]$ in a unique way where $\psi(\{g\}) \in I^d$ and $[g] \in \Gamma$.
\end{lemma}
\begin{proof}
By Lemma~\ref{l:multiplication}, if $\psi(x) = (t_1, \dots, t_d)$ and $\psi(y) = (u_1, \dots, u_d)$, then
$$\psi(xy) = (t_1 + u_1, t_2 + u_2 + P_2(t_1, u_1), \dots, t_d + u_d + P_{d }(t_1, \dots, t_{d - 1}, u_1, \dots, u_{d - 1}))$$
where $P_i$ are polynomials of degree $O_k(1)$ and have coefficient of size at most $O_k(Q)^{O_k(1)}$ and height at most $O_k(Q)^{O_k(d^{O(1)})}$ which vanish when one of their arguments (among $\vec{t}$ or $\vec{u}$) is zero. The point is if we let
$$\psi(g) = (s_1, \dots, s_d)$$
then we can inductively add an element to $g$ on step $i$ such that (denoting $g_i$ to be the value of $g$ at step $i$),
$$\psi(g_i) = (s_1^i, \dots, s_d^i)$$
where $s_1^i, \dots, s_i^i \in I$. Supposing we can do this in step $i$, then we take an element $y$ such that it is Mal'cev coordinates are zero everywhere except in the $(i + 1)^{\mathrm{st}}$ coordinate, where it is the unique integer $b$ such that $s_{i + 1}^i + b$ lies inside $I$. Then by the nesting property, we can take
$$\psi(g_ib) = (s_1^{i + 1}, \dots, s_d^{i + 1})$$
which has the desired property. \\\\
Uniqueness follows from the fact that if $\psi(x\gamma)$ and $\psi(x)$ both lie inside $I^d$ with $\gamma \in \Gamma$, then by equating coefficients, we see that the first component of $\gamma$ is zero. We claim inductively that this implies that all components of $\gamma$ are zero. Suppose the first $i$ components of $\gamma$ are zero. Then letting
$$\psi(x) = (t_1, \dots, t_d)$$
$$\psi(\gamma) = (u_1, \dots, u_d),$$
it follows that $P_{i + 1}(t_1, \dots, t_i, u_1, \dots, u_i) = 0$. Thus, $u_{i + 1} = 0$.
\end{proof}
We now repeat \cite[Lemma A.15]{GT12}, giving that the metric on $G/\Gamma$ is nondegenerate.
\begin{lemma}\label{l:nondegen}
Let $\mathcal{X}$ be a $Q$-rational basis for $G/\Gamma$ with the degree $k$ nesting property. Let $d$ be the distance with respect to this Mal'cev basis. Suppose that $x, y\in G$ and $d(x\Gamma, y\Gamma) = 0$. Then $x$ and $y$ descend to the same point in $G/\Gamma$.
\end{lemma}
\begin{proof}
Since the distance $d$ on $G$ is right-invariant, it follows that
$$d(x\Gamma, y\Gamma) = \inf_{\gamma \in \Gamma} d(x, y\gamma).$$
It suffices to show this infimum is achieved. Thus, it suffices to prove each $M$, there are finitely many $\gamma \in \Gamma$ such that $d(x, y\gamma) \le M$. Letting $\gamma$ be such an element, we have (using right-invariance and the triangle inequality)
\begin{align*}
d(\gamma, \mathrm{id}_G) &= d(\gamma^{-1}, \mathrm{id}_G) \le d(\gamma^{-1}, y) + d(y, \mathrm{id}_G) \\
&= d(\mathrm{id}_G, y\gamma) + d(y, \mathrm{id}_G) \\
&\le d(x, y\gamma) + d(x, \mathrm{id}) + d(y, \mathrm{id}_G).
\end{align*}
The result follows.
\end{proof}
We next prove that nilmanifolds have bounded diameter. This is analogous to \cite[Lemma A.16]{GT12}.
\begin{lemma}\label{l:bounded}
Let $Q \ge 2$ and suppose $\mathcal{X}$ is $Q$-rational basis for $G/\Gamma$ with the degree $k$ nesting property. Then $d(x\Gamma, y\Gamma) \ll_k Q^{O_k(d^{O(1)})}$.
\end{lemma}
\begin{proof}
By Lemma~\ref{l:decomposition}, we may pick $\gamma$ and $\gamma'$ such that $|\psi(x\gamma)| \le \frac{1}{2}$ and $|\psi(y\gamma')| \le \frac{1}{2}$. The claim then follows from Lemma~\ref{l:distancecomparison}.
\end{proof}

We next need a lemma regarding the comparison between metrics between a nilmanifold and a rational subnilmanifold. This is analogous to \cite[Lemma A.17]{GT12}.
\begin{lemma}
Suppose that $G' \subseteq G$ is a closed, connected, and simply connected subgroup and that $\mathcal{X}$, $\mathcal{X}'$ are $Q$-rational Mal'cev bases (with the degree $k$ nesting property) for $G/\Gamma$ and $G'/\Gamma'$ respectively such that each $X_i'$ is a $Q$-rational combination of $X_i$. Let $d, d'$ be the metrics induced on $G/\Gamma$ and $G'/\Gamma'$, respectively. Then for any $x, y \in G'$,
$$d'(x\Gamma', y\Gamma') \ll Q^{O_k(d^{O(1)})} d(x\Gamma, y\Gamma)$$
$$d(x\Gamma, y\Gamma) \ll Q^{O_k(d^{O(1)})} d'(x\Gamma', y\Gamma').$$
\end{lemma}
\begin{proof}
We prove the second inequality first. First, we assume via Lemma~\ref{l:decomposition} that $|\psi'(x)|, |\psi'(y)| \le 1$. Since we showed in Lemma~\ref{l:nondegen} that the infimum in $d'(x\Gamma', y\Gamma')$ is achieved, there exists some $\gamma' \in \Gamma'$ such that
$$d'(x\Gamma', y\Gamma') = d'(x, y\gamma').$$
Thus, $d'(x, y\gamma') \ll Q^{O_k(d^{O(1)})}$, so $d'(\mathrm{id}_G, y\gamma') \ll Q^{O_k(d^{O(1)})}$ so $d(x, y\gamma') \ll Q^{O_k(d^{O(1)})}d'(x, y\gamma')$. This shows the second inequality. \\\\
Now we show the first inequality. Similar to above, we may assume that $|\psi(x)|, |\psi(y)| \le 1$ and that there exists some $\gamma \in \Gamma$ such that $d(x\Gamma, y\Gamma) = d(x, y\gamma)$. Let $\delta = Q^{-C_kd^2}$ for $C = C_k$ be a sufficiently large constant to be specified later. If $d(x, y\gamma) \ge \delta$, then by Lemma~\ref{l:bounded}, the first inequality follows. Thus, we may assume that $d(x, y\gamma) < \delta$. By Lemma~\ref{l:leftinvariance}, it follows that
$$d(z, \gamma) \ll Q^{O_k(d^{O(1)})}\delta$$
where $z = y^{-1}x$. Thus, $d(\gamma, \mathrm{id}_G) \ll Q^{O_k(d^{O(1)})}$. By Lemma~\ref{l:distancecomparison}, it follows that
$$|\psi(z) - \psi(\gamma)| \ll Q^{O_k(d^{O(1)})}\delta.$$
It follows from this and Lemma~\ref{l:transition} that
$$|\psi_{\exp}(z) - \psi_{\exp}(\gamma)| \ll Q^{O_k(d^{O(1)})}\delta.$$
This is because $\psi_{\exp}(z) - \psi_{\exp}(\gamma)$ can be written as a polynomial of $\psi(z) = t$ and $\psi(\gamma) = u$ that vanishes at $t = u$. Thus,
$$|\psi_{\exp}(z) - \psi_{\exp}(\gamma)| \ll Q^{O_k(d^{O(1)})} |\psi(z) - \psi(\gamma)| (1 + |\psi(z)| + |\psi(\gamma)|)^{O_k(1)}.$$
The subalgebra $\mathfrak{g}'$ is defined as the intersection of at most $d$ linear forms with rational coefficients of height $Q^{O(d)}$ (by Cramer's rule). For each $\gamma_1 \in \Gamma$, the coordinates of $\psi_{\exp}(\gamma_1)$ can be taken to be rationals with a fixed denominator of at most $Q^{O(d)}$. Let $\ell_1, \dots, \ell_j$ denote those linear forms. Without a loss of generality, we may assume that these linear forms have integer coefficients of size at most $Q^{O(d^2)}$. Thus, if for each $i$, $\ell_i(\gamma_1) \ll Q^{-O(d)}$, it follows that $\ell_i(\gamma_1) = 0$ for each $i$, or that $\gamma_1$ lies inside $G' \cap \Gamma = \Gamma'$. Thus, choosing $C$ to be sufficiently large, it follows that $\gamma$ must lie in $\Gamma'$. Thus, $d(x, y\gamma) \ll Q^{O_k(d^{O(1)})}$. By Lemma~\ref{l:subgroupcomparison}, it follows that $d'(x, y\gamma') \ll Q^{O_k(d^{O(1)})}d(x, y\gamma')$, from which the first inequality follows.
\end{proof}

We next turn to the construction of a Mal'cev basis given a weak basis. \cite[Definition A.7]{GT12} defines a weak basis of rationality $Q$ as a set $\{X_1, \dots, X_d\}$ of $\mathfrak{g}$ such that there exists some integer $q \le Q$ such that (with respect to the basis), $\frac{1}{q} \mathbb{Z}^d \supseteq \psi_{\exp}(\Gamma) \supseteq q\mathbb{Z}^d$, and the basis satisfies 
$$[X_i, X_j] = \sum_{k} a_{ijk}X_k$$
with $a_{ijk}$ having height at most $Q$. The next three lemmas quantify Green and Tao's construction given in \cite[Lemma A.8, Proposition A.9--A.10]{GT12}.
\begin{lemma}\label{l:weakbases}
Weak bases enjoy the following properties.
\begin{itemize}
\item Suppose $\mathcal{X}$ is a $Q$-rational weak basis for $G/\Gamma$ and $\mathcal{X}' = \{X_1', \dots, X_{d'}'\}$ is a basis for a $Q$-rational sub-Lie algebra $\mathfrak{g}_0$ with the property that each $X_i'$ is a $Q$-rational combination of $X_i$. Then $\mathcal{X}'$ is a $Q^{O(d^{O(1)})}$-rational weak basis for $G_0/\Gamma_0$ where $G_0 = \exp(\mathfrak{g}_0)$ and $\Gamma_0 = \Gamma \cap G_0$.
\item Suppose $\mathcal{X}$ is a Mal'cev basis of $G/\Gamma$ of complexity $Q$ adapted to some filtration $(G_i)$ of degree $k$. Then $\mathcal{X}$ is a $Q^{O_k(d^{O(1)})}$-rational weak basis for $G/\Gamma$.
\end{itemize}
\end{lemma}
\begin{proof}
We first extend $\{X_1', \dots, X_{d'}'\}$ to a basis $\{X_1', \dots, X_{d}'\}$ such that each $X_i'$ is a $Q^{O(d^{O(1)})}$-rational combination of $\mathcal{X}$. For the first part, we note by Cramer's rule (or by Gaussian elimination) that we can write each element in $\mathcal{X}$ as a $Q^{O(d^2)}$-rational combination of elements in $\mathcal{X}'$. Thus, if we may write an element $\gamma \in \Gamma_0$ as $\exp(\sum_i t_i X_i)$ we may write
$$\gamma = \exp(\sum_i t_i' X_i')$$
where $t_i'$ is a combination of a product of $t_j$ and a rational with height at most $Q^{O(d^{O(1)})}$. The point, though, is that since $\log(\gamma) \subseteq \mathfrak{g}_0$, for $d' + 1 < i \le d$, its $X_i'$ components are zero. By clearing denominators, we have thus established that there exists some $q \le Q^{O(d^{O(1)})}$ such that $q^{-1}\mathbb{Z}^{d'} \supseteq \psi_{\exp, \mathcal{X}'} \supseteq q\mathbb{Z}^{d'}$. In addition, we may write for $i, j \le d'$, $[X_i', X_j']$ as a $Q^{O(d^{O(1)})}$-rational combination of $\mathcal{X}$, which we can then in turn write as a $Q^{O(d^{O(1)})}$-rational basis of the extended basis $\{X_1', \dots, X_d'\}$. However, since $[X_i', X_j']$ actually lies inside $\mathfrak{g}_0$, it follows that it may actually be written uniquely as a $Q^{O(d^{O(1)})}$-rational combination of combination of $\mathcal{X}'$. The second part of the lemma follows from the fact that the map $\psi \circ \psi_{\exp}^{-1}$ has polynomial coordinates with coefficients of height at most $Q^{O_k(d^{O(1)})}$.
\end{proof}
\begin{lemma}\label{l:ConstructingMalcev}
Suppose $\mathcal{X}$ is $Q$-rational weak basis for $G/\Gamma$ and that $(G_i)$ is a sequence of nested subgroups in which each subgroup is $Q$-rational with $G_0 = G_1 \subseteq G$ and $[G_i, G_1] \subseteq G_{i + 1}$ and $G_{k + 1} = \{\mathrm{id}_G\}$. Then taking $\Gamma_i = \Gamma \cap G_i$, there exists a Mal'cev basis $\mathcal{X}' = \{X_1', \dots, X_d'\}$ for $G_0/\Gamma_0$ adapted to the sequence of subgroups in which each $X_i'$ is a $Q^{O_k(d^{O(1)})}$-rational combination of the basis elements $X_i$.
\end{lemma}
\begin{proof}
Take a basis of $\mathfrak{g}_d$ consisting of $Q$-rational linear combinations of $X_i$. By standard linear algebra, we may extend to a basis of $\mathfrak{g}_{d - 1}$, consisting of $Q$-rational linear combinations of the $X_i$. This can keep on be extended, obtaining a basis $\mathcal{Y} = \{Y_1, \dots, Y_d\}$. It follows that $\mathcal{Y}$ is a $Q^{O_k(d^{O(1)})}$-rational weak basis of $G_0/\Gamma_0$ satisfying a nesting property. We convert this basis $\mathcal{Y}$ into a desired basis $X_i'$ such that $X_i'$ is spanned by $Y_i, \dots, Y_d$ and so that
$$\exp(\text{Span}(Y_{i}, \dots, Y_d)) \cap \Gamma = \{\exp(n_i X_i') \cdots \exp(n_dX_d'): n_j \in \mathbb{Z}\}.$$
We show this inductively. By dimension counting, the quotient group
$$(\exp(\text{Span}(Y_j, \dots, Y_d)) \cap \Gamma) / (\exp(\text{Span}(Y_{j + 1}, \dots, Y_d)) \cap \Gamma_0)$$
(since $\mathcal{Y}$ satisfies the nesting property, both of the above are genuine groups with the second being a normal subgroup of the first) is generated by an element of the form
$$\overline{\exp\big(\sum_{\ell \ge j} c_\ell Y_\ell\big)}.$$
Let $q \le Q^{O_k(d^{O(1)})}$ be a nonzero integer so that $q\mathbb{Z}^d \subseteq \psi_{\exp, \mathcal{Y}}(\Gamma_0) \subseteq \frac{1}{q}\mathbb{Z}^d$. We claim that we can take $c_\ell$ for $\ell \ge j$ to have height at most $q^2$. Since $q\mathbb{Z}^d \subseteq \psi_{\exp}(\Gamma_0)$, it suffices to show that we can take $c_\ell$ to have absolute value at most $q$. \\\\
Fix some $\ell \ge j$. Suppose $(f_m)_{m = j}^d$ is a sequence of real numbers such that $f_j, \dots, f_\ell$ are all of absolute value at most $q$. We claim that there exists some $\gamma_{\ell+1} \in \Gamma_0 \cap \exp(\text{Span}(Y_\ell, \dots, Y_d))$ such that
$$\psi_{\exp, \mathcal{Y}}\big(\exp\big(\sum_{m \ge j} f_m Y_m\big) \gamma_{\ell+1}\big)$$
has coordinates $j, \dots, \ell+1$ that are of absolute value at most $q$. To do this, supposing this holds for $\ell$, select some integer $r \in q\mathbb{Z}$ such that $|d_{\ell + 1} - r| \le q$ and set $\gamma_{\ell + 1} = \exp(rY_{\ell + 1})$. Then by Baker-Campbell-Hausdorff (and using the fact that $\mathcal{Y}$ satisfies the nesting property), it follows that
$$\psi_{\exp, \mathcal{Y}}\big(\exp\big(\sum_{m \ge j} f_m Y_m\big) \gamma_{\ell + 1}\big)$$
has coordinates $j, \dots, \ell + 1$ that are of absolute value at most $q$. Now applying this claim iteratively to $(c_m)_{m = j}^d$, we obtain $\gamma_j, \dots, \gamma_d \in \Gamma_0$ such that 
$$\psi_{\exp, \mathcal{Y}}\big(\exp\big(\sum_{m \ge j} c_m Y_m\big)\gamma_j\gamma_{j + 1} \gamma_{j + 2} \dots \gamma_d\big) \in \{0\}^{j - 1} \times [-q, q]^{d - j + 1}.$$
In addition, since $\gamma_{j + 1}, \dots, \gamma_d$ lie inside $\exp(\text{span}(Y_{j + 1}, \dots, Y_d))$, we have
$$\overline{\exp\big(\sum_{m \ge j} c_m Y_m\big)\gamma_j\gamma_{j + 1} \gamma_{j + 2} \dots \gamma_d} = \overline{\exp\big(\sum_{\ell \ge j} c_\ell Y_\ell\big)}.$$
Hence, we have ensured that we may take $c_\ell$ to have absolute value at most $q$, and thus height at most $q^2 = Q^{O_k(d^{O(1)})}$. \\\\
We now let $X_j' = \sum_j c_j Y_j$. We claim inductively that this satisfies the property we want. This is obviously true in the base case for $j = d$. Assuming this is satisfied for $i + 1$, we see that if $\gamma \in \exp(\text{Span}(Y_{i}, \dots, Y_d)) \cap \Gamma$, then the image of $\gamma$ in 
$$\exp(\text{Span}(Y_i, \dots, Y_d)) \cap \Gamma / (\exp(\text{Span}(Y_{i + 1}, \dots, Y_d)) \cap \Gamma)$$
can be written as
$$\overline{\exp(t_i X_i')}$$
where $t_i$ is an integer. This is to say, $\exp(-t_iX_i')\gamma \in \exp(\text{Span}(Y_{i + 1}, \dots, Y_d))$ and it must also lie in $\Gamma$ since $t_i$ is an integer. We may thus write
$$ \exp(-t_iX_i')\gamma = \prod_{j = i + 1}^d \exp(t_i X_i')$$
so
$$\gamma = \prod_{j = i}^d \exp(t_i X_i')$$
as desired.
%Since $\psi \circ \psi_{\exp}^{-1}$ has polynomial coordinates with coefficients of height at most $O_k(Q^{O_k(d^{O(1)})})$, it follows also that
\end{proof}
\begin{remark}
The proof written down in \cite[Proposition A.9]{GT12} is incorrect but easily fixable (as we have shown). The issue is that they write their generators $X_i'$ to be a multiple of $Y_i$. This does not have to be true. For instance, consider the lattice $\{(x, y, z) \in \mathbb{Z}^3: 2x_1 + 3x_2 + 5x_3 = 0\}$. One can take $Y_1 = (-5, 0, 2)$ and $Y_2 = (-3, 2, 0)$. The proof written in \cite[Proposition A.9]{GT12} would give $X_i' = Y_i$, but the point $(1, 1, -1)$ lies inside the lattice, but is not in the integral span of $Y_1$ and $Y_2$.
\end{remark}
We have the following corollary.
\begin{lemma}\label{l:Subgroupfiltration}
Suppose $\mathcal{X} = \{X_1, \dots, X_d\}$ is $Q$-rational Mal'cev basis for $G/\Gamma$ adapted to a filtration $(G_i)_{i = 1}^k$ and $G' \subseteq G$ is a $Q$-rational subgroup, then each group in the adapted filtration $(G_i')_{i = 1}^k$ with $G_i' := G_i \cap G'$ is also $Q^{O_k(d^{O(1)})}$-rational and there exists an adapted Mal'cev basis $\{X_1', \dots, X_{d'}'\}$ where $X_i'$ is $Q^{O_k(d^{O(1)})}$-rational combination of $X_i$.
\end{lemma}
\begin{proof}
By linear algebra, there is a basis $\mathcal{Y} = \{Y_1, \dots, Y_{d'}\}$ of $\mathfrak{g}'$ with an extension $\tilde{Y} = \{Y_1, \dots, Y_d\}$ such that each element of $Y_i$ can be written as a $Q$-rational combination of $X_i$. We now apply Lemma~\ref{l:ConstructingMalcev} to $\mathcal{Y}$.
\end{proof}

We next consider rational points. These next three lemmas are quantifications of \cite[Lemma A.11-A.13]{GT12}.

\begin{lemma}\label{l:multiplyrational}
Suppose $\mathcal{X}$ is a $Q$-rational Mal'cev basis for $G/\Gamma$ with the degree $k$ nesting property. \begin{itemize}
\item[(i)] If $\gamma \in G$ is $Q$-rational, then $\psi(\gamma) \in \frac{1}{Q'}\mathbb{Z}^d$ where $1 \le Q' \ll_k Q^{O_k(d^{O(1)})}$ which does not depend on $\gamma$.
\item[(ii)] If $\gamma \in G$ is such that $\psi(\gamma) \in \frac{1}{Q} \mathbb{Z}^d$, then $\gamma$ is $Q^{O_k(d^{O(1)})}$-rational.
\item[(iii)] If $\gamma, \gamma'$ are $Q$-rational, then $\gamma \gamma'$ and $\gamma^{-1}$ are $Q^{O_k(d^{O(1)})}$ rational.
\end{itemize}
\end{lemma}
\begin{proof}
For (i), since $\gamma$ is $Q$-rational, there exists some integer $r \le Q$ such that $\gamma^r \in \Gamma$. We now argue inductively as in Lemma~\ref{l:distancecomparison}. Using the fact that $\psi(\gamma^r) \in \mathbb{Z}^d$, we obtain inductively that
\begin{itemize}
\item[1.] The first through $d_i$th coordinates of $\psi(\gamma)$ have rationality bounded by $Q$.
\item[2.] The $d_i + 1$th through $d_{i + 1}$th coordinates of $\psi_{\mathcal{Y}}(\gamma)$ have rationality bounded by $Q^{O_k(d^{O(1)})}$. 
\end{itemize}
Thus, each coordinate has rationality bounded by $Q^{O_k(d^{O(1)})}$. \\\\
For (ii), note that a $Q$-rational Mal'cev basis is a $Q^{O_k(d^{O(1)})}$-weak basis by Lemma~\ref{l:weakbases}. As $\psi(\gamma) \in \frac{1}{Q}\mathbb{Z}^d$, $\psi_{\exp}(\gamma) \in \frac{1}{Q'}\mathbb{Z}^d$ with $Q' \le Q^{O_k(d^{O(1)})}$. The result follows as $\psi_{\exp}(\gamma^\ell) = \ell \cdot \psi_{\exp}(\gamma)$. \\\\
(iii) is then an immediate consequence of Lemma~\ref{l:multiplication}.
\end{proof}

\begin{lemma}\label{l:rationalpolynomialsequence}
Suppose that $\gamma\colon \mathbb{Z}^r \to G$ is a polynomial sequence of degree $d$. Suppose that there is a $Q$-rational Mal'cev basis $\mathcal{X}$ with the degree $k$ nesting property for $G/\Gamma$ and that $\gamma$ is $Q$-rational. Then $\gamma(n)\Gamma$ is periodic with period $\ll Q^{O_k(d^{O(1)})}$. 
\end{lemma}

\begin{proof}
Write
$$\psi(\gamma(\vec{n})) = \sum_{\vec{i}} \binom{\vec{n}}{\vec{i}} t_{\vec{i}}.$$
By repeatedly taking discrete differences and invoking Lemma~\ref{l:multiplyrational}(i), we see that each coefficient is $Q^{O_k(d^{O(1)})}$-rational. The result follows.
\end{proof}

\begin{lemma}
Suppose $\mathcal{X} = \{X_1, \dots, X_d\}$ is a $Q$-rational Mal'cev basis for $G/\Gamma$ with the degree $k$ nesting property. Suppose that $\gamma \in G$ is $Q$-rational and additionally that the coordinates $\psi(\gamma)$ are all bounded in magnitude by $Q$. Suppose that $G' \subseteq G$ is a $Q$-rational subgroup. Then the conjugate $\gamma G' \gamma^{-1}$ is $Q^{O_k(d^{O(1)})}$-rational.
\end{lemma}

\begin{proof}
Let $X_1', \dots, X_{d'}'$ be a $Q$-rational basis for the Lie algebra $\mathfrak{g}'$ of $G'$. Letting $\text{Ad}(h)X = \log(h \exp(X) h^{-1})$ be the adjoint automorphism, it follows that since it is an automorphism, $\tilde{X}_i:= \log(\gamma \exp(X_i') \gamma^{-1})$ is a basis for the Lie algebra of $\gamma G' \gamma^{-1}$. Thus, by Lemma~\ref{l:multiplication}, this basis is $Q^{O_k(d^{O(1)})}$-rational.
\end{proof}

\bibliographystyle{amsplain1.bst}
\bibliography{main.bib}

\providecommand{\bysame}{\leavevmode\hbox to3em{\hrulefill}\thinspace}
\providecommand{\MR}{\relax\ifhmode\unskip\space\fi MR }
% \MRhref is called by the amsart/book/proc definition of \MR.
\providecommand{\MRhref}[2]{%
  \href{http://www.ams.org/mathscinet-getitem?mr=#1}{#2}
}
\providecommand{\href}[2]{#2}
\begin{thebibliography}{10}

\bibitem{Alt22b}
D.~Altman, \emph{A non-flag arithmetic regularity lemma and counting lemma},
  arXiv:2209.14083.

\bibitem{Alt22}
D.~Altman, \emph{On a conjecture of {G}owers and {W}olf}, Discrete Anal.
  (2022), Paper No. 10, 13.

\bibitem{CS14}
Pablo Candela and Olof Sisask, \emph{Convergence results for systems of linear
  forms on cyclic groups and periodic nilsequences}, SIAM J. Discrete Math.
  \textbf{28} (2014), 786--810.

\bibitem{Gow98}
W.~T. Gowers, \emph{A new proof of {S}zemer\'edi's theorem for arithmetic
  progressions of length four}, Geom. Funct. Anal. \textbf{8} (1998), 529--551.

\bibitem{Gow01a}
W.~T. Gowers, \emph{A new proof of {S}zemer\'{e}di's theorem}, Geom. Funct.
  Anal. \textbf{11} (2001), 465--588.

\bibitem{GW11b}
W.~T. Gowers and J.~Wolf, \emph{Linear forms and higher-degree uniformity for
  functions on {$\mathbb F^n_p$}}, Geom. Funct. Anal. \textbf{21} (2011),
  36--69.

\bibitem{GW11}
W.~T. Gowers and J.~Wolf, \emph{Linear forms and quadratic uniformity for
  functions on {$\mathbb Z_N$}}, J. Anal. Math. \textbf{115} (2011), 121--186.

\bibitem{GT10b}
B.~Green and T.~Tao, \emph{An arithmetic regularity lemma, an associated
  counting lemma, and applications}, An irregular mind, Bolyai Soc. Math.
  Stud., vol.~21, J\'{a}nos Bolyai Math. Soc., Budapest, 2010, pp.~261--334.

\bibitem{GT10}
B.~Green and T.~Tao, \emph{Linear equations in primes}, Ann. of Math. (2)
  \textbf{171} (2010), 1753--1850.

\bibitem{GT12}
B.~Green and T.~Tao, \emph{The quantitative behaviour of polynomial orbits on
  nilmanifolds}, Ann. of Math. (2) \textbf{175} (2012), 465--540.

\bibitem{GTZ11}
B.~Green, T.~Tao, and T.~Ziegler, \emph{An inverse theorem for the {G}owers
  {$U^4$}-norm}, Glasg. Math. J. \textbf{53} (2011), 1--50.

\bibitem{GT12c}
Ben Green and Terence Tao, \emph{The {M}\"{o}bius function is strongly
  orthogonal to nilsequences}, Ann. of Math. (2) \textbf{175} (2012), 541--566.

\bibitem{GT14}
Ben Green and Terence Tao, \emph{On the quantitative distribution of polynomial
  nilsequences---erratum [mr2877065]}, Ann. of Math. (2) \textbf{179} (2014),
  1175--1183.

\bibitem{GT17}
Ben Green and Terence Tao, \emph{New bounds for {S}zemer\'{e}di's theorem,
  {III}: a polylogarithmic bound for {$r_4(N)$}}, Mathematika \textbf{63}
  (2017), 944--1040.

\bibitem{GTZ12}
Ben Green, Terence Tao, and Tamar Ziegler, \emph{An inverse theorem for the
  {G}owers {$U^{s+1}[N]$}-norm}, Ann. of Math. (2) \textbf{176} (2012),
  1231--1372.

\bibitem{HK05}
Bernard Host and Bryna Kra, \emph{Nonconventional ergodic averages and
  nilmanifolds}, Ann. of Math. (2) \textbf{161} (2005), 397--488.

\bibitem{HK18}
Bernard Host and Bryna Kra, \emph{Nilpotent structures in ergodic theory},
  Mathematical Surveys and Monographs, vol. 236, American Mathematical Society,
  Providence, RI, 2018.

\bibitem{Kuc21}
Borys Kuca, \emph{True complexity of polynomial progressions in finite fields},
  Proc. Edinb. Math. Soc. (2) \textbf{64} (2021), 448--500.

\bibitem{Kuc23}
Borys Kuca, \emph{On several notions of complexity of polynomial progressions},
  Ergodic Theory Dynam. Systems \textbf{43} (2023), 1269--1323.

\bibitem{Lei05}
Alexander Leibman, \emph{Pointwise convergence of ergodic averages for
  polynomial sequences of translations on a nilmanifold}, Ergodic Theory and
  Dynamical Systems \textbf{25} (2005), 201--213.

\bibitem{Len22}
J.~Leng, \emph{A {Q}uantitative {B}ound {F}or {S}zemer{\'e}di's {T}heorem for a
  {C}omplexity {O}ne {P}olynomial {P}rogression over $\mathbb{Z}/n\mathbb{Z}$},
  arXiv:2205.05540.

\bibitem{LenNew}
James Leng, \emph{Efficient {e}quidistribution of periodic nilsequences and
  applications}, arXiv:2306.13820.

\bibitem{Len22b}
James Leng, \emph{Improved {Q}uadratic {G}owers {U}niformity for the
  {M}\"{o}bius {F}unction}, arXiv:2212.09635.

\bibitem{LSS24c}
James Leng, Ashwin Sah, and Mehtaab Sawhney, \emph{Improved {B}ounds for
  {S}zemer\'{e}di's {T}heorem}, manuscript.

\bibitem{LSS24b}
James Leng, Ashwin Sah, and Mehtaab Sawhney, \emph{Quasipolynomial bounds for
  the inverse theorem for the {G}owers {$U^{s+1}[N]$}-norm}, manuscript.

\bibitem{MW22}
Lilian Matthiesen and Mengdi Wang, \emph{Smooth numbers are orthogonal to
  nilsequences}, arXiv:2211.16892.

\bibitem{PW23}
Mayank Pandey and Katharine Woo, \emph{Small scale distribution of linear
  patterns of primes}, arXiv:2304.14267.

\bibitem{PSS23}
Sarah Peluse, Ashwin Sah, and Mehtaab Sawhney, \emph{Effective bounds for
  {R}oth's theorem with shifted square common difference}, arXiv:2309.08359.

\bibitem{Str87}
Robert~S. Strichartz, \emph{The {C}ampbell-{B}aker-{H}ausdorff-{D}ynkin formula
  and solutions of differential equations}, J. Funct. Anal. \textbf{72} (1987),
  320--345.

\bibitem{Tao12}
T.~Tao, \emph{Higher order {F}ourier analysis}, Graduate Studies in
  Mathematics, vol. 142, American Mathematical Society, Providence, RI, 2012.

\bibitem{TT21}
T.~Tao and J.~Ter{\"a}v{\"a}inen, \emph{Quatitative bounds for {G}owers
  uniformity of the {M}\"{o}bius and von {M}angoldt functions},
  arXiv:2107.02158.

\bibitem{TaoBlog1}
Terence Tao, \emph{Equidistribution for multidimensional polynomial phases},
  blog post.
  \url{https://terrytao.wordpress.com/2015/08/06/equidistribution-for-multidimensional-polynomial-phases/}.

\bibitem{TV10}
Terence Tao and Van~H. Vu, \emph{Additive combinatorics}, Cambridge Studies in
  Advanced Mathematics, vol. 105, Cambridge University Press, Cambridge, 2010,
  Paperback edition [of MR2289012].

\bibitem{Vin37}
Ivan~Matveevich Vinogradov, \emph{A new method in analytic number theory},
  Trudy Matematicheskogo Instituta imeni VA Steklova \textbf{10} (1937),
  5--122.

\bibitem{Zie07}
Tamar Ziegler, \emph{Universal characteristic factors and {F}urstenberg
  averages}, J. Amer. Math. Soc. \textbf{20} (2007), 53--97.

\end{thebibliography}

\end{document}